\documentclass[twoside,11pt]{article}

 \usepackage[abbrvbib, preprint]{jmlr2e}

\usepackage{pifont}
\usepackage{multirow}
\usepackage{slashbox}
\usepackage{comment}
\usepackage[utf8]{inputenc} %
\usepackage{url}            %
\usepackage{booktabs}       %
\usepackage{amsfonts}       %
\usepackage{nicefrac}       %
\usepackage{microtype}      %
\usepackage{xcolor}         %

\usepackage{algorithm}
\usepackage{algorithmic}
\usepackage{amsmath} 
\usepackage[nameinlink,capitalize,noabbrev]{cleveref}
\usepackage{bm}

\newcommand{\bmx}{\bm{x}}
\newcommand{\bmy}{\bm{y}}
\newcommand{\bmxi}{\bm{\xi}}
\newcommand{\bmxic}{\bm{\xi_t}}
\newcommand{\bmxin}{\bm{\xi_{t+1}}}

\newcommand{\bmc}{\bm{x}_t}
\newcommand{\bmn}{\bm{x}_{t+1}}

\newcommand{\bmo}{\bm{x}^*}

\newcommand{\grad}{\nabla}

\newcommand{\tmix}{\tau_{\text{mix}}}
\DeclareMathOperator*{\argmin}{argmin}
\DeclareMathOperator*{\argmax}{argmax}

\newtheorem{assumption}{Assumption}

\usepackage{lastpage}
\jmlrheading{23}{2023}{1-\pageref{LastPage}}{1/21; Revised 5/22}{9/22}{21-0000}{Yeongjong Kim and Dabeen Lee}

\ShortHeadings{Stochastic-Constrained Stochastic Optimization with Markovian Data}{Yeongjong Kim and Dabeen Lee}
\firstpageno{1}

\begin{document}

\title{Stochastic-Constrained Stochastic Optimization with Markovian Data}

\author{\name Yeongjong Kim \email kimyj@kaist.ac.kr \\
       \addr Department of Mathematical Sciences\\
       KAIST\\
       Daejeon 34141, South Korea
       \AND
       \name Dabeen Lee \email dabeenl@kaist.ac.kr \\
       \addr Department of Industrial and Systems Engineering\\
       KAIST\\
       Daejeon 34141, South Korea}

\editor{My editor}

\maketitle

\begin{abstract}%
This paper considers stochastic-constrained stochastic optimization where the stochastic constraint is to satisfy that the expectation of a random function is below a certain threshold. In particular, we study the setting where data samples are drawn from a Markov chain and thus are not independent and identically distributed. We generalize the drift-plus-penalty framework, a primal-dual stochastic gradient method developed for the i.i.d. case, to the Markov chain sampling setting. We propose two variants of drift-plus-penalty; one is for the case when the mixing time of the underlying Markov chain is known while the other is for the case of unknown mixing time. In fact, our algorithms apply to a more general setting of constrained online convex optimization where the sequence of constraint functions follows a Markov chain. Both algorithms are adaptive in that the first works without knowledge of the time horizon while the second uses AdaGrad-style algorithm parameters, which is of independent interest. We demonstrate the effectiveness of our proposed methods through numerical experiments on classification with fairness constraints.
\end{abstract}


\section{Introduction}

In this paper, we consider the stochastic-constrained stochastic optimization (SCSO) problem 
\begin{equation}\label{scso}
    \min_{\bmx\in\mathcal{X}} \quad \mathbb{E}_{\bmxi\sim \mu}[f(\bmx, \bmxi)]\notag\quad  \text{s.t.}\quad \mathbb{E}_{\bmxi\sim \mu}[g(\bmx, \bmxi)]\leq 0\tag{SCSO}
\end{equation}
where the expectation is taken with respect to the random parameter $\bmxi$ over a stationary distribution $\mu$, $f$ and $g$ are convex loss and constraint functions, and $\mathcal{X}$ is a compact domain. This formulation of \ref{scso} has applications in stochastic programming with a risk constraint~\citep{cvar}, chance-constrained programming~\citep{doi:10.1137/050622328}, portfolio optimization~\citep{doi:10.1137/S1052623402420528}, sparse matrix completion~\citep{conservative}, semi-supervised learning~\citep{10.5555/1841234}, classification with fairness constraints~\citep{fairness-constraints-jmlr, Classification-w-fairness-constraints, empirical-risk-min-under-fairness}, Neyman-Pearson classification~\citep{1522642,JMLR:v12:rigollet11a}, ranking with fairness constraints~\citep{ranking-with-fairness-constraints}, recommendation systems with fairness constraints \citep{fairness-in-Collb-Filtering}, scheduling in distributed data centers~\citep{OCO-stochastic}, safe reinforcement learning~\citep{JMLR:v16:garcia15a}, and stochastic simple bilevel optimization~\citep{jalilzadeh2023stochastic,cao2023projectionfree}.

Stochastic approximation (SA) algorithms are prevalent solution methods for \ref{scso}. Basically, we run gradient-based algorithms with an oracle providing i.i.d. samples of $f(\bmx,\bmxi),g(\bmx,\bmxi),\grad f(\bmx,\bmxi),\grad g(\bmx, \bmxi)$ for a given solution $\bmx$. \cite{cooperative} proposed the cooperative stochastic approximation scheme for \ref{scso}, which is a stochastic extension of Polyak's subgradient method developed for constrained optimization~\citep{polyak}. \cite{Xiantao} developed the penalized stochastic gradient method that takes the square of the constrained function as a penalty term. \cite{JMLR:v21:19-1022} developed a level set-based algorithm for \ref{scso}. \cite{conservative} studied an augmented Lagrangian-based stochastic primal-dual algorithm, which was inspired by the primal-dual framework for online convex optimization with long-term constraints due to~\cite{long-term-1}. Furthermore, motivated by recent success in adaptive gradient algorithms, \cite{aprild}  considered an adaptive primal-dual stochastic gradient method for \ref{scso}. \cite{doi:10.1287/moor.2022.1257} proposed a stochastic variant of the proximal method of multipliers. \cite{doi:10.1287/ijoc.2022.1228} studied another stochastic proximal method of multipliers based on linearization.

SA algorithms for other related problem settings are as follows. \cite{OCO-stochastic} studied online convex optimization with stochastic constraints where the constraint functions are stochastic i.i.d. while the objective functions are arbitrary, for which they proposed the drift-plus-penalty (DPP) algorithm. DPP applies to \ref{scso} given that i.i.d. samples of $g(\bmx,\bmxi)$, $\grad g(\bmx,\bmxi)$, and $\grad f(\bmx,\bmxi)$ are available. \cite{dpp-md} developed an extension of DPP, and \cite{lee2023projectionfree} provided a projection-free algorithm for the setting. Moreover, \ref{scso} can be formulated as a stochastic saddle-point problem by taking the Lagrangian if certain constraint qualifications hold. \cite{mirrorprox1,mirrorprox2} developed stochastic mirror-prox algorithms for general stochastic saddle point problems. \cite{doi:10.1287/moor.2021.1175} devised an accelerated stochastic framework for convex-concave saddle-point problems, and \cite{pmlr-v206-yazdandoost-hamedani23a} devised a randomized adaptive primal-dual method.

The aforementioned solution methods for \ref{scso} require i.i.d. data samples from the stationary distribution $\mu$ when running SA algorithms. However, there are several application scenarios in which sampling from the stationary distribution $\mu$ independently and identically is difficult. For example, federated learning serves as an alternative to traditional machine learning systems that require data centralization, with the purpose of improving data privacy~\citep{https://doi.org/10.48550/arxiv.1806.00582}. The basic framework is that data is stored on individual local devices while the training is governed by a central server. Another related setting is distributed optimization over sensor networks~\citep{1307319,4276978,4434888} and multi-agent systems~\citep{JOHANSSON20081175}. For these applications, an enormous amount of data is distributed over distinct nodes of a network, for which data exchange and massage passing are between immediate neighboring nodes.

One resolution approach for such application scenarios is Markov chain stochastic gradient descent (SGD)~\citep{johansson2010,doi:10.1137/080726380}. Basically, Markov chain SGD takes a random walk over a network of local data centers and updates solutions based on data acquired from the data centers visited. Here, as the data is generated by a Markov random walk, there is inherent dependence and bias between data samples. More generally, the framework can be viewed as a variant of the Markov chain Monte Carlo method~\citep{mcmc}. That said, Markov chain SGD can be applied to other applications domains where data is collected from a Markov process, such as decentralized learning~\citep{NEURIPS2021_a87d27f7}, robust estimation~\citep{POLJAK198053,pmlr-v97-sarkar19a}, and reinforcement learning~\citep{NEURIPS2020_c22abfa3,NEURIPS2021_fd2c5e46}.

Following the Markov incremental subgradient methods due to~\cite{johansson2010,doi:10.1137/080726380} designed for distributed optimization problems, Markov chain SGD with data sampled from a general Markov process have been studied. \cite{emd} developed Markov chain SGD with data sampled from an ergodic process. Later, \cite{NEURIPS2018_1371bcce} studied Markov chain SGD for convex and nonconvex problems when the underlying Markov chain is nonreversible. \cite{doan2020convergence} proposed an accelerated version of Markov chain SGD for both convex and nonconvex settings. \cite{Levy22} considered the setting where the mixing time of the underlying Markov chain is unknown, for which they developed Markov chain SGD based on the multi-level Monte Carlo estimation scheme~\citep{giles_2015,blanchet-WSC}. \cite{Markov-constrained-sto-opt} studied nonconvex problems where the transition of the underlying Markov chain is state-dependent, motivated by strategic classification and reinforcement learning. Recent works~\citep{9769873,pmlr-v202-even23a} established some performance guarantees of Markov chain SGD under minimal assumptions.

Applications of \ref{scso} naturally motivate and necessitate algorithmic frameworks that can handle data sets with inherent dependence and bias. Portfolio optimization in finance takes time series data. Machine learning with fairness constraints processes heterogeneous data sets from disjoint sources. Scheduling in distributed data centers will benefit from distributed optimization technologies. Safe and constrained reinforcement learning can be formulated as~\ref{scso} for which the training data is obtained from trajectories of the underlying Markov decision process. Despite this immediate need, no prior work exists for solving~\ref{scso} with non-i.i.d. data samples. Motivated by this, the objective of this paper is to develop stochastic approximation algorithms that run with data sampled from a Markov chain.

\subsection{Our Contributions}

This paper initiates the study of stochastic approximation algorithms for stochastic-constrained stochastic optimization~\eqref{scso} using non-i.i.d. data samples. Inspired by recent advances in Markov chain SGD, we develop primal-dual stochastic gradient methods using data sampled from a Markov chain, which can be viewed as primal-dual variants of Markov chain SGD. Specifically, we extend the drift-plus-penalty algorithm by~\cite{OCO-stochastic} that was originally developed for the i.i.d. setting. We adopt the approach of ergodic mirror descent by~\cite{emd} for the case of known mixing time and the framework of~\cite{Levy22} for the setting where the mixing time is unknown. 

Our key technical contribution is to develop two variants of the drift-plus-penalty algorithm that can take a sequence of dependent constraint functions. These two algorithms solve constrained online convex optimization where the objective functions can be arbitrary and the constraint functions are generated from a Markov chain. One of them is for the case of known mixing time, while the other is for the case when the mixing time is unknown. We provide regret and constraint violation bounds for the algorithms, which delineate how their performance depends on the mixing time. Based on the regret and constraint violation guarantees, we analyze the optimality gap and feasibility gap for~\ref{scso}. The connection between the constrained online convex optimization formulation and~\ref{scso} for our Markovian setting is not as immediate as the i.i.d. setting because the expectation in~\eqref{scso} is taken with respect to the stationary distribution of the underlying Markov chain.

What follows is a more detailed description of our contributions. Our results are also summarized in \Cref{results}.
\begin{table*}     
\begin{center}
\begin{tabular}{c|c|c|c}
\toprule
& \multirow{2}{*}{\textbf{\Cref{alg0}}} & \textbf{\Cref{alg0} under} & \multirow{2}{*}{\textbf{\Cref{alg1}}}\\
& & \textbf{ Slater's condition} & \\
\hline
\textbf{Oblivious to $\tmix$} &\ding{55} &\ding{55}&\ding{52}\\
\hline
\textbf{Regret} &$\tilde O\left(\tmix^{1-\beta} T^{1-\beta}\right)$& $\tilde O\left(\sqrt{\tmix T}\right)$&adaptive\\
\hline
\textbf{Constraint violation} &$O\left(\tmix^{\beta/2} T^{(\beta+1)/2}\right)$&$\tilde O\left(\sqrt{\tmix T}\right)$&adaptive\\
\hline
\textbf{Optimality gap} &$\tilde O\left(\frac{\tmix^{1-\beta}}{T^{\beta}} + \frac{\tmix^{\beta/2}}{T^{(1-\beta)/2}}\right)$&$\tilde O\left({\frac{\sqrt{\tmix}}{ \sqrt{T}}}\right)$&$\tilde O\left(\frac{\tmix^{1-\beta}}{T^{\beta}}\right)$\\
\hline
\textbf{Feasibility gap} &$\tilde O\left(\frac{\tmix^{\beta/2}}{T^{(1-\beta)/2}}\right)$& $\tilde O\left({\frac{\sqrt{\tmix}}{\sqrt {T}}}\right)$ & $\tilde O\left(\frac{\tmix^{(2\beta+1)/4}}{T^{(1-\beta)/2}}\right)$ \\
\bottomrule
\end{tabular}
\caption{Bounds on regret, constraint violation, optimality gap, and feasibility gap under Algorithms~\ref{alg0} and~\ref{alg1} ($\beta\in(0,1/2]$ is a predetermined algorithm parameter that controls the balance between regret and constraint violation)}\label{results}
\end{center}
\end{table*}
\begin{itemize}
    \item In \Cref{sec:tmix}, we consider the case when the mixing time of the underlying Markov chain is known, for which we develop \Cref{alg0}, a variant of the drift-plus-penalty algorithm. We first prove that for online convex optimization with Markovian constraints, the regret of~\Cref{alg0} is $\tilde O(\tmix^{1-\beta} T^{1-\beta})$ and the constraint violation is $O(\tmix^{\beta/2} T^{(\beta+1)/2})$ where $\tmix$ is the mixing time, $T$ is the length of horizon, $\beta$ can be chosen to be any number in $(0,1/2]$, and $\tilde O$ hides a $\text{poly}\log T$ factor. If we further assume that \eqref{scso} satisfies Slater's constraint qualification, then the regret and constraint violation of~\Cref{alg0} can be both bounded by $\tilde O(\sqrt{\tmix T})$. We remark that~\Cref{alg0} is adaptive in that its parameters are chosen without knowledge of $T$, unlike the vanilla drift-plus-penalty algorithm. These results generalize the work of~\cite{OCO-stochastic}.
    \item Based on the regret and constraint violation analysis for~\Cref{alg0}, we show that the averaging of the sequence of solutions generated by~\Cref{alg0} guarantees that the optimality gap is bounded by $\tilde O(\tmix^{1-\beta}T^{-\beta} + \tmix^{\beta/2}T^{-(1-\beta)/2})$ while the feasibility gap is bounded by $\tilde O(\tmix^{\beta/2}T^{-(1-\beta)/2})$. If we further assume that \eqref{scso} satisfies Slater's constraint qualification, then the optimality gap and feasibility gap of~\Cref{alg0} can be both bounded by $\tilde O(\sqrt{\tmix/T})$.
    \item In \Cref{sec:unknown}, we propose~\Cref{alg1}, another variant of the drift-plus-penalty algorithm, for the setting where the mixing time is unknown. The parameters of \Cref{alg1} are set in an adaptive fashion, as in the AdaGrad method~\citep{adagrad,levy2017}. Then we apply \Cref{alg1} to constrained online convex optimization where the objective and constraint functions are given by the Multi-level Monte Carlo estimation scheme as in~\cite{Levy22}. We provide adaptive regret and constraint violation bounds for \Cref{alg1}. We note that \Cref{alg1} is the first AdaGrad-style adaptive variant of the drift-plus-penalty algorithm. In fact, \Cref{alg1} applies to online convex optimization with adversarial constraints~\citep{yu-neely} and provides adaptive performance guarantees. We include this result in \Cref{sec:oco-adversarial}.
    \item Combining the estimation accuracy bounds on the Multi-level Montel Carlo method by~\cite{Levy22} and our adaptive regret and constraint violation bounds, we prove that the optimality gap under \Cref{alg1} is bounded by $\tilde O(\tmix^{1-\beta}T^{-\beta})$ while the feasibility gap is bounded by $\tilde O(\tmix^{(2\beta+1)/4}T^{-(1-\beta)/2} )$.
    \item In \Cref{sec:exp}, we provide numerical results from experiments on a classification problem with fairness constraints. Specifically, we take the logistic regression formulation proposed in~\cite{fairness-constraints-jmlr}. The numerical results on random problem instances demonstrate the efficacy of our proposed algorithmic frameworks for solving \ref{scso} with Markovian data.
\end{itemize}
The main component of our analysis is to provide bounds on the terms
$$\mathbb{E}\left[\sum_{t=1}^TQ_tg_t(\bmx)\right],\quad \mathbb{E}\left[\sum_{t=1}^T\frac{Q_t}{V_t}g_t(\bmx)\right],\quad \mathbb{E}\left[Q_t\right],\quad \mathbb{E}\left[\frac{Q_t}{V_t}\right]$$
under our Markovian data regime. All these terms involve the virtual queue size $Q_t$, and therefore, controlling the virtual queue size is crucial to guarantee a fast convergence rate. We include our proofs of the main theorems in \Cref{sec:edpp-proof,sec:mdpp-proof}. We include some of the known lemmas and tools for analyzing the drift-plus-penalty method due to \cite{OCO-stochastic} in \Cref{sec:adpp,sec:appendix-drift}. In fact, \Cref{sec:adpp} state these results for any adaptive version of the drift-plus-penalty algorithm, which uses time-varying parameters $V_t$ and $\alpha_t$.

\subsection{Related Work}

Although we have listed and explained important lines of previous work that motivate this paper, we supplement the list by mentioning a few more relevant results in stochastic approximation algorithms and Markov chain stochastic gradient methods.

\paragraph{Stochastic Approximation for Stochastic Optimization} Starting from the seminar paper by~\cite{10.1214/aoms/1177729586}, stochastic approximation algorithms, also known as stochastic gradient methods, have been a central topic of research in the domain of machine learning, operations research, and optimization \citep{doi:10.1080/17442508308833246,pflugbook,Ruszczynski1986,doi:10.1137/16M1080173}. There already exist numerous works on stochastic approximation algorithms for stochastic optimization. In particular, for \ref{scso} or stochastic optimization with expectation constraints, various stochastic approximation methods were proposed and studied by \cite{cooperative,Xiantao,JMLR:v21:19-1022,conservative,aprild,doi:10.1287/moor.2022.1257,doi:10.1287/ijoc.2022.1228}, as discussed earlier. 

We note that online convex optimization with stochastic constraints is a superclass of \ref{scso} under the i.i.d. data regime. \cite{long-term-1,long-term-2} developed augmented Lagrangian-based primal-dual algorithms for online convex optimization with deterministic long-term constraints. It turns out that these algorithms and their analysis can be adapted to the setting of stochastic constraints~\citep{conservative}. Later, \cite{OCO-stochastic,dpp-md} proposed the drift-plus-penalty algorithm that achieves better regret and constraint violation guarantees under Slater's condition. \cite{cumulative-1,cumulative-2,cumulative-3} also provided algorithms for online convex optimization with stochastic constraints.

\paragraph{Markov Chain Stochastic Gradient Descent}  

The asymptotic convergence of Markov chain SGD was studied in \cite{10.5555/560669,alma99752790902101,10.5555/95267} based on ordinary differential equation methods. \cite{4276978,4434888,johansson2010,doi:10.1137/080726380} developed incremental subgradient methods for the distributed optimization setting where there is a network of data centers and an algorithm performs a random walk over the network to obtain data samples. These methods are also referred to as token algorithms. More recently, \cite{9050511} devised what they called the walkman algorithm, which is a token algorithm based on an augmented Lagrangian method. \cite{pmlr-v162-sun22b,10.1109/JSAC.2023.3244250} considered adaptive variants of token algorithms. \cite{pmlr-v206-hendrikx23a} developed a more general framework for token algorithms that allows multiple tokens for improving communication efficiency and may adopt existing optimization tools such as variance reduction and acceleration. 

In addition to random walk-based token algorithms, there exist general Markov chain sampling frameworks for stochastic gradient methods. As mentioned earlier, \cite{emd,NEURIPS2018_1371bcce,doan2020convergence,Levy22,Markov-constrained-sto-opt} developed Markov chain SGD methods for convex and nonconvex stochastic optimization problems. Moreover, \cite{markov-bcd} developed a Markov chain sampling-based block coordinate descent method. \cite{10208129} proposed a decentralized variant of Markov chain SGD. \cite{NEURIPS2022_f61538f8} considered the stability of Markov chain SGD and deduced its generalization bounds. \cite{9769873} derived convergence guarantees on Markov chain SGD without a smoothness assusmption, and \cite{pmlr-v202-even23a} studied convergence of Markov chain SGD without the bounded gradient assumption.

\section{Preliminaries}\label{sec:prelim}

In this section, we provide problem formulations for stochastic-constrained stochastic and online convex optimization under the Markovian data sampling regime. In addition, \Cref{subsec:ergodic} gives the formal definition of the mixing time of a Markov chain. \Cref{subsec:notations} describes a list of assumptions considered throughout the paper. 

\subsection{Markov Chain and Mixing Time}\label{subsec:ergodic}
Given two probability distributions $\mathbb{P}, \mathbb{Q}$ over the probability space $(\mathcal{S},\mathcal{F})$, the total variation distance between them is defined as
$$\lVert \mathbb{P}-\mathbb{Q} \rVert_{TV}:=\sup_{A\in\mathcal{F}}\left|\mathbb{P}(A)-\mathbb{Q}(A)\right|.$$
Let $\{\bmxic\}_{t=1}^\infty$ be a time-homogeneous ergodic Markov chain with a finite state space $\mathcal{S}$. For a distribution $\nu$ over $(\mathcal{S},\mathcal{F})$, we denote by $\mathbb{P}^t(\nu,\cdot)$ the conditional probability distribution of $\bmxin$ given $\bm{\xi_1} \sim \nu$. %
Since $\{\bmxic\}_{t=1}^\infty$ is ergodic, it has a unique stationary distribution $\mu$, i.e., $\mathbb{P}^t(\mu,\cdot)=\mu(\cdot)$. 
In fact, the long-term distribution of the ergodic Markov chain converges to $\mu$ regardless of the initial distribution, which can be demonstrated as follows. It is known that $\mathbb{P}^t(\nu,\cdot)$ for any $t$ and  $\nu$ satisfies 
$$\lVert \mathbb{P}^t(\nu, \cdot)-\mu\rVert_{TV}\leq C\alpha^t$$
for some $\alpha\in (0,1)$ and $C>0$~\citep{levin2017markov}. Then we define quantities $d_{\text{mix}}$ and $\tmix(\epsilon)$ as follows.
\begin{align*}
    d_{\text{mix}}:=\sup_{\nu} \lVert P^t(\nu,\cdot)-\mu\rVert_{TV},\quad 
    \tmix(\epsilon):=\inf \{t\in\mathbb{N} : d_{\text{mix}}(t)\leq\epsilon\}.
\end{align*}
Moreover, following the convention, we define $\tmix$ as
$$\tmix:=\tmix(1/4),$$
and we refer to $\tmix$ as the \emph{mixing time} of the underlying Markov chain. It is known that
$d_{\text{mix}}(l\tau_{\text{mix}})\leq 2^{-l}$
for every $l\in\mathbb{N}$ \citep[Chapter 4]{levin2017markov}, which implies that
$$\tau_{\text{mix}}(\epsilon)\leq\lceil \log_2 \epsilon^{-1}\rceil\tau_{\text{mix}}.$$
In particular, throughout the paper, we will use quantity $\tau$ defined as
$$\tau := \tmix(1/T) = O (\tmix \log T)$$
where $T$ is the length of the time horizon.

\subsection{Stochastic-Constrained Online Convex Optimization with Markovian Data}\label{subsec:oco}

Let $\mathcal{X}\subset\mathbb{R}^d$ be a compact convex domain. Let $\{f_t:\mathcal{X}\to\mathbb{R}\}_{t=1}^\infty$ be a sequence of arbitrary convex functions. Another sequence of convex functions $\{g_t:\mathcal{X}\to\mathbb{R}\}_{t=1}^\infty$ is assumed to follow an ergodic Markov chain. More precisely, there exists a time-homogeneous ergodic Markov chain $\{\bmxic\}_{t=1}^\infty$ such that $g_t(\bmx)=g(\bmx,\bmxic)$. Let $$\bar{g}(\bmx)=\mathbb{E}_{\bmxi\sim \mu}[g(\bmx,\bmxi)]$$ for $\bmx\in\mathcal{X}$ where the expectation is taken with respect to the stationary distribution $\mu$ of the Markov chain. Then the problem is to solve and compute \begin{align*}
    \bmo\quad\in \quad\argmin_{\bmx\in\mathcal{X}}\quad  \sum_{t=1}^T f_t(\bmx)\quad \text{s.t.}\quad \bar{g}(\bmx)\leq 0.
\end{align*}
However, the information about the functions $\{f_t\}_{t=1}^T$ and $\{g_t\}_{t=1}^T$ is revealed \emph{online}. Basically, at each step $t$, we choose our decision $\bmc\in\mathcal{X}$ before observing $f_t, g_t$, after which we receive feedback about them. The stochastic setting studied in~\cite{OCO-stochastic} is that  constraint functions $g_1,\ldots, g_T$ are i.i.d., which means that $\bm{\xi_1},\ldots, \bm{\xi_T}$ are i.i.d., while we consider the case where constraint functions $g_1,\ldots, g_T$ follow an ergodic Markov chain and thus are dependent. That said, we may refer to the problem setting as \emph{online convex optimization with ergodic constraints}.

To measure the performance of a learning algorithm for the online optimization problem, we adopt the regret and cumulative constraint violation definition due to~\cite{OCO-stochastic}. The regret and cumulative constraint violation of an algorithm that generates solutions $\bm{x_1},\ldots, \bm{x_T}$ over $T$ time steps are given by
\begin{align*}
    \operatorname{Regret}(T) = \sum_{t=1}^T f_t(\bmc)-\sum_{t=1}^T f_t(\bmx^*),\quad \operatorname{Violation}(T) =\sum_{t=1}^T g_t(\bmc).
\end{align*}
We want these quantities to have a sublinear growth in $T$.
Note that the benchmark solution $\bmx^*$ satisfies the expectation constraint $\mathbb{E}_{\bmxi\sim\mu}[g(\bmx^*,\bmxi)]\leq 0$.

\subsection{Stochastic-Constrained Stochastic Optimization with Markovian Data}\label{subsec:scso}

Next, we consider the setting where the objective functions $\{f_t\}_{t=1}^\infty$ as well as the constraint functions $\{g_t\}_{t=1}^\infty$ are given by an ergodic Markov chain. Without loss of generality, we may assume that $f_t(\bmx)=f(\bmx,\bmxic)$ and $g_t(\bmx)=g(\bmx,\bmxic)$ for some time-homogeneous ergodic Markov chain $\{\bmxic\}_{t=1}^\infty$ with a stationary distribution $\mu$. As before, let $$\bar{f}(\bmx)=\mathbb{E}_{\bmxi\sim \mu}[f(\bmx,\bmxi)]$$and $\bar{g}(\bmx)=\mathbb{E}_{\bmxi\sim \mu}[g(\bmx,\bmxi)]$ for $\bmx\in\mathcal{X}$. Then the problem is to solve and compute
\begin{align*}
\bm{x^\#}\quad\in    \quad\argmin_{\bmx\in\mathcal{X}}\quad \bar{f}(\bmx)\quad \text{s.t.}\quad \bar{g}(\bmx)\leq 0.
\end{align*}
Here, we do not have direct access to $(\bar{f}, \bar{g})$, but we receive samples  $(f_t, g_t)$ which converge to $(\bar{f}, \bar{g})$ in expectation. 

For the performance measure of a learning algorithm for this problem, we consider the following standard notions of optimality gap and constraint violation. For a solution $\bmx\in\mathcal{X}$
\begin{align*}
    \operatorname{Gap}(\bmx)=\bar{f}(\bmx)-\bar{f}(\bm{x^\#}),\quad
    \operatorname{Infeasibility}(\bmx)=\bar{g}(\bm{x}).
\end{align*}
We want these quantities to approach $0$ as $T$ grows. Hereinafter, we refer to $\operatorname{Gap}(\bmx)$ and $\operatorname{Infeasibility}(\bmx)$ as the optimality gap and the feasibility gap, respectively.

In the special case where $\{\bmxic\}_{t=1}^\infty$ are i.i.d. samples drawn from $\mu$, solving stochastic-constrained online convex optimization would provide a solution to stochastic-constrained stochastic optimization. One can argue that
if $\bm{\xi_1},\ldots, \bm{\xi_T}$ are i.i.d., then 
\begin{align*}\operatorname{Gap}(\bm{\bar x_T})\leq \frac{1}{T}\mathbb{E}\left[\operatorname{Regret}(T)\right],\quad
\operatorname{Infeasibility}(\bm{\bar x_T})\leq \frac{1}{T}\mathbb{E}\left[\operatorname{Violation}(T)\right]
\end{align*}
where $\bm{\bar x_T}$ denotes the simple average of $\bm{x_1},\ldots, \bm{x_T}$: $$\bm{\bar x_T}=\frac{1}{T}\sum_{t=1}^T\bmc.$$
However, in contrast to the i.i.d. case, the above inequalities do not hold for the case of an arbitrary ergodic Markov chain. This is because the distribution of $\bmxin$ conditioned on $\bmxic$ is not equal to the stationary distribution $\mu$. 

On the other hand, we will use the fact that the long-term distribution of an ergodic Markov chain converges to its stationary distribution. Based on this observation, we first bound the expected regret and the expected constraint violation of the online problem, and then we use this to bound the optimality gap and the feasibility gap of the stochastic optimization problem.

\subsection{Notations and Assumptions}\label{subsec:notations}
We work over a norm $\lVert\cdot\rVert$ and its dual norm $\lVert\cdot\rVert_*$. We choose a convex mirror map $\Phi:\mathcal{C}\to \mathbb{R}$, where $\mathcal{C}\subseteq\mathbb{R}^d$ is a convex domain containing $\mathcal{X}$. We use the corresponding \textit{Bregman divergence} defined as
$$D(\bmx,\bmy)=\Phi(\bmx)-\Phi(\bmy)-\grad \Phi(\bmy)^\top(\bmx-\bmy).$$

\begin{assumption}\label{ass:bounded}
There is a constant $R>0$ such that $D(\bmx,\bmy)\leq R^2$ for any $\bmx,\bmy\in\mathcal{X}$, and $\Phi$ is $2$-strongly convex with respect to norm $\lVert\cdot\rVert$ i.e. $ \lVert\bmx-\bmy\rVert^2\leq D(\bmx,\bmy)$ for any $\bmx,\bmy\in\mathcal{X}$. Moreover, $f(\cdot, \bmxi)$ and $g(\cdot, \bmxi)$ are differentiable for each $\bmxi\in \mathcal{S}$.
\end{assumption}
We use notations $F_t, G_t, H_t$ for $t\in[T]$ given by
\begin{equation*}
    F_{t}:=\lVert\grad f_t(\bmc)\rVert_*, \ G_{t}:=\lVert\grad g_t(\bmc)\rVert_*, \ H_{t}:=|g_t(\bmc)|
\end{equation*}
which are parameters used in our adaptive algorithm. Due to the stochasticity of $f_t, g_t$, these are random variables.
We assume these quantities are bounded.
\begin{assumption}\label{ass:lip} There exist constants $F, G, H >0$ such that  
    \begin{align*}
        \lVert \grad_{\bmx} f(\bmx, \bmxi)\rVert_*\leq F, \ \lVert \grad_{\bmx} g(\bmx, \bmxi)\rVert_*\leq G, \ |g(\bmx, \bmxi)|\leq H
    \end{align*}
    for any $\bmx \in\mathcal{X}$ and $\bmxi\in\mathcal{S}$.
\end{assumption}
Assumptions~\ref{ass:bounded} and~\ref{ass:lip} are common in online convex optimization, stochastic optimization, and the Markov chain SGD literature. The following assumption is referred to as \emph{Slater's condition} or \emph{Slater's constraint qualification} for constrained optimization. 
\begin{assumption}[Slater's condition]\label{ass:slater}
    There exists  ${\bm{\hat x}}\in\mathcal{X}$ such that $\bar{g}(\bm{\hat x})\leq -\epsilon$ for some $\epsilon>0$.
\end{assumption}
For online convex optimization with stochastic i.i.d. constraint functions, Slater's condition leads to improvement in regret and constraint violation~\citep{OCO-stochastic}. In \Cref{sec:tmix}, we will show that even for online convex optimization with ergodic constraint functions, assuming that Slater's condition holds results in an improvement in the regret and constraint violation of \Cref{alg0}.

\section{Known Mixing Time}\label{sec:tmix}

We first focus on the setting where we have access to the mixing time $\tmix$ of the underlying Markov chain $\{\bmxic\}_{t=1}^\infty$. \cite{emd} studied this case for stochastic convex minimization without a stochastic functional constraint, for which they modified the step size of the stochastic gradient descent method based on the mixing time parameter $\tmix$. Inspired by this approach, we take and modify the drift-plus-penalty (DPP) algorithm due to~\cite{yu-neely,OCO-stochastic} developed for stochastic-constrained online convex optimization. Based on DPP, we develop our algorithm by setting the algorithm parameters properly to adapt to the mixing time $\tmix$.

\subsection{Ergodic Drift-Plus-Penalty}\label{sec:tmix-1}

The DPP algorithm has two parameters, $V$ and $\alpha$, where $V$ is the penalty parameter and $\alpha$ determines the step size. \cite{OCO-stochastic} set $V=\sqrt{T}$ and $\alpha= T$. In contrast to the vanilla DPP algorithm, our algorithm uses parameters
$$V_t=(\tmix t)^\beta,\quad \alpha_t = \tmix t$$
for iterations $t=1,\ldots, T$, where $T$ is the length of the horizon and $\beta$ is another algorithm parameter that controls the balance between the regret and the constraint violation. Our algorithm, which we call \emph{ergodic drift-plus-penalty (EDPP)}, is described in \Cref{alg0}.

\begin{algorithm}[tb]
		\caption{\textbf{E}rgodic \textbf{D}rift-\textbf{P}lus-\textbf{P}enalty (EDPP)}
		\label{alg0}
		\begin{algorithmic}
	 \STATE {\bfseries Initialize:} Initial iterates $\bm{x_1}\in\mathcal{X}$, $Q_1=0$, and $0<\beta\leq 1/2$.
   \FOR{$t=1$ {\bfseries to} $T$}
   \STATE Observe $f_t$ and $g_t$.
   \STATE Set penalty parameter $V_t$ and step size parameter $\alpha_t$ as $$V_t=(\tau_{\text{mix}}t)^\beta, \quad \alpha_t=\tau_{\text{mix}}t.$$
   \STATE {\bfseries Primal update:} Set $\bmn$ as
   $$\bmn = \argmin_{\bmx\in\mathcal{X}}\left\{\left(V_t\grad f_t(\bmc)+Q_t\grad g_t(\bmc)\right)^\top \bmx + \alpha_t D(\bmx,\bmc)\right\}$$
   \STATE {\bfseries Dual update:} Set $Q_{t+1}$ as
   $$Q_{t+1}=\left[Q_t + g_t(\bmc) +\grad g_t(\bmc)^\top(\bmn-\bmc) \right]_+$$
   \ENDFOR
		\end{algorithmic}
\end{algorithm}

Note that the mixing time $\tmix$ is now part of parameters $V_t$ and $\alpha_t$, as in the ergodic mirror descent algorithm by~\cite{emd}. Second, $V_t$ and $\alpha_t$ are time-varying, so our algorithm is \emph{adaptive}~\citep{long-term-2} and oblivious to the length of the time horizon $T$. One may wonder why we do not use $T$ instead of $t$, i.e., $V=(\tmix T)^\beta$ and $\alpha = \tmix T$. In fact, our numerical results, which will be presented in \Cref{sec:exp}, demonstrate that \Cref{alg0} with the adaptive parameters $V_t$ and $\alpha_t$ outperforms the algorithm with fixed parameters $V=(\tmix T)^\beta$ and $\alpha = \tmix T$. However, the performance analysis of the vanilla DPP algorithm by~\cite{OCO-stochastic} does not immediately extend to such adaptive parameters. Hence, we prove that the DPP framework with parameters of varying $t$ still achieves the desired regret and constraint violation guarantees. 

Let us also briefly explain how the DPP framework initially developed by~\cite{OCO-stochastic} as well as our \Cref{alg0} works. We may regard $Q_t$ as the size of a \emph{virtual queue} at time $t$. Then we consider the associated \emph{quadratic Lyapunov term} $L_t=Q_t^2/2$ and study the corresponding \emph{drift} given by
$\Delta_t = L_{t+1}-L_t = ({Q_{t+1}^2}-{Q_t^2})/2$. It is not difficult to see that $$\Delta_t\leq Q_t\left(g_t(\bmc)+\grad g_t(\bmc)^\top(\bmn-\bmc)\right)+\frac{1}{2}(H_t+G_tR)^2$$
holds~(Lemma~\ref{lemma:drift-bound}). Here, the upper bound on the drift $\Delta_t$ has term $Q_t\grad g_t(\bmc)^\top(\bmn-\bmc)$ that depends on the next iterate $\bmn$. Hence, by choosing $\bmn$ that minimizes $Q_t\grad g_t(\bmc)^\top(\bmn-\bmc)$, we may attempt to control the drift. In fact, the primal update sets $\bmn$ to be the minimizer of
\begin{align*}
	&\underbrace{Q_t\grad g_t(\bmc)^\top(\bmx-\bmc)}_{\text{drift}}+ \underbrace{V_t \grad f_t(\bmc)^\top(\bmx-\bmc) + \alpha_tD(\bmx,\bmc)}_{\text{penalty}}
\end{align*}
over $\mathcal{X}$. Consequently, at each iteration, we get to choose a solution that minimizes the drift term $\Delta_t$ and a penalty term for controlling the objective simultaneously.

\subsection{Performance Guarantees of Ergodic Drift-Plus-Penalty}

First, we analyze the regret and constraint violation of EDPP for the constrained online convex optimization setting under the Markovian sampling regime, formulated in \Cref{subsec:oco}. Recall that $\tmix=\tmix(1/4)$ and that parameter $\tau$ is defined as
$\tau=\tmix(T^{-1}),$
which satisfies
$\tau\leq \lceil\log_2 T\rceil \tmix$.

\begin{theorem}\label{thm1}
Suppose that Assumptions \ref{ass:bounded} and \ref{ass:lip} hold. Then for online convex optimization with ergodic constraints, \Cref{alg0} achieves
\begin{align*}
\mathbb{E}\left[\operatorname{Regret}(T)\right] &=O\left(\tmix^{-\beta}\tau T^{1-\beta}\right),\\
\mathbb{E}\left[\operatorname{Violation}(T)\right] &= O\left(\tmix^{{\beta}/{2}}T^{{(\beta+1)}/{2}}+\sqrt{(\tau-1)T}\right)
\end{align*}
where the expectation is taken with respect to the randomness in running the algorithm.
\end{theorem}
One of the key components of the analysis for proving \Cref{thm1} is that we bound the expected size of the virtual queue $Q_t$ at time $t$ as follows.
$$\mathbb{E}[Q_t]=O\left(\tmix^{\beta/2} T^{(\beta+1)/2} + \sqrt{(\tau-1) T}\right).$$
Another key part is to analyze the term
$$\sum_{t=1}^T\mathbb{E}\left[Q_t g_t(\bmx)\right]$$
where function $g_t$ is not independent of the virtual queue size $Q_t$ under our Markovian sampling regime. In contrast, if $g_1,\ldots, g_T$ were i.i.d., then $\mathbb{E}\left[Q_t g_t(\bmx)\right]=\mathbb{E}\left[Q_t \bar g(\bmx)\right]$ would hold. Instead, we relate the term with $$\sum_{t=1}^{T-\tau +1}\mathbb{E}\left[Q_t g_{t+\tau -1}(\bmx)\right]$$
and use the intuition that the distribution of the ergodic Markov chain after $\tau$ steps is close to its stationary distribution. We provide a bound on the term in Lemma~\ref{Q_t-bound1}. Similarly, we also need to analyze the term 
$$\sum_{t=1}^T\mathbb{E}\left[\frac{Q_t}{V_t} g_t(\bmx)\right],$$
and an upper bound on the sum is given in Lemma~\ref{Q_t-bound2}. Based on these observations, we provide an upper bound on the expected virtual queue size $\mathbb{E}\left[Q_t\right]$, which is given in Lemma~\ref{lem:eq}. Then we apply the performance analysis results of the general adaptive drift-plus-penalty method provided in \Cref{sec:adpp}. The complete proof of \Cref{thm1} is given in \Cref{sec:thm1proof}.

Next we analyze the performance of EDPP on the stochastic-constrained stochastic optimization problem under Markovian data sampling, whose formulation is given by~\eqref{scso}.
\begin{theorem}\label{thm2}
   Suppose that Assumptions \ref{ass:bounded} and \ref{ass:lip} hold. Then for stochastic-constrained stochastic optimization~\eqref{scso}, \Cref{alg0} guarantees that
    \begin{align*}
        \mathbb{E}\left[\operatorname{Gap}(\bm{\bar x_T})\right] &=O\left(\frac{\tau}{\tmix^\beta T^{\beta}}+\frac{\tau-1}{\tmix^{1-\beta/2}T^{{(1-\beta)}/{2}}}+\frac{(\tau-1)^{3/2}}{\tmix T^{1/2}}+\frac{\tau}{T}\right),\\
        \mathbb{E}\left[\operatorname{Infeasibility}(\bm{\bar x_T})\right] &=O\left(\frac{\tau}{\tmix^{1-\beta/2}T^{(1-\beta)/2}}+\frac{(\tau-1)^{3/2}}{\tmix T^{1/2}}+\frac{\tau}{T}\right)
    \end{align*}
    where $\bm{\bar x_T}= \sum_{t=1}^T \bmc/T$ and the expectation is taken with respect to the randomness in running the algorithm.
\end{theorem}
Note that \Cref{thm1} implies
$$\frac{1}{T}\mathbb{E}\left[\operatorname{Regret}(T)\right] =O\left(\frac{\tau}{\tmix^\beta T^{\beta}}\right),$$
but the bound on the optimality gap given in \Cref{thm2} has additional terms due to the difference between the stationary distribution of the ergodic Markov chain and the distribution of $f_{t+1}$ conditioned on $f_t$. In fact, under  Markovian sampling, we have 
$$\sum_{t=1}^T\mathbb{E}\left[f_t(\bmc)\right]\neq \sum_{t=1}^T\mathbb{E}\left[\bar f (\bmc)\right].$$
Here, to get around this issue, we also use the intuition that 
$\mathbb{E}\left[\bar f(\bmc)\right]$ is close to $\mathbb{E}\left[f_{t+\tau -1}(\bmc)\right]$ as the distribution of the Markov chain after $\tau$ steps is close to its stationary distribution. The proof of \Cref{thm2} is given in \Cref{sec:thm2proof}. Moreover, since $\tau=\tilde{O}(\tmix)$, it follows that
\begin{align*}
\mathbb{E}\left[\operatorname{Regret}(T)\right] &=\tilde{O}\left(\tmix^{1-\beta} T^{1-\beta}\right),\quad
\mathbb{E}\left[\operatorname{Violation}(T)\right] = \tilde{O}\left(\tmix^{{\beta}/{2}}T^{{(\beta+1)}/{2}}+\sqrt{\tmix T}\right),\\
\mathbb{E}\left[\operatorname{Gap}(\bm{\bar x_T})\right] &=\tilde{O}\left(\frac{\tmix^{1-\beta}}{T^{\beta}}+\frac{\tmix^{{\beta}/{2}}}{T^{{(1-\beta)}/{2}}}+\frac{\tmix}{T}\right),\\
 \mathbb{E}\left[\operatorname{Infeasibility}(\bm{\bar x_T})\right] &=\tilde{O}\left(\frac{\tmix^{\beta/2}}{T^{(1-\beta)/2}}+\frac{\sqrt{\tmix}}{\sqrt{T}}+\frac{\tmix}{T}\right)
 \end{align*}
 where the $\tilde O$ hides a $\log T$ factor. In particular, if we set $\beta = 1/3$, then we have $\mathbb{E}\left[\operatorname{Regret}(T)\right]=\tilde{O}(\tmix^{2/3} T^{2/3})$, 
$\mathbb{E}\left[\operatorname{Violation}(T)\right] = O(\tmix^{1/6}T^{2/3}+\tmix^{1/2}T^{1/2})$, 
$\mathbb{E}\left[\operatorname{Gap}(\bm{\bar x_T})\right]=\tilde{O}(\tmix^{2/3} T^{-1/3} + \tmix T^{-1})$, and 
$\mathbb{E}\left[\operatorname{Infeasibility}(\bm{\bar x_T})\right] = O(\tmix^{1/6}T^{-1/3} + \tmix^{1/2} T^{-1/2} + \tmix T^{-1})$.
Furthermore, observe that if $\{\bmxic\}_{t=1}^\infty$ is a sequence of i.i.d. random variables, then we have $\tmix=\tau=1$. In this case, by \Cref{thm1,thm2}, \Cref{alg0} guarantees that
$\mathbb{E}\left[\operatorname{Regret}(T)\right]=O(T^{1-\beta})$, $\mathbb{E}\left[\operatorname{Violation}(T)\right]=O(T^{(\beta+1)/2})$, $\mathbb{E}\left[\operatorname{Gap}(\bm{\bar x_T})\right]=O({T^{-\beta}})$, and $\mathbb{E}\left[\operatorname{Infeasibility}(\bm{\bar x_T})\right]=O({T^{-(1-\beta)/2}})$ for any $\beta\in(0,1/2]$, which recovers the result of~\cite{long-term-2}.

If we further assume that Slater's constraint qualification holds, we can argue that we get a better control on the size of $\mathbb{E}[Q_t]$. Note that the upper bound on the expected queue size given by Lemma~\ref{lem:eq} is $\mathbb{E}[Q_t]=O(\tmix^{\beta/2} t^{(\beta+1)/2})$ which holds regardless of whether Slater's condition holds or not. On the other hand, we will argue that under Slater's condition, we have
$$\mathbb{E}[Q_t]= O\left(\frac{\tau(\tau+t)}{\sqrt{\tmix t}}\right)=\tilde O(\sqrt{\tmix t}).$$
This is consistent with~\cite{OCO-stochastic} as they proved that $\mathbb{E}[Q_t]=O(\sqrt{T})$ for the i.i.d. setting. In fact, our proof for bounding $\mathbb{E}[Q_t]$ is more involved than the argument of~\cite{OCO-stochastic} because we use adaptive step sizes for \Cref{alg0} and consider non-i.i.d. constraint functions. This leads to improvements as stated in the following result.

\begin{theorem}\label{thm3}
Suppose that Assumptions \ref{ass:bounded}--\ref{ass:slater} hold. Then for online convex optimization with ergodic constraints and stochastic-constrained stochastic optimization, \Cref{alg0} with $\beta = 1/2$ guarantees 
\begin{align*}
\mathbb{E}\left[\operatorname{Regret}(T)\right]&=O\left(\frac{\tau\sqrt{T}}{\sqrt{\tmix}}\right),\quad \mathbb{E}\left[\operatorname{Violation}(T)\right]=O\left(\frac{\tau(\tau+T)}{\sqrt{\tmix T}}\right),\\
 \mathbb{E}[\operatorname{Gap}(\bm{\bar x_T})]&=O\left(\frac{\tau^{2}}{\tmix^{3/2}\sqrt{T}}+\frac{\tau^{5/2}}{\tmix^{3/2}T}\right),\\
    \mathbb{E}[\operatorname{Infeasibility}(\bm{\bar x_T})]&=O\left(\frac{\tau^{2}}{\tmix^{3/2}\sqrt{T}} + \frac{\tau^{5/2}}{\tmix^{3/2}T}+\frac{\tau^2}{\tmix^{1/2} T^{3/2}}\right),
\end{align*}
where the expectation is taken with respect to the randomness in running the algorithm.
\end{theorem}
Our analysis takes into account the time-varying algorithm parameters $V_t$ and $\alpha_t$ as well as the fact that the functions are correlated according to a Markov chain. To consider this, we prove Lemmas \ref{lem:process} and \ref{slater-drift-ergodic} that lead to time-varying bounds on the expected virtual queue size. The complete proof of \Cref{thm3} is included in \Cref{sec:thm3proof}.
Since $\tau = \tilde O(\tmix)$, 
\begin{align*}
\mathbb{E}\left[\operatorname{Regret}(T)\right]&=\tilde O\left({\sqrt{\tmix T}}\right),\quad \mathbb{E}\left[\operatorname{Violation}(T)\right]=\tilde O\left(\sqrt{\tmix T}+\frac{\tmix^{3/2}}{\sqrt{T}}\right),\\
 \mathbb{E}[\operatorname{Gap}(\bm{\bar x_T})]&=\tilde O\left(\frac{\sqrt{\tmix}}{\sqrt{T}}+\frac{\tmix}{T}\right),\quad \mathbb{E}[\operatorname{Infeasibility}(\bm{\bar x_T})]=\tilde O\left(\frac{\sqrt{\tmix}}{\sqrt{T}} + \frac{\tmix}{T}+\frac{\tmix^{3/2}}{ T^{3/2}}\right).
\end{align*}

\section{Unknown Mixing Time}\label{sec:unknown}

Next, we study the setting where the mixing time $\tmix$ is not observable. Even if we do not know the mixing time of the underlying Markov chain, we would still want to provide a learning algorithm that provides performance guarantees of similar orders. To achieve this goal, we develop yet another variant of the drift-plus-penalty algorithm which incorporates the multi-level Monte Carlo (MLMC) gradient estimation scheme~\citep{giles_2015,blanchet-WSC,Levy22}. \cite{Levy22} first introduced the approach of combining stochastic gradient descent with the MLMC gradient estimation framework for stochastic optimization with no stochastic functional constraint. 

For the case of known mixing time, we may update the step size based on the mixing time $\tmix$ to achieve an optimal dependence on the parameter. When $\tmix$ is not known, \cite{Levy22} used AdaGrad-based adaptive step sizes~\citep{adagrad,levy2017,ward2019}. We take this idea to develop an AdaGrad variant of the drift-plus-penalty algorithm for \ref{scso}, described in~\Cref{alg1}, that incorporates the MLMC gradient estimation framework. The AdaGrad version of DPP itself is of independent interest. 

\subsection{Multi-Level Monte Carlo Sampling}

The idea behind the multi-level Monte Carlo estimation scheme is to obtain many consecutive samples from an ergodic Markov chain and take their average. At the same time, we may control the expected number of consecutive samples required for each time step by $O(\log T)$.

More precisely, for each time step $t$, we $N_t$ sample $\bm{\xi_t^{(1)}},\ldots, \bm{\xi_t^{(N_t)}}$ where $N_t$ itself is a random variable given by $$N_t=\begin{cases}
    \tilde{N}_t, &\text{if }\tilde{N}_t\leq T^{2}\\
    1, &\text{otherwise}
\end{cases}$$ and $\tilde{N}_t=2^{J_t}$ with $J_t\sim \operatorname{Geom}({1}/{2})$. Note that in our case, the condition is that $\tilde{N}_t\leq T^{2}$ where the bound on $\tilde N_t$ is $T^2$, while it was set to $T$ in~\cite{Levy22}. With this sampling strategy, we define $\mathcal{F}_t$ as the $\sigma$-field
$$\mathcal{F}_t=\sigma\left(\left\{ N_1, \ldots, N_t\right\}\cup \bigcup_{s=1}^t\left\{\bm{\xi_s^{(1)}},\ldots, \bm{\xi_s^{(N_s)}}\right\}\right).$$
Let $\mathbb{E}_t\left[\cdot\right]$ denote the conditional expectation with respect to $\mathcal{F}_t$, i.e., $\mathbb{E}_t\left[\cdot\right]=\mathbb{E}\left[\cdot \mid \mathcal{F}_t\right]$.
Next, for $t\geq 1$ and $N\geq1$, we define
$$f_t^N(\bmx):=\frac{1}{N}\sum_{i=1}^{N}f(\bmx, \bm{\xi_t^{(i)}}),\quad g_t^N(\bmx):=\frac{1}{N}\sum_{i=1}^{N}g(\bmx, \bm{\xi_t^{(i)}}).$$
Based on this, we define the MLMC estimators of $f$ and $g$ as follows.
\begin{equation*}
    (f_t,g_t)=(f_t^1,g_t^1)+
    \begin{cases}
        N_t\left((f_t^{N_t},g_t^{N_t})-(f_t^{N_t/2},g_t^{N_t/2})\right), &\text{if }N_t > 1\\
        0, &\text{otherwise}.
    \end{cases}
\end{equation*}
Basically, functions $f_t$ and $g_t$ are obtained after applying the MLMC estimation scheme to the underlying ergodic Markov chain. One thing to note, however, is that  $f_t$ and $g_t$ are not necessarily convex anymore~\citep{Levy22}. To remedy this issue, what we can argue instead is that $\mathbb{E}_{t-1}[f_t]$ and $\mathbb{E}_{t-1}[g_t]$ are convex. Based on~\citep[Lemma 3.1]{Levy22}, we deduce the following lemma.
\begin{lemma}\label{lem:MLMC}
    Let $j_{\max}:=\max\left\{j\in\mathbb{N}:2^j\leq T^{2}\right\}=\lfloor 2\log_2 T\rfloor$. Then for each $t$, 
    \begin{align*}
        &\mathbb{E}_{t-1}[f_t]=\mathbb{E}_{t-1}\left[f_t^{2^{j_{\max}}}\right],\quad   \mathbb{E}_{t-1}[\grad f_t]=\mathbb{E}_{t-1}\left[\grad f_t^{2^{j_{\max}}}\right],\\
       &\mathbb{E}_{t-1}[g_t]=\mathbb{E}_{t-1}\left[g_t^{2^{j_{\max}}}\right],\quad\mathbb{E}_{t-1}[\grad g_t]=\mathbb{E}_{t-1}\left[\grad g_t^{2^{j_{\max}}}\right].
    \end{align*}
Moreover, we have
     \begin{equation*}
       \mathbb{E}\left[\|\grad f_t(\bmc)\|_*^2\right] = \tilde{O}( F^2 \tau_{\text{mix}}),\quad \mathbb{E}\left[\|\grad g_t(\bmc)\|_*^2\right] = \tilde{O}(G^2\tau_{\text{mix}}),\quad  \mathbb{E}\left[| g_t(\bmc)|^2\right] = \tilde{O}(H^2\tau_{\text{mix}}).
    \end{equation*}
Lastly, the expected number of samples for time step $t$ satisfies $\mathbb{E}[N_t]\leq 2\log_2 T +1$.
\end{lemma}
The reason for setting the upper bound $T^2$ on $\tilde N_t$ instead of $T$ is to achieve high accuracy of estimation for $f_t$, $g_t$, $\grad f_t$, and $\grad g_t$ that leads to the desired performance guarantees of \cref{alg1}. To be specific, we use the following estimation bounds based on \citep[Lemma A.6]{Levy22}.
\begin{lemma}\label{lem:concentration}
There exists $C(T)>0$ with
$$C(T)=O\left(\left(\log(T)\log\left(\tmix T^2\log(T)\right)\right)^{1/2}\right)$$
such that
    \begin{align*}
        \mathbb{E}_{t-1}\left[\left|f_t^{2^{j_{\max}}}(\bmx)-\bar{f}(\bmx)\right|^2\right]&\leq C(T)^2\frac{\tmix}{T^2},\quad \mathbb{E}_{t-1}\left[\lVert\grad f_t^{2^{j_{\max}}}(\bmx)-\grad\bar{f}(\bmx)\rVert_*^2\right]\leq C(T)^2
        \frac{\tmix}{T^2},\\
        \mathbb{E}_{t-1}\left[\left|g_t^{2^{j_{\max}}}(\bmx)-\bar{g}(\bmx)\right|^2\right]&\leq C(T)^2\frac{\tmix}{T^2},\quad \mathbb{E}_{t-1}\left[\lVert\grad g_t^{2^{j_{\max}}}(\bmx)-\grad\bar{g}(\bmx)\rVert_*^2\right]\leq
        C(T)^2\frac{\tmix}{T^2}.
    \end{align*}
    hold for any $\bmx\in\mathcal{X}$ that is measurable with respect to $\mathcal{F}_{t-1}$ and any $t\in[T]$.
\end{lemma}
The complete proof of this lemma is given in \Cref{sec:appendix-mlmc}.

\subsection{Adaptive Drift-Plus-Penalty}

The second component of our algorithm for the unknown mixing time setting is the AdaGrad variant of the  drift-plus-penalty algorithm. To develop AdaGrad-style step sizes, let us define the following sequence of parameters. For some positive constant $\delta>0$, we set
\begin{align*}
    a_0=S_0=\delta,\quad a_t:=\frac{F_t^2}{4}+R^2 G_t^2+H_t^2+\delta, \quad S_t:=\delta+\sum_{s=1}^t a_s
\end{align*}
for $t\geq 1$. Here, we may choose any positive number for $\delta$.  Recall that $F_{t}=\lVert\grad f_t(\bmc)\rVert_*$,  $G_{t}=\lVert\grad g_t(\bmc)\rVert_*$, and  $H_{t}=|g_t(\bmc)|$. Based on these parameters, we set the algorithm parameters $V_t$ and $\alpha_t$ as follows.
\begin{equation}\label{eq:param}
    V_t=\frac{S_{t-1}^\beta}{R},\quad \alpha_t=\frac{S_{t-1}}{R^2}
\end{equation}
for some $0<\beta\leq 1/2$. Here, the penalty parameter $V_t$ and the step size parameter $\alpha_t$ can be chosen without knowledge of the global upper bounds $F, G, H$  on $F_t, G_t, H_t$.

Now we are ready to describe our algorithm, which we call the \emph{MLMC adaptive drift-plus-penalty (MDPP)} algorithm. 
\begin{algorithm}[tb]
		\caption{\textbf{M}LMC Adaptive \textbf{D}rift-\textbf{P}lus-\textbf{P}enalty (MDPP)}
		\label{alg1}
		\begin{algorithmic}
	 \STATE {\bfseries Initialize:} Initial iterates $\bm{x_1}\in\mathcal{X}$, $Q_1=0$ and parameters $0<\beta\leq 1/2$, $\delta>0$.
   \FOR{$t=1$ {\bfseries to} $T$}
   \STATE Observe $f_t$ and $g_t$ via MLMC method.
   \STATE Set penalty parameter $V_t$, step size parameter $\alpha_t$ as (\ref{eq:param}).
   \STATE {\bfseries Primal update:} Set $\bmn$ as
  $$\bmn = \argmin_{\bmx\in\mathcal{X}}\left\{\left(V_t\grad f_t(\bmc)+Q_t\grad g_t(\bmc)\right)^\top \bmx + \alpha_t D(\bmx,\bmc)\right\}$$
  \STATE {\bfseries Dual update:} Set $Q_{t+1}$ as
  $$Q_{t+1}=\left[Q_t + g_t(\bmc) +\grad g_t(\bmc)^\top(\bmn-\bmc) \right]_+$$
   \ENDFOR
		\end{algorithmic}
\end{algorithm}
Here, MDPP is a combination of the AdaGrad-style adaptive drift-plus-penalty algorithm with the MLMC estimator presented in the previous subsection. More importantly, the algorithm is designed to solve constrained online convex optimization where the MLMC estimators $f_1,\ldots, f_T$ are the objective loss functions and the MLMC estimators $g_1,\ldots, g_T$ are the constraint functions. 
The following lemma provides an adaptive regret guarantee and an adaptive constraint violation bound for \Cref{alg1} for the associated online convex optimization. 
\begin{lemma}\label{lem:st1}
Suppose that Assumptions \ref{ass:bounded} and \ref{ass:lip} hold. Then for the constrained online convex optimization problem where the MLMC estimators $f_1,\ldots, f_T$ are the objective loss functions and the MLMC estimators $g_1,\ldots, g_T$ are the constraint functions, \Cref{alg1} achieves the following. For any $\bmx\in \mathcal{X}$ with $\bar g(\bmx)\leq 0$, we have
    \begin{align*}
        &\mathbb{E}\left[\sum_{t=1}^T f_t(\bmc)-\sum_{t=1}^T f_t(\bmx)\right]\\
        &= \tilde{O}\left(\mathbb{E}\left[S_T\right]^{1-\beta}+\tmix^{1/2}\mathbb{E}\left[S_T\right]^{1/2-\beta}+\tmix^{1/2}T^{1/4}\mathbb{E}\left[S_T\right]^{1/4-\beta/2}+\tmix^{1/2}T^{1-\beta}\right),\\
        &\mathbb{E}\left[\sum_{t=1}^T g_t(\bmc)\right]\\
        &=\tilde{O}\left(\mathbb{E}[S_T]^{1/2}+T^{1/4}\mathbb{E}[S_T]^{\beta/2+1/4}+\tmix^{1/2}+\frac{\mathbb{E}[S_T]+T^{1/2}\mathbb{E}[S_T]^{\beta+1/2}}{\tmix^{\beta/2+1/4}T^{\beta/2+1/2}}\right.\\
        &\qquad\quad \left.+\frac{\tmix^{1/2}+\mathbb{E}[S_T]^{1/2}+\mathbb{E}[S_T]^{\beta/2+1/4}}{\tmix^{\beta/2+1/4}T^{\beta/2+1/2}}+(\log S_T)^2\tmix^{\beta/2-1/4}T^{\beta/2+1/2}\right).
    \end{align*}
\end{lemma}
Recall that the parameters $V_t$ and $\alpha_t$ depend on the parameter $\delta$. Here, we may decide any positive number for $\delta$, its choice does affect the performance of \Cref{alg1}. Although the bounds given in Lemma~\ref{lem:st1} do not exhibit an explicit dependence on $\delta$, our proof of Lemma~\ref{lem:st1} in \Cref{sec:thm5proof} reveals that increasing $\delta$ increases the objective gap $\mathbb{E}[\sum_{t=1}^T f_t(\bmc)-\sum_{t=1}^T f_t(\bmx)]$ and decreases the constraint violation $\mathbb{E}[\sum_{t=1}^T g_t(\bmc)]$. Likewise, decreasing $\delta$ decreases the objective gap and increases the constraint violation.

Lemma~\ref{lem:MLMC} implies that $\mathbb{E}\left[S_T\right] = \tilde O(\tmix T)$. Plugging in this bound on $\mathbb{E}\left[S_T\right]$ to the adaptive performance guarantees given in Lemma~\ref{lem:st1}, we deduce the following result. 

\begin{proposition}\label{prop}
   Suppose that Assumptions \ref{ass:bounded} and \ref{ass:lip} hold. Then for the constrained online convex optimization problem where the MLMC estimators $f_1,\ldots, f_T$ are the objective loss functions and the MLMC estimators $g_1,\ldots, g_T$ are the constraint functions, \Cref{alg1} achieves the following. For any $\bmx\in \mathcal{X}$ with $\bar g(\bmx)\leq 0$, we have 
    \begin{align*}
        \mathbb{E}\left[\sum_{t=1}^T f_t(\bmc)-\sum_{t=1}^T f_t(\bmx)\right]&= \tilde{O}\left(\tmix^{1-\beta} T^{1-\beta}\right),\\
        \mathbb{E}\left[\sum_{t=1}^T g_t(\bmc)\right]&= \tilde{O}\left(\tmix^{\beta/2+1/4}T^{\beta/2+1/2}+\tmix^{3/4-\beta/2}T^{1/2-\beta/2}\right).
    \end{align*}
    for any $\bmx\in\mathcal{X}$ satisfying $\bar g(\bmx)\leq 0$.
\end{proposition}
We remark that the performance bounds given in Lemma~\ref{lem:st1} and Proposition~\ref{prop} are not comparable to the regret and constraint violation bounds for stochastic-constrained stochastic optimization. When each pair of loss and constraint functions for stochastic-constrained stochastic optimization corresponds to a single data, the associated regret and constraint violation measure are given by the following.
\begin{align*}
    \operatorname{Regret}(T) &= \sum_{t=1}^T\sum_{j=1}^{N_t} f_t^{(j)}(\bmc)-\sum_{t=1}^T\sum_{j=1}^{N_t} f_t^{(j)}(\bmx^*),\\
    \operatorname{Violation}(T) &= \sum_{t=1}^T\sum_{j=1}^{{N}_t} g_t^{(j)}(\bmc).
\end{align*}
Here, $f_t^{(1)},\ldots, f_t^{(N_t)}$ are the $N_t$ sampled functions from which we derive the MLMC estimator $f_t$ for $t\in[T]$. In contrast, Lemma~\ref{lem:st1} and Proposition~\ref{prop} analyze the performance of \Cref{alg1} on the sequence of the MLMC estimators. Despite this, it would be an interesting question to understand the performance of \Cref{alg1} for the latter online convex optimization setting where each function pair corresponds to a single data.

The proof of Lemma~\ref{lem:st1} is given in \Cref{sec:mdpp-proof}. One of the main components of the analysis is to provide adaptive bounds on the terms
$\mathbb{E}\left[Q_t\right]$ and $\mathbb{E}\left[{Q_t}/{V_t}\right]$. This is possible thanks to our subtle choice of parameters $V_t$ and $\alpha_t$. It turns out that the AdaGrad style analysis for drift-plus-penalty is more sophisticated than that for unconstrained stochastic gradient descent.

Finally, we state the following theorem providing upper bounds on the optimality gap and the feasibility gap under \Cref{alg1} for \ref{scso}. The argument is to use the results of Proposition~\ref{prop} and the estimation error bounds due to Lemma~\ref{lem:concentration}. In contrast to the setting of \Cref{sec:tmix} for which we had to rely on the mixing property of ergodic Markov chains directly, the MLMC estimators are already close to $\bar f$ and $\bar g$.

\begin{theorem}\label{thm5}
   Suppose that Assumptions \ref{ass:bounded} and \ref{ass:lip} hold. Then for stochastic-constrained stochastic optimization \eqref{scso}, \Cref{alg1} guarantees that
    \begin{align*}
        \mathbb{E}\left[\operatorname{Gap}(\bm{\bar x_T})\right] &=\tilde{O}\left(\frac{\tmix^{1-\beta}}{T^{\beta}}\right),\\
        \mathbb{E}\left[\operatorname{Infeasibility}(\bm{\bar x_T})\right] &=\tilde{O}\left(\frac{\tmix^{(2\beta+1)/4}}{T^{(1-\beta)/2}}+\frac{\tmix^{(3-2\beta)/4}}{T^{(\beta+1)/2}}\right).
    \end{align*}
\end{theorem}

\section{Numerical Experiments}\label{sec:exp}

We examine the performance of the ergodic drift-plus-penalty algorithm (\Cref{alg0}) for the known mixing time case and the MLMC adaptive drift-plus-penalty algorithm (\Cref{alg1}) for the unknown mixing time case on a linear classification problem with fairness constraints using synthetic data. We follow the experimental setup of \cite{fairness-constraints-jmlr}. We adopt \cite{fairness-constraints-jmlr} for creating data points and sensitive features (with $\phi=\pi/2$) and imposing fairness constraints.
To make our Markov chain more meaningful and complex enough, we experiment with a 3-state chain instead of a 2-state chain in contrast to \cite{Levy22}. The $3$-state Markov chain is given with the following transition matrix
\begin{center}
$
\begin{pmatrix}
1-2p & p & p\\
p & 1-2p & p\\
p & p & 1-2p
\end{pmatrix}$ 
\end{center}
which has stationary distribution $(1/3,1/3,1/3)$. 

\begin{theorem}\label{thm:tmix}\citep[Theorems 12.4 and 12.5]{levin2017markov}
    For an ergodic and reversible Markov chain with $n$ states, whose transition matrix is $P$, let $1=\lambda_1>\lambda_2\geq\ldots\geq \lambda_n$ be the eigenvalues of $P$ and $\mu_{\min}$ be the minimum entry of the stationary distribution. Then, the mixing time of the chain satisfies
    $$\frac{|\lambda_2|}{1-\max\{|\lambda_2|, |\lambda_n|\}}\log 2 \leq \tmix\leq \frac{1}{1-\max\{|\lambda_2|, |\lambda_n|\}}\log\left(\frac{4}{\mu_{\min}}\right).$$
\end{theorem}
Then it follows that the mixing time of the 3-state Markov chain satisfies $$\frac{1-3p}{3p}\log 2\leq\tmix\leq \frac{1}{3p}\log 12.$$
In our experiment, we used $1/3p$ as an approximation of the mixing time.

For each set of data points, we generate two clusters, each of which has 1,000 data points in $\mathbb{R}^2$ sampled from a multivariate normal distribution. All data points from a cluster have the same label in $\{-1,1\}$. We denote the index set corresponding to each state $j\in\{1,2,3\}$ as $D_j$ and the whole index set as $D=D_1\cup D_2\cup D_3$.
The data clusters are drawn in \Cref{fig:data}. 
\begin{figure}[ht]
	\includegraphics[width=6.5cm]{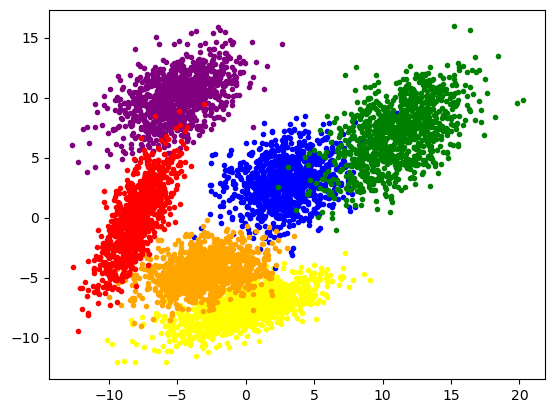}
	\centering
	\caption{Data Points}
	\label{fig:data}
\end{figure}
\noindent
The color corresponding to each state and label is summarized in Table \ref{tab:data}.
 \begin{table}[ht]
 	\begin{center}
 	\begin{tabular}{c|cc}\toprule
 	\backslashbox{state}{label} &1 &$-1$\\
 	\hline
 	1 & purple & yellow \\
 	2 & blue & orange \\
 	3 & green & red \\
 	\bottomrule
 \end{tabular}
 		\caption{Colors for (State, Label) Pairs}\label{tab:data}
 	\end{center}
 \end{table}
We then generate binary sensitive feature $z_i\in\{0,1\}$ randomly for each data point $x_i\in\mathbb{R}^d$, i.e., gender. We want the binary-sensitive feature of the data points to have low covariance with the results from our classifier. More details about how to create sensitive features are included in the supplement.

We use logistic regression classifiers with the following loss functions
$$f_j(w,b)=\frac{1}{|D_j|}\sum_{i\in D_j}\log(1+e^{-y_i(w^\top x_i+b)})$$
and constraint functions
\begin{align*}
    g_j(w,b) &= \frac{1}{|D_j|}\sum_{i\in D_j}(z_i-\bar{z})(w^\top x_i+b)-c\\
    h_j(w,b) &=-\frac{1}{|D_j|}\sum_{i\in D_j}(z_i-\bar{z})(w^\top x_i+b)-c
\end{align*}
for each state $j$, where $\bar{z}=\sum_{i\in D}z_i/|\mathcal{D}|$ and $c>0$.
Then the stochastic-constrained stochastic optimization problem is to minimize the usual logistic regression loss function under fairness constraints, which was proposed by \cite{fairness-constraints-jmlr}, as follows.
\begin{align}%
\begin{aligned}
    &\min_{(w,b)\in\mathcal{X}}\quad  \frac{1}{|D|}\sum_{i\in D}\log(1+e^{-y_i(w^\top x_i+b)})\notag\\
    &\text{s.t.} \quad -c\leq \frac{1}{|D|}\sum_{i\in D}(z_i-\bar{z})(w^\top x_i+b)\leq c. 
\end{aligned}
\end{align}
Solving this problem with our framework can be viewed as a distributed optimization scheme of \cite{ram09} and \cite{johansson2010}. Basically, 
there are agents 1,2, and 3 sharing $\bar{z}$, and agent $i$ has data ${D}_i$. After we update our parameters $w_t$ and $b_t$, they are sent to agent $i_t$, and the agent sends us back the information $\grad f_{i_t}(w_t,b_t), g_t(w_t, b_t), \grad g_{i_t}(w_t, b_t)$, $h_t(w_t, b_t), \grad h_{i_t}(w_t, b_t)$. Here, we may impose that the sequence of selected agents gives rise to an ergodic Markov chain.

The experimental setup involves two constraints. Although we consider the single constraint setting in the paper for simplicity, our results can be easily extended to the case with multiple constraint functions $\bar g_1,\ldots, \bar g_n$. For each $\bar g_i$, we obtain sampled functions $g_{t,i}$ for $t\in[T]$. Then for each $t$, we update $\bmx_t$ and $\{Q_{t,i}\}_{i=1}^n$ as
\begin{align*}
    \bmx_{t+1} &=\argmin_{\bmx\in\mathcal{X}}\left\{\left(V_t\grad f_t(\bmx_t)+\sum_{i=1}^n Q_{t,i} g_{t,i}(\bmx_t)\right)^\top \bmx+\alpha_t D(x,\bmx_t)\right\},\\
    Q_{t+1,i} &= \left[Q_{t,i}+g_{t,i}(\bmx_t)+\grad g_{t,i}(\bmx_t)^\top (\bmx_{t+1}-\bmx_t)\right]_+,\quad i=1,\ldots,n.
\end{align*}
For \Cref{alg0}, we use the parameters $V_t=(\tmix t)^\beta$ and $\alpha_t=\tmix t$ as before. For \Cref{alg1}, we define the MLMC estimator $g_{t,i}$ using $\{g_{t,i}^{(j)}\}_{j=1}^{N_t}$ for each $i\in [n]$ and define $$a_0=S_0=\delta,\quad a_t=\frac{F_t^2}{4}+\sum_{i=1}^n R^2 G_{t,i}^2+\sum_{i=1}^n H_{t,i}^2,\quad S_t=\delta+\sum_{s=1}^t a_s,$$
where $G_{t,i}=\|\grad g_{t,i}(\bmx_t)\|_*, H_{t,i}=|g_{t,i}(\bmx_t)|$.
The parameters $V_t$ and $\alpha_t$ are defined as in \eqref{eq:param}.

We compare our DPP-based algorithms with some existing algorithms developed for stochastic-constrained stochastic optimization with i.i.d. data. The list of algorithms that we tested is given as follows.
\begin{itemize}
    \item PD : Primal-dual method by \cite{long-term-1}.
    \item PD2 : Primal-dual method by \cite{long-term-2}.
    \item DPP : Drift-plus-penalty algorithm by \cite{OCO-stochastic}.
    \item EDPP-t : Ergodic drift-plus-penalty (\Cref{alg0}).
    \item EDPP-T : modification of \Cref{alg0} with non-adaptive parameters $V_t=\sqrt{\tmix T}$ and $\alpha_t=\tmix T$.
    \item MDPP : MLMC adaptive drift-plus-penalty (\Cref{alg1}).
\end{itemize}
For MDPP, we observed that the MLMC estimator usually has a high variance in practice, making experimental results unstable. Hence, we truncated MLMC sampling so that the number of samples per iteration is at most $2^4$. We chose $\delta=F^2/4+2R^2 G^2+2H^2$, for which we computed the constants $F,G,H>0$ such that $F_t\leq F, G_{t,i}\leq G, H_{t,i}\leq H$.

We set parameters to $p=0.001$ and $c=0.5$ with which we ran the list of algorithms with the same initial parameters and sequence of states. We first ran MDPP with 25,000 iterations which created 101,034 samples. The results on the optimality gap are summarized in \Cref{fig:opt_gap}, and the results on the infeasibility are presented in  \Cref{fig:c1v} and \Cref{tab:final-value}.
The results on the regret and the cumulative constraint violations are shown in  \Cref{fig:regret,fig:cc1v}, respectively.

\begin{figure}[ht]
    \vspace{.3in}
    \includegraphics[width=6.5cm]{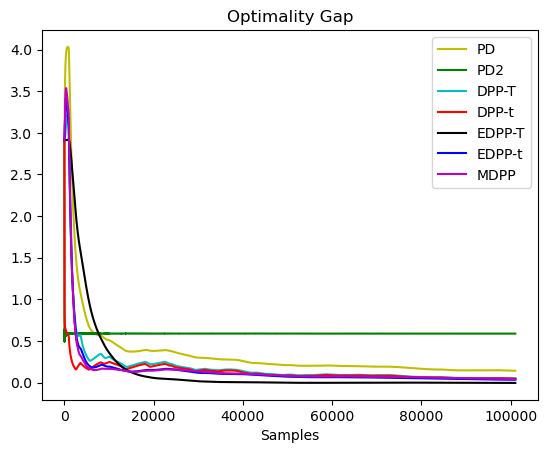}
    \quad 
    \includegraphics[width=6.5cm]{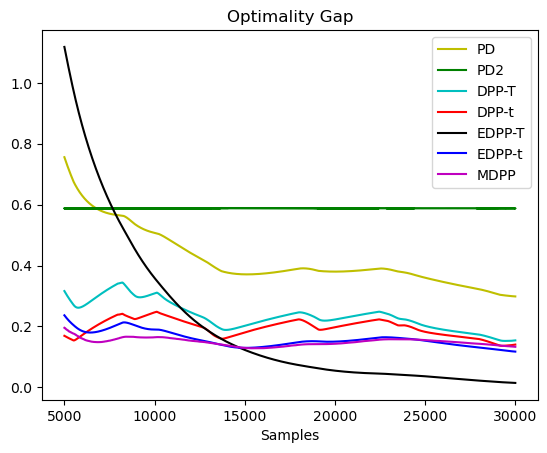}
    \centering
    \caption{Optimality Gap (Left), Enlarged Figure around 5,000 - 30,000 Samples (Right)}
    \label{fig:opt_gap}
\end{figure}

\begin{figure}[ht]
    \vspace{.3in}
    \includegraphics[width=6.35cm]{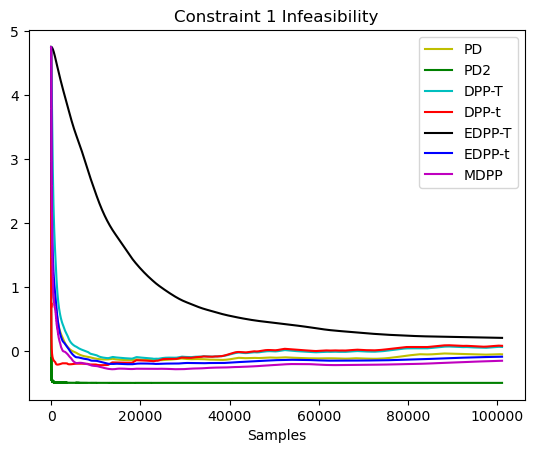}
    \quad 
     \includegraphics[width=6.5cm]{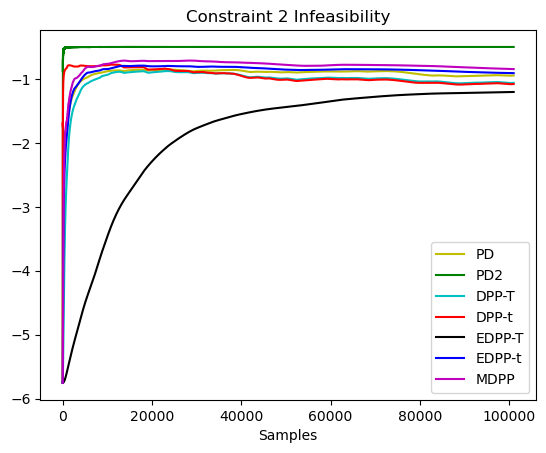}
    \centering
    \caption{Constraint 1 Infeasibility (Left), Constraint 2 Infeasibility (Right)}
    \label{fig:c1v}
\end{figure}

\begin{table}[ht]
   
    \begin{center}
    \begin{tabular}{c|c}\toprule
    Algorithm & Final values of\\ 
    & constraint 1 infeasibility\\
    \hline
    PD & $-0.0533$ \\
    PD2 & $-0.4997$ \\
    DPP-T & 0.0622 \\
    DPP-t & 0.0800\\
    EDPP-T & 0.2042\\
    EDPP-t & $-0.0904$\\
    MDPP & $-0.1536$\\
\bottomrule
    \end{tabular}
 \caption{Final Values of Constraint 1 Infeasibility}\label{tab:final-value}
    \end{center}
\end{table}

\begin{figure}[ht]
    \vspace{.3in}
    \includegraphics[width=6.5cm]{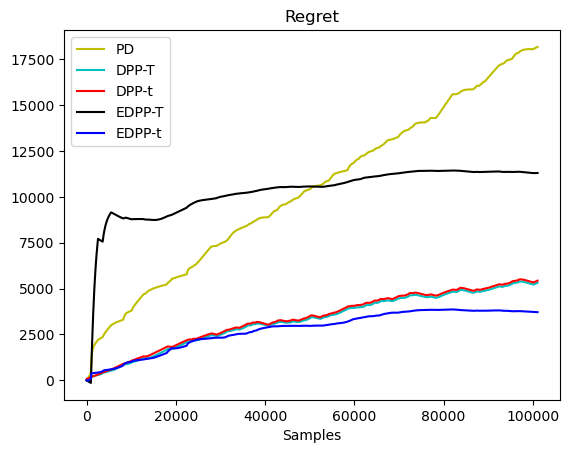}
    \centering
    \caption{Regret}
    \label{fig:regret}
\end{figure}

\begin{figure}[ht]
    \vspace{.3in}
    \includegraphics[width=6.4cm]{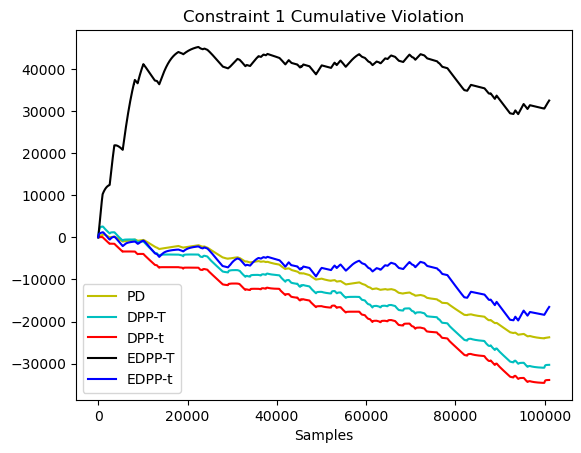}
    \quad 
    \includegraphics[width=6.5cm]{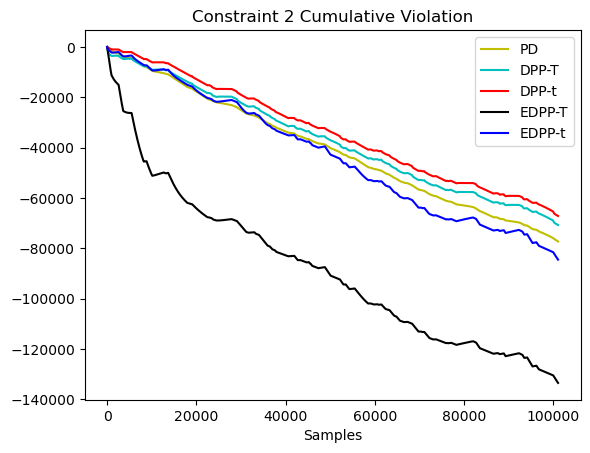}
    \centering
    \caption{Constraint 1 Cumulative Violation (Left), Constraint 2 Cumulative Violation (Right)}
    \label{fig:cc1v}
\end{figure}

As shown in the figures, Our algorithms (EDPP-T, EDPP-t, MDPP) outperform the other algorithms in terms of the optimality gap. DPP also shows a good optimality gap but it ends with a positive constraint 1 infeasibility. In contrast, EDPP-T, EDPP-t, and MDPP all end with a negative constraint infeasibility. Note that after 20,000 samples, EDPP-T achieves the smallest optimality gap, followed by EDPP-t and MDPP. However, \Cref{fig:c1v} shows that EDPP-T incurs a significantly higher infeasibility for constraint 1, given by
$$\frac{1}{|D|}\sum_{i\in D}(z_i-\bar{z})(w^\top x_i+b)\leq c$$
than the other algorithms. In contrast, EDPP-t outperforms DPP-T and DPP-t in terms of both the optimality gap and the infeasibility measure. It is also interesting to see that DPP-T and DPP-t behave similarly, while DPP-t performs  better than  DPP-T. 

\Cref{fig:regret} shows the regret values under various algorithms for the online convex optimization setting where each pair $(f_t,g_t)$ of functions corresponds to one data sample. That said, we excluded MDPP as it requires multiple samples in one round. In addition, we excluded PD2, which exhibits much higher regret values than the other algorithms, to focus on and better present the performance of the other algorithms. \Cref{fig:cc1v} shows that EDPP-T incurs a positive cumulative constraint 1 violation, while the other algorithms result in a negative cumulative constraint violation. We may check from \Cref{fig:regret,fig:cc1v} that EDPP-t performs the best for online convex optimization with ergodic constraints.

\section{Analysis of Ergodic Drift-Plus-Penalty for the Known Mixing Time Case}\label{sec:edpp-proof}

This section presents the proofs of \Cref{thm1,thm2,thm3} given in \Cref{sec:tmix}. \Cref{sec:thm1proof} contains the proof of \Cref{thm1} which provides regret and constraint violation bounds on the formulation of online convex optimization with ergodic constraints. Then \Cref{sec:thm2proof} presents the proof of \Cref{thm2} that gives bounds on the optimality gap and the feasibility gap for~\ref{scso}. In \Cref{sec:thm3proof}, we prove \Cref{thm3} for the case where Slater's condition is satisfied. As before, throughout this section, we denote by $\mathcal{F}_t$ the $\sigma$-field generated by the information accumulated up to time step $t$. That is,
$$\mathcal{F}_t=\sigma\left(\left\{\bm{\xi_1,\ldots, \xi_t}\right\}\right).$$
Moreover, $\mathbb{E}_t[\cdot]$ in this section refers to the conditional expectation with respect to $\mathcal{F}_t$, i.e., $\mathbb{E}_t\left[\cdot\right]=\mathbb{E}\left[\cdot \mid \mathcal{F}_t\right]$ (In \Cref{sec:unknown,sec:mdpp-proof}, $\mathcal{F}_t$ denotes the $\sigma$-field generated up to time step $t$ under the MLMC estimation scheme). Furthermore, $\mathbb{P}_{[t]}^s$ for $t>s$ denotes the probability measure of $\bm{\xi_s}$ conditional on $\mathcal{F}_t$.

\subsection{Ergodic Drift-Plus-Penalty for Online Convex Optimization with Ergodic Constraints}\label{sec:thm1proof}

The three important components of our analysis are the one (Lemma~\ref{Q_t-bound1}) bounding the term 
$\mathbb{E}\left[\sum_{t=1}^T Q_t g_t(\bmx)\right],$
the part (Lemma~\ref{Q_t-bound2}) providing an upper bound on the term
$\mathbb{E}\left[\sum_{t=1}^T ({Q_t}/{V_t}) g_t(\bmx)\right],$
and Lemma~\ref{lem:eq} that derives an upper bound on the expected queue size $\mathbb{E}\left[Q_t\right]$.
Then, plugging in the deduced bounds to the lemmas in \Cref{sec:adpp} analyzing the general template of adaptive drift-plus-penalty, we prove \Cref{thm1}.

By the update rule and the convexity of $g_t$, we deduce the following straightforward bound on $|Q_{t+1}-Q_t|$.
\begin{lemma}\label{lem:trivial}
For $t\geq 1$, we have $- H - GR\leq Q_{t+1}-Q_t\leq H$.
\end{lemma}
\begin{proof}
Note that we have
\begin{align*}
Q_{t+1}&=\left[Q_t+g_t(\bmc)+\grad g_t(\bmc)^T(\bmn-\bmc)\right]_+\leq \left[Q_t+g_t(\bmn)\right]_+\leq Q_t+H.
\end{align*}
For the lower bound, 
\begin{align*}
Q_{t+1}&\geq Q_t+g_t(\bmc)+\grad g_t(\bmc)^T(\bmn-\bmc)\\
&\geq Q_t- \left|g_t(\bmc)\right|-\left|\grad g_t(\bmc)\right\|_*\left\|\bmn-\bmc\right\|\\
&\geq Q_t - H - GR,
\end{align*}
as required.
\end{proof}
As $Q_1=0$, it follows that $Q_t\leq (t-1)H$. Recall that we denote $\tau=\tmix(T^{-1})$. %
The next lemma provides a bound on the term $\sum_{t=1}^T \mathbb{E}\left[ Q_t g_t(\bmx)\right]$. As mentioned in \Cref{sec:tmix-1}, we need to take into account that $g_1,\ldots, g_T$ are not i.i.d. while they are generated by an ergodic Markov chain.

\begin{lemma}\label{Q_t-bound1}
    For any $\bmx\in\mathcal{X}$ such that $\bar{g}(\bmx)\leq 0$,
    $$\mathbb{E}\left[\sum_{t=1}^T Q_t g_t(\bmx)\right]\leq H(H+GR)(\tau-1) T+\frac{2H}{T}\sum_{t=1}^T\mathbb{E}[Q_t].$$
\end{lemma}
\begin{proof}
    If $T<\tau$, then
    $$\sum_{t=1}^T Q_t g_t(\bmx)\leq \sum_{t=1}^T(t-1)H^2=\frac{T(T-1)H^2}{2}<H^2(\tau-1) T,$$
    and the statement follows.
    If $T\geq \tau$, then
    \begin{align*}
        &\mathbb{E}\left[\sum_{t=1}^T Q_t g_t(\bmx)\right]\\
        &=\sum_{t=1}^{\tau-1}\mathbb{E}[Q_t g_t(\bmx)]+\sum_{t=1}^{T-\tau+1}\mathbb{E}[Q_{t+\tau-1}g_{t+\tau-1}(\bmx)]\\
        &\leq \sum_{t=1}^{\tau-1} (t-1)H^2 + \sum_{t=1}^{T-\tau+1} \mathbb{E}[(Q_{t+\tau-1}-Q_t)g_{t+\tau-1}(\bmx)]+\sum_{t=1}^{T-\tau+1} \mathbb{E}[Q_t g_{t+\tau-1}(\bmx)]\\
        &\leq \frac{(\tau-1)(\tau-2)H^2}{2}+\sum_{t=1}^{T-\tau+1}(\tau-1) H(H+GR)+\sum_{t=1}^{T-\tau+1} \mathbb{E}\left[Q_t \mathbb{E}_{t-1}[g_{t+\tau-1}(\bmx)-\bar{g}(\bmx)]\right]\\
        &=\frac{(\tau-1)(2T-\tau)H(H+GR)}{2}+\sum_{t=1}^{T-\tau+1} \mathbb{E}\left[Q_t \mathbb{E}_{t-1}[g_{t+\tau-1}(\bmx)-\bar{g}(\bmx)]\right]
    \end{align*}
   where the second inequality holds because $(Q_{t+\tau-1}-Q_t)g_{t+\tau-1}(\bmx)\leq (H+GR)H$ due to Lemma~\ref{lem:trivial}, $\bar g(\bmx)\leq 0$, and $Q_t$ is $\mathcal{F}_{t-1}$-measurable. Here, to bound $\mathbb{E}_{t-1}[g_{t+\tau-1}(\bmx)-\bar{g}(\bmx)]$, we consider 
    \begin{align*}
        \mathbb{E}_{t-1}[g_{t+\tau-1}(\bmx)-\bar{g}(\bmx)]&= \mathbb{E}_{t-1}[g(\bmx,\xi_{t+\tau-1})-g(\bmx,\xi)]\\
        &\leq \int_A H (d\mathbb{P}_{[t-1]}^{t+\tau-1}-d\mu)+\int_{A^C} H (d\mathbb{P}_{[t-1]}^{t+\tau-1}-d\mu)\\
        &\leq 2H\lVert \mathbb{P}_{[t-1]}^{t+\tau-1}-\mu\rVert_{TV}
    \end{align*}
    where $\mathbb{P}_{[t-1]}^{t+\tau-1}$ is the probability measure of $\bm{\xi_{t+\tau-1}}$ conditional on $\mathcal{F}_{t-1}$, $\bmxi\sim \mu$, and $A$ is any measurable set. By the definition of $\tau = \tmix(T^{-1})$, it follows that
    $\lVert \mathbb{P}_{[t-1]}^{t+\tau-1}-\mu\rVert_{TV}\leq {1}/{T},$
    which implies that
    $\mathbb{E}_{t-1}[g_{t+\tau-1}(\bmx)-\bar{g}(\bmx)]\leq {2H}/{T}.$
    From this, we deduce the desired statement.
\end{proof}
Next we provide an upper bound on the term $\sum_{t=1}^T \mathbb{E}\left[ Q_t g_t(\bmx)/V_t\right]$, as we mentioned in \Cref{sec:tmix-1}.
\begin{lemma}\label{Q_t-bound2}
    For any $\bmx \in\mathcal{X}$ such that $\bar{g}(\bmx)\leq 0$,
    $$\mathbb{E}\left[\sum_{t=1}^T\frac{Q_t}{V_t}g_t(\bmx)\right]\leq \frac{2H(H+GR)(\tau+1)}{\tmix^\beta (1-\beta)}(T+1)^{1-\beta}.$$
\end{lemma}
\begin{proof}
Consider the case $T<\tau$ first. Note that
    $$\sum_{t=1}^T\frac{Q_t}{V_t}g_t(\bmx)\leq \frac{H^2}{\tmix^\beta}\sum_{t=1}^T t^{1-\beta}\leq \frac{H^2}{\tmix^{\beta}(2-\beta)}(T+1)^{2-\beta}\leq \frac{2H^2(\tau+1)}{\tmix^\beta (1-\beta)}(T+1)^{1-\beta}$$
    where the second inequality is due to Corollary~\ref{cor3} and the third inequality is from $T<\tau$. Next we consider the case $T\geq\tau$. Note that
    $$\mathbb{E}\left[\sum_{t=1}^T\frac{Q_t}{V_t}g_t(\bmx)\right]=\sum_{t=1}^{\tau-1} \mathbb{E}\left[\frac{Q_t}{V_t}g_t(\bmx)\right]+\sum_{t=1}^{T-\tau+1}\mathbb{E}\left[\frac{Q_{t+\tau-1}}{V_{t+\tau-1}}g_{t+\tau-1}(\bmx)\right].$$
    Here the first term on the right-hand side can be bounded as
    $$\sum_{t=1}^{\tau-1} \mathbb{E}\left[\frac{Q_t}{V_t}g_t(\bmx)\right]\leq \frac{H^2}{\tmix^\beta}\sum_{t=1}^{\tau-1} t^{1-\beta}\leq \frac{H^2}{\tmix^\beta (2-\beta)}\tau^{2-\beta}\leq \frac{H^2\tau}{\tmix^\beta (1-\beta)} (T+1)^{1-\beta}$$
    as above. Moreover, the second term can be bounded as follows.
    \begin{align*}
      &\sum_{t=1}^{T-\tau+1}\mathbb{E}\left[\frac{Q_{t+\tau-1}}{V_{t+\tau-1}}g_{t+\tau-1}(\bmx)\right]\\
        &\leq \sum_{t=1}^{T-\tau+1}\mathbb{E}\left[\frac{(Q_{t+\tau-1}-Q_t)}{V_{t+\tau-1}}g_{t+\tau-1}(\bmx)\right]+\sum_{t=1}^{T-\tau+1}\mathbb{E}\left[\frac{Q_t}{V_{t+\tau-1}} g_{t+\tau-1}(\bmx)\right]\\
        &\leq \frac{(\tau-1) H(H+GR)}{\tmix^\beta}\sum_{t=1}^{T-\tau+1}\frac{1}{(t+\tau-1)^\beta}+\sum_{t=1}^{T-\tau+1}\frac{1}{V_{t+\tau-1}}\mathbb{E}[Q_t \mathbb{E}_{t-1}[g_{t+\tau-1}(\bmx)-\bar{g}(\bmx)]]\\
        &\leq \frac{(\tau-1) H(H+GR)}{\tmix^\beta (1-\beta)}(T^{1-\beta}-(\tau-1)^{1-\beta})+\frac{2H}{T}\sum_{t=1}^{T-\tau+1}\mathbb{E}\left[\frac{Q_t}{V_{t+\tau-1}}\right]\\
        &\leq \frac{H(H+GR)}{\tmix^\beta (1-\beta)}\tau (T+1)^{1-\beta}+\frac{2H}{T}\sum_{t=1}^T \mathbb{E}\left[\frac{Q_t}{V_t}\right]\\
        &\leq \frac{H(H+GR)\tau}{\tmix^\beta (1-\beta)} (T+1)^{1-\beta}+\frac{2H^2}{\tmix^\beta T}\sum_{t=1}^T t^{1-\beta}\\
        &\leq \frac{H(H+GR)\tau}{\tmix^\beta (1-\beta)} (T+1)^{1-\beta}+\frac{2H^2}{\tmix^\beta T(2-\beta)}(T+1)^{2-\beta}\\
        &\leq \frac{H(H+GR)(\tau+1)}{\tmix^\beta (1-\beta)}(T+1)^{1-\beta},
    \end{align*}
    where the second inequality is from Lemma~\ref{lem:trivial} and $\bar g(\bmx)\leq 0$, the third inequality holds due to Corollary~\ref{cor3} and the choice of $\tau$, the sixth inequality comes from Corollary~\ref{cor3}, and the last inequality holds because $T(2-\beta)>(T+1)(1-\beta)$.
\end{proof}

Based on Lemmas~\ref{Q_t-bound1} and~\ref{Q_t-bound2}, we can now prove the first part of \Cref{thm1}, which upper bounds the regret of \cref{alg0} for online convex optimization with ergodic constraints.

\begin{proof}[The first part of Theorem \ref{thm1}] 
Lemma \ref{lem:regret} implies that
\begin{align*}
    \mathbb{E}\left[\sum_{t=1}^T f_t(\bmc)-\sum_{t=1}^T f_t(\bmo)\right]\leq \frac{\alpha_T}{V_T}R^2+\frac{F^2}{4}\sum_{t=1}^T\frac{V_t}{\alpha_t}+\frac{(H+GR)^2}{2}\sum_{t=1}^T \frac{1}{V_t}+\sum_{t=1}^T \mathbb{E}\left[\frac{Q_t}{V_t} g_t(\bmo)\right].
\end{align*}
Next using Lemma~\ref{Q_t-bound2} with $\bmo$ satisfying $\bar g(\bmo)\leq0$ and plugging in our choice of $\alpha_t$ and $V_t$, we deduce the following.
\begin{align*}
    &\mathbb{E}\left[\sum_{t=1}^T f_t(\bmc)-\sum_{t=1}^T f_t(\bmx)\right]\\
    &\leq R^2(\tmix T)^{1-\beta}+\frac{F^2 \tmix^{\beta-1}}{4\beta}T^\beta+\frac{(H+G R)^2\tmix^{-\beta}}{2(1-\beta)}T^{1-\beta}+\frac{2H(H+GR)(\tau+1)}{\tmix^\beta (1-\beta)}(T+1)^{1-\beta}\\
    &=O(\tmix^{-\beta}\tau T^{1-\beta}),
\end{align*}
as required.
\end{proof}

Next, we prove the second part of \Cref{thm1}, which gives an upper bound on constraint violation under \cref{alg0}. The following lemma provides a time-varying bound on the expected virtual queue size.

\begin{lemma}\label{lem:eq}
    For $t\in [T+1]$, $\mathbb{E}[Q_t]$ is bounded above by
    $$\frac{3}{2}\sqrt{2I(t-1)}+\frac{\beta+3}{\beta+1}\sqrt{\frac{2FR\tmix^\beta t^{\beta+1}}{1+\beta}}+\frac{3R}{2}\sqrt{2 \tmix (t-1)}+\frac{3H}{2}\sqrt{2(\tau-1) (t-1)}+4H$$
    where $I= H^2 +G^2R^2+ F^2/4$.
\end{lemma}
\begin{proof}
To prove the lemma, we argue by induction. Note that the statement of the lemma trivially holds when $t=1$ because $Q_1=0$. Suppose the statement holds for $t\leq s$  for some $s\geq 1$. What remains is to provide an upper bound on $\mathbb{E}[Q_{s+1}]$. Note that
\begin{align*}
    \mathbb{E}\left[Q_t\left(g_t(\bmc)+\grad g_t(\bmc)^\top (\bmx-\bmc)\right)\right]&=\mathbb{E}\left[Q_t\mathbb{E}_{t-1}\left[g_t(\bmc)+\grad g_t(\bmc)^\top (\bmx-\bmc)\right]\right]\\
    &\leq\mathbb{E}[Q_t \mathbb{E}_{t-1}[g_t(\bmx)]]\\
    &=\mathbb{E}[Q_t g_t(\bmx)].
\end{align*}
Then Lemma \ref{lem:q} together with Jensen's inequality implies the following.
\begin{align*}
    \mathbb{E}[Q_{s+1}]&\leq \sqrt{2s\left(H^2+R^2G^2\right)+2RF\sum_{t=1}^s V_t+2R^2\alpha_s+2\sum_{t=1}^s\mathbb{E}[Q_t g_t(\bmx)]}.
\end{align*}
Moreover, $\sum_{t=1}^s\mathbb{E}[Q_t g_t(\bmx)]$ can be upper bounded based on Lemma~\ref{Q_t-bound1}. Then it follows that

\begin{align*}
    &\mathbb{E}[Q_{s+1}]\\
    &\leq \sqrt{2Is+\frac{2FR\tmix^\beta (s+1)^{\beta+1}}{1+\beta}+{2R^2 \tmix s}+{2H(H+GR)(\tau-1)s+\frac{4H}{s}\sum_{t=1}^s \mathbb{E}[Q_t]}}\\
    &\leq \sqrt{2Is}+\sqrt{\frac{2FR\tmix^\beta (s+1)^{\beta+1}}{1+\beta}}+\sqrt{2R^2 \tmix s}+\sqrt{2H^2(\tau-1) s}+2\sqrt{\frac{H}{s}\sum_{t=1}^s\mathbb{E}[Q_t]}\\
    &\leq \sqrt{2Is}+\sqrt{\frac{2FR\tmix^\beta (s+1)^{\beta+1}}{1+\beta}}+\sqrt{2R^2 \tmix s}+\sqrt{2H^2(\tau-1) s}+2H+\frac{1}{2s}\sum_{t=1}^s\mathbb{E}[Q_t]
\end{align*}
for $s\leq T$. By the induction hypothesis, it follows that $\mathbb{E}[Q_t]$ for any $1\leq t\leq s$ is upper bounded by
    \begin{align*} 
    &\frac{3}{2}\sqrt{2I(t-1)}+\frac{\beta+3}{\beta+1}\sqrt{\frac{2FR\tmix^\beta t^{\beta+1}}{1+\beta}}+\frac{3R}{2}\sqrt{2 \tmix (t-1)}+\frac{3H}{2}\sqrt{2(\tau-1) (t-1)}+4H\\
        &\leq \frac{3}{2}\sqrt{2Is}+\frac{\beta+3}{\beta+1}\sqrt{\frac{2FR\tmix^\beta (s+1)^{\beta+1}}{1+\beta}}+\frac{3R}{2}\sqrt{2 \tmix s}+\frac{3H}{2}\sqrt{2(\tau-1) s}+4H.
    \end{align*}
    This leads to the desired upper bound on $\mathbb{E}[Q_{s+1}]$.
\end{proof}We are now ready to prove the second part of Theorem \ref{thm1}, which proves the constraint violation bound of \Cref{alg0}.

\begin{proof}[The second part of Theorem \ref{thm1}]
   Combining Lemmas~\ref{lem:constraint} and~\ref{lem:eq}, we deduce that
    \begin{align*}
        \mathbb{E}\left[\sum_{t=1}^T g_t(\bmc)\right] &\leq \mathbb{E}[Q_{T+1}]+\frac{FG}{2}\sum_{t=1}^T\frac{V_t}{\alpha_t}+\frac{G^2}{2}\sum_{t=1}^T\frac{\mathbb{E}[Q_t]}{\alpha_t}\notag=O\left(\tmix^{\beta/2}T^{\beta/2+1/2}+\sqrt{(\tau-1)T}\right), 
    \end{align*}
    as required.
\end{proof}

\subsection{Ergodic Drift-Plus-Penalty for Stochastic-Constrained Stochastic Optimization}\label{sec:thm2proof}

In this section, we provide our formal proof of \Cref{thm2}. To better organize and present the result, we divide the analysis into two, one of which is for bounding the optimality gap while the other is for bounding the feasibility gap. The main idea is to use the performance bounds given in \Cref{thm1} and relate them to the optimality gap and the feasibility gap. The relation between them is not as clear as in the i.i.d. setting, but we use the mixing property of ergodic Markov chains.
\begin{proof}[The first part of \Cref{thm2}]
For $T\geq\tau=\tmix(T^{-1})$, we have the following
\begin{align*}
    \sum_{t=1}^T\left(\bar{f}(\bmc)-\bar{f}(\bmx)\right)&= \sum_{t=1}^{T-\tau+1}\left(\bar{f}(\bmc)-\bar{f}(\bmx)-f_{t+\tau-1}(\bmc)+f_{t+\tau-1}(\bmx)\right)\\
    &\quad +\sum_{t=1}^{T-\tau+1}\left(f_{t+\tau-1}(\bmc)-f_{t+\tau-1}(\bm{x_{t+\tau-1}})\right)\\
    &\quad +\sum_{t=\tau}^T\left(f_t(\bmc)-f_t(\bmx)\right)+\sum_{t=T-\tau+2}^T\left(\bar{f}(\bmc)-\bar{f}(\bmx)\right).
\end{align*}
We consider the four parts of the right-hand side separately. Here, the first part satisfies
\begin{align*}
    &\mathbb{E}_{t-1}\left[\bar{f}(\bmc)-\bar{f}(\bmx)-f_{t+\tau-1}(\bmc)+f_{t+\tau-1}(\bmx)\right]\\&=\int \left(f(\bmc,\xi)-f(\bmx,\xi)\right) \left(d\mu(\xi)-d\mathbb{P}_{[t-1]}^{t+\tau-1}(\xi)\right)\\
    &\leq FR\int|d\mu(\xi)-d\mathbb{P}_{[t-1]}^{t+\tau-1}(\xi)|\\
    &\leq 2FR\lVert \mu-\mathbb{P}_{[t-1]}^{t+\tau-1}\rVert_{TV}\\
    &\leq\frac{2FR}{T}.
\end{align*}
Hence, it follows that $\sum_{t=1}^{T-\tau+1}\mathbb{E}[\bar{f}(\bmc)-\bar{f}(\bmx)-f_{t+\tau-1}(\bmc)+f_{t+\tau-1}(\bmx)]\leq 2FR$. To bound the second part, we consider the following.
\begin{align*}
    \mathbb{E}[f_{t+\tau-1}(\bmc)-f_{t+\tau-1}(\bm{x_{t+\tau-1}})]
    &=\sum_{s=t}^{t+\tau-2}\mathbb{E}[f_{t+\tau-1}(\bm{x_s})-f_{t+\tau-1}(\bm{x_{s+1}})]\notag\\
    &\leq \sum_{s=t}^{t+\tau-2}\mathbb{E}[F\lVert\bm{x_s}-\bm{x_{s+1}}\rVert]\\
    &\leq \sum_{s=t}^{t+\tau-2}\frac{F}{2\alpha_s}\mathbb{E}[V_s F+Q_s G]\\
    &=\frac{F^2 \tmix^{\beta-1}}{2}\sum_{s=t}^{t+\tau-2}s^{\beta-1}+\frac{FG}{2\tmix}\sum_{s=t}^{t+\tau-2}\frac{\mathbb{E}[Q_s]}{s}\\
    &\leq \frac{F^2\tmix^{\beta-1}}{2\beta}\left((t+\tau-2)^\beta-(t-1)^\beta\right)+\frac{FG}{2\tmix}\sum_{s=t}^{t+\tau-2}\frac{\mathbb{E}[Q_s]}{s}
\end{align*}
where the last inequality holds due to Corollary~\ref{cor3}. Here, we consider the following to provide an upper bound on the second term on the right-most side.
\begin{align*}
  \sum_{s=t}^{t+\tau-2}\frac{\mathbb{E}[Q_s]}{s}
   &\leq 3\left(\sqrt{2I(t+\tau-2)}-\sqrt{2I(t-1)}\right)\\
    &\quad +\frac{2(\beta+3)}{(\beta+1)^2}\sqrt{\frac{2FR\tmix^\beta}{\beta+1}}\left(\sqrt{(t+\tau-2)^{\beta+1}}-\sqrt{(t-1)^{\beta+1}}\right)\\
    &\quad +3R\left(\sqrt{2 \tmix(t+\tau-2)}-\sqrt{2\tmix(t-1)}\right)\\
    &\quad +3H\left(\sqrt{2(\tau-1)(t+\tau-2)}-\sqrt{2(\tau-1)(t-1)}\right)\\
    &\quad +4H\log \left[\frac{t+\tau-2}{\sigma(t-1)}\right].
\end{align*}
Here, we handle the above terms in the following manner. Note that
\begin{align*}
    &\sum_{t=1}^{T-\tau+1}\left(\sqrt{(t+\tau-2)^{\beta+1}}-\sqrt{(t-1)^{\beta+1}}\right)\\
    &=(T-1)^{\frac{\beta+1}{2}}+\cdots+(T-\tau+1)^{\frac{\beta+1}{2}}-(\tau-2)^{\frac{\beta+1}{2}}-\cdots-1\\
    &\leq (\tau-1)T^{\beta/2+1/2}
\end{align*}
holds and that
\begin{align*}
    \sum_{t=1}^{T-\tau+1}\left(\sqrt{(t+\tau-2)}-\sqrt{(t-1)}\right)&=(T-1)^{\frac{1}{2}}+\cdots+(T-\tau+1)^{\frac{1}{2}}-(\tau-2)^{\frac{1}{2}}-\cdots-1\\
    &\leq (\tau-1)T^{1/2}.
\end{align*}
The resulting terms form dominant terms,
which means that \begin{align*}
&\sum_{t=1}^{T-\tau+1}\mathbb{E}[f_{t+\tau-1}(\bmc)-f_{t+\tau-1}(\bm{x_{t+\tau-1}})]\\
&=O\left(\tmix^{\beta/2-1}(\tau-1)T^{\beta/2+1/2}+\tmix^{-1}(\tau-1)^{3/2}T^{1/2}\right).\end{align*}
Moreover, by  \Cref{thm1} and the triangular inequality, the third part is bounded as
$$\mathbb{E}\left[\sum_{t=\tau}^T \left(f_t(\bmc)-f_t(\bmx)\right)\right] =O\left(\tmix^{-\beta}\tau T^{1-\beta}+(\tau-1)\right).$$
The fourth part can be bounded as follows.
$$\mathbb{E}\left[\sum_{t=T-\tau+2}^T \left(\bar{f}(\bmc)-\bar{f}(\bmx)\right)\right]\leq RF(\tau-1).$$
Combining the bounds on the four parts, we can conclude that
for any $\bmx\in\mathcal{X}$ such that $\bar{g}(\bmx)\leq 0$,
\begin{align*}
&\mathbb{E}\left[\sum_{t=1}^T\left(\bar{f}(\bmc)-\bar{f}(\bmx)\right)\right]\\
&= O\left(\tmix^{-\beta}\tau T^{1-\beta}+\tmix^{\beta/2-1}(\tau-1)T^{\beta/2+1/2}+\tmix^{-1}(\tau-1)^{3/2}T^{1/2}+\tau\right),
\end{align*}
which implies
\begin{align*}
    \mathbb{E}\left[\bar{f}({\bm{\bar x_T}})-\bar{f}(\bmx)\right]&\leq\frac{1}{T}\mathbb{E}\left[\sum_{t=1}^T \left(\bar{f}(\bmc)-\bar{f}(\bmx)\right)\right]\\
    &=O\left(\frac{\tau}{\tmix^\beta T^\beta}+\frac{\tau-1}{\tmix}\frac{\tmix^{\beta/2}}{T^{(1-\beta)/2}}+\frac{(\tau-1)^{3/2}}{\tmix T^{1/2}}+\frac{\tau}{T}\right),
\end{align*}
as required.
\end{proof}

Next we prove the feasibility gap bound of \Cref{alg0}, which is given as the second part of \Cref{thm2}.

\begin{proof}[Proof of the second part of Theorem \ref{thm2}]
For $T\geq\tau=\tmix(T^{-1})$, we have the following,
\begin{align*}
    \sum_{t=1}^T \bar{g}(\bmc) &= \sum_{t=1}^{T-\tau+1}\left(\bar{g}(\bmc)-g_{t+\tau-1}(\bmc)\right)+\sum_{t=1}^{T-\tau+1}\left(g_{t+\tau-1}(\bmc)-g_{t+\tau-1}(\bm{x_{t+\tau-1}})\right)\\
    &\quad +\sum_{t=\tau}^T g_t(\bmc)+\sum_{t=T-\tau+2}^T \bar{g}(\bmc)
\end{align*}
where the right-hand side consists of four terms. We separately upper bound the four parts of the right-hand side. As before, the first part satisfies
\begin{align*}
    \mathbb{E}_{t-1}[\bar{g}(\bmc)-g_{t+\tau-1}(\bmc)]&=\int g(\bmc,\xi)\left(d\mu(\xi)-d\mathbb{P}_{[t-1]}^{t+\tau-1}(\xi)\right)\\
    &\leq H\int|d\mu(\xi)-d\mathbb{P}_{[t-1]}^{t+\tau-1}(\xi)|\\
    &\leq 2H\lVert\mu-\mathbb{P}_{[t-1]}^{t+\tau-1}\rVert_{TV}\\
    &\leq\frac{2H}{T}.
\end{align*}
This implies that $\sum_{t=1}^{T-\tau+1}\mathbb{E}[\bar{g}(\bmc)-g_{t+\tau-1}(\bmc)]\leq 2H$. Next, the second part satisfies
\begin{align*}
    \mathbb{E}[g_{t+\tau-1}(\bmc)-g_{t+\tau-1}(\bm{x_{t+\tau-1}})]
    &=\sum_{s=t}^{t+\tau-2}\mathbb{E}[g_{t+\tau-1}(\bm{x_s})-g_{t+\tau-1}(\bm{x_{s+1}})]\notag\\
    &\leq \sum_{s=t}^{t+\tau-2}\mathbb{E}[G\lVert\bm{x_s}-\bm{x_{s+1}}\rVert]
\end{align*}
As in the proof of the first part of \Cref{thm2}, we deduce that
$$\sum_{t=1}^{T-\tau+1}\mathbb{E}[g_{t+\tau-1}(\bmc)-g_{t+\tau-1}(\bmc)]\leq O\left(\tmix^{\beta/2-1}(\tau-1)T^{\beta/2+1/2}+\tmix^{-1}(\tau-1)^{3/2}T^{1/2}\right).$$
Moreover, it follows from \Cref{thm1} and the triangular inequality that the third part can be bounded as
\begin{equation*}%
\mathbb{E}\left[\sum_{t=\tau+2}^T g_t(\bmc)\right]=O\left(\tmix^{\beta/2}T^{\beta/2+1/2}+\sqrt{(\tau-1)T}+\tau-1\right).
\end{equation*}
Lastly, the following provides an upper bound on the fourth part.
$$\mathbb{E}\left[\sum_{t=T-\tau+2}^T \bar{g}(\bmc)\right]\leq H(\tau-1).$$
Combining the bounds on the four parts, we have that
$$\mathbb{E}\left[\sum_{t=1}^T \bar{g}(\bmc)\right]\leq O\left(\tmix^{\beta/2}T^{\beta/2+1/2}+\tmix^{\beta/2-1}(\tau-1)T^{\beta/2+1/2}+\tmix^{-1}(\tau-1)^{3/2}T^{1/2}+\tau\right),$$
implying in turn that
\begin{equation*}
    \mathbb{E}\left[\bar{g}({\bm{\bar x_T}})\right]\leq\frac{1}{T}\mathbb{E}\left[\sum_{t=1}^T \bar{g}(\bmc)\right]= O\left(\frac{\tmix^{\beta/2-1}\tau}{T^{(1-\beta)/2}}+\frac{(\tau-1)^{3/2}}{\tmix T^{1/2}}+\frac{\tau}{T}\right),
\end{equation*}
as required.
\end{proof}

\subsection{Ergodic Drift-Plus-Penalty under Slater's Condition}\label{sec:thm3proof}

In this section, we prove \Cref{thm3}, which analyzes the performance of \Cref{thm3} under Slater's constraint qualification stated in \Cref{ass:slater}. We will see that Slater's condition does lead to improvement.
Basically, we deduce a reduction on $\mathbb{E}[Q_t]$. To argue this, we follow the proof path of \cite{OCO-stochastic}. We first present a lemma, which is analogous to \citep[Lemma 5]{OCO-stochastic}. The difference is that we use a time-varying parameter $\theta(t)$ which allows time-varying algorithm parameters $V_t$ and $\alpha_t$. In fact, its proof is similar to that of \citep[Lemma 5]{OCO-stochastic}, so we defer it to the appendix (\Cref{sec:appendix-drift}). 

\begin{lemma}\label{lem:process}
    Let $\{Z_t, t\geq 0\}$ be a discrete-time stochastic process adapted to a filtration $\{\mathcal{W}_t, t\geq 0\}$ with $Z_0=0$ and $\mathcal{W}_0=\{\emptyset, \mathcal{S}\}$.
    Suppose there exists a positive integer $t_0$, real numbers $0<\zeta\leq\delta_{\max}$, and a non-decreasing function $\theta(t)>0$ such that
    \begin{align*}
        |Z_{t+1}-Z_t| &\leq\delta_{\max},\\
        \mathbb{E}[Z_{t+t_0}-Z_t|\mathcal{W}_t]&\leq
        \begin{cases}
            t_0\delta_{\max}, &\quad \text{if }Z_t<\theta(t)\\
            -t_0\zeta, &\quad \text{if }Z_t\geq\theta(t)
        \end{cases}
    \end{align*}
    for all positive integer $t$. Then,
    $$\mathbb{E}[Z_t]\leq \theta(t)+t_0 \delta_{\max}+t_0\frac{4\delta_{\max}^2}{\zeta}\log\left(\frac{8\delta_{\max}^2}{\zeta^2}\right)$$
    for all positive integer $t$.
\end{lemma}
Here, the important part is that although $\theta(t)$ is time-varying, the parameter $t_0$ is some fixed value that does not depend on $t$. That is why the condition in Lemma~\ref{lem:process} can be recursively applied. 

Lemma \ref{lem:process} implies that if the stochastic process given by $\{Q_t, t\geq 0\}$ where $Q_0=0$ satisfies the drift condition as in Lemma~\ref{lem:process} with appropriate parameters $\delta_{\max}$, $t_0$, and $\zeta$, then the expected queue size $\mathbb{E}\left[Q_t\right]$ can be bounded. Moreover, by properly setting $\theta(t)$ as well as the parameters, we may control the size of $\mathbb{E}\left[Q_t\right]$. The next lemma shows that the stochastic process given by $\{Q_t, t\geq 0\}$ indeed satisfies the desired drift condition.

While proving the lemma, we need to consider the term
$$\mathbb{E}_{t-1}\left[ Q_i g_i(\bm{\hat x})\right]$$ for $i\geq t$
where $\bm{\hat x}$ is the solution satisfying Slater's constraint qualification, i.e. $\bar g(\bm{\hat x})\leq -\epsilon$. Here, we need to relate $\bar g(\bm{\hat x})\leq -\epsilon$ and the term $\mathbb{E}_{t-1}\left[ Q_i g_i(\bm{\hat x})\right]$, from which we deduce the drift condition. Although $g_i(\bm{\hat x})$ is not necessarily negative, we again use the property of ergodic Markov chains that the distribution of $g_i(\bm{\hat x})$ conditional on $\mathcal{F}_{t-1}$ gets close to the stationary distribution for a sufficiently large $i$.

\begin{lemma}\label{slater-drift-ergodic}
    Let $T\geq 4H/\epsilon$, $t_0>2\tau=2\tmix(T^{-1})$, and 
    \begin{align*}
        \theta_{t_0}(t)&= \frac{2}{\epsilon(t_0-2\tau)}\left(\epsilon(t_0-\tau)^2(H+GR)+4\tau(t_0+t-1)(H+GR)^2\right)\\
        &\quad +\frac{2}{\epsilon(t_0-2\tau)}\left(2t_0(H^2+G^2 R^2)+2RF\tmix^\beta t_0(t+t_0-1)^{\beta}\right).
    \end{align*}
    Then \Cref{alg0} under \Cref{ass:slater} satisfies
    \begin{align*}
    |Q_{t+1}-Q_t|&\leq H+GR,\\
    \mathbb{E}_{t-1}[Q_{t+t_0}-Q_t]&\leq 
        \begin{cases}
            (H+GR)t_0, &\quad \text{if }Q_t<\theta_{t_0}(t)\\
            -\epsilon t_0/4, &\quad \text{if }Q_t\geq \theta_{t_0}(t).
        \end{cases}
        \end{align*}
\end{lemma}

\begin{proof}
    $|Q_{t+1}-Q_t|\leq H + GR$ follows directly from Lemma~\ref{lem:trivial}.

    For the second part, note that  by Jensen's inequality, we have $$\mathbb{E}_{t-1}[Q_{t+t_0}]^2\leq \mathbb{E}_{t-1}[Q_{t+t_0}^2].$$
Moreover, observe that
    \begin{align*}
        \theta_{t_0}(t) &\geq \frac{2\epsilon(t_0-\tau)^2(H+GR)}{\epsilon(t_0-2\tau)}\geq (t_0-\tau)(H+GR)\geq \frac{\epsilon t_0}{2},
    \end{align*}
    where the last inequality holds because $t_0>2\tau$ and $H>\epsilon$. This means that $Q_t>\epsilon t_0/4$ if $Q_t>\theta(t_0)$. Therefore, it is sufficient to show that if $Q_t\geq \theta(t_0)$,
    $$\mathbb{E}_{t-1}[Q_{t+t_0}^2]\leq \left(Q_t-\frac{\epsilon t_0}{4}\right)^2.$$
    Recall that By Lemma~\ref{lem:drift2} and the convexity of $g_t$, we have for any $i\geq 1$, $$\Delta_i\leq H^2+G^2 R^2+ Q_i g_i(\bm{\hat x})+V_i RF+\alpha_i\left(D(\bm{\hat x}, \bm{x_i})-D(\bm{\hat x}, \bm{x_{i+1}})\right).$$ 
    Thus, we have 
    \begin{align*}
        &\sum_{i=t}^{t+t_0-1} \mathbb{E}_{t-1}[\Delta_i]\\
        &\leq t_0(H^2+G^2 R^2)+\sum_{i=t}^{t+t_0-1}\mathbb{E}_{t-1}[Q_i g_i(\bm{\hat x})]+RF\tmix^\beta\sum_{i=t}^{t+t_0-1}i^\beta\\
        &\quad +\alpha_t D(\bm{\hat x}, \bmc)+\sum_{i=t}^{t+t_0-2} (\alpha_{i+1}-\alpha_i)D(\bm{\hat x},\bm{x_{i+1}}) - \alpha_{t+t_0-1} D(\bm{\hat x}, x_{t+t_0})\\
        &\leq t_0(H^2+G^2 R^2)+RF\tmix^\beta t_0(t+t_0-1)^{\beta}+\alpha_{t+t_0-1}R^2+\sum_{i=t}^{t+t_0-1}\mathbb{E}_{t-1}[Q_i g_i(\bm{\hat x})].
    \end{align*}
    We factor the last term as follows.
    \begin{align*}
        &\sum_{i=t}^{t+t_0-1} \mathbb{E}_{t-1}[Q_i g_i(\bm{\hat x})] \\
        &= \sum_{i=t}^{t+\tau-2}\mathbb{E}_{t-1}[Q_i g_i(\bm{\hat x})]+\sum_{i=t}^{t+t_0-\tau}\mathbb{E}_{t-1}[(Q_{i+\tau-1}-Q_i)g_{i+\tau-1}(\bm{\hat x})]\\
        &\quad 
        +\sum_{i=t}^{t+t_0-\tau}\mathbb{E}_{t-1}[Q_i(g_{i+\tau-1}(\bm{\hat x})-\bar{g}(\bm{\hat x}))]+\sum_{i=t}^{t+t_0-\tau}\mathbb{E}_{t-1}[Q_i \bar{g}(\bm{\hat x})]\\
        &\leq H^2\sum_{i=t-1}^{t+\tau-3}i+(t_0-\tau+1)(\tau-1) H(H+GR)+\left(\frac{2H}{T}-\epsilon\right)\sum_{i=t}^{t+t_0-\tau}\mathbb{E}_{t-1}[Q_i]
\end{align*}
where the inequality is due to
\begin{align*}
    \mathbb{E}_{t-1}[\bar{g}(\bmc)-g_{t+\tau-1}(\bmc)]&\leq H\int|d\mu(\xi)-d\mathbb{P}_{[t-1]}^{t+\tau-1}(\xi)|\leq 2H\lVert\mu-\mathbb{P}_{[t-1]}^{t+\tau-1}\rVert_{TV}\leq\frac{2H}{T}.
\end{align*}
Then, since $T\geq 4H/\epsilon$ and thus $2H/T\leq \epsilon/2$, it follows that
\begin{align*}
    &\sum_{i=t}^{t+t_0-1} \mathbb{E}_{t-1}[Q_i g_i(\bm{\hat x})]     \\&\leq (t_0+t-2)(\tau-1) H(H+GR)-\frac{\epsilon}{2}\sum_{i=t}^{t+t_0-\tau}\mathbb{E}_{t-1}[Q_i]\\
        &\leq (t_0+t-2)(\tau-1) H(H+GR)-\frac{\epsilon}{2}\sum_{i=0}^{t_0-\tau}\left(Q_t-i(H+GR)\right)\\
        &\leq (t_0+t-2)(\tau-1) H(H+GR)-\frac{\epsilon}{2}\left((t_0-\tau)Q_t-(t_0-\tau)^2(H+GR)\right).
    \end{align*}
As a result, we obtain
    \begin{align*}
        \mathbb{E}_{t-1}[Q_{t+t_0}^2]&=Q_t^2+2\sum_{i=t}^{t+t_0-1}\mathbb{E}_{t-1}[\Delta_i]\\
        &\leq Q_t^2-\epsilon(t_0-\tau)Q_t+\epsilon(t_0-\tau)^2(H+GR)+2(t_0+t-2)(\tau-1) H(H+GR)\\
        &\quad +2t_0(H^2+G^2 R^2)+2RF\tmix^\beta t_0(t+t_0-1)^{\beta}+2\tmix(t+t_0-1)R^2\\
        &\leq Q_t^2-\epsilon(t_0-\tau)Q_t+\epsilon(t_0-\tau)^2(H+GR)+4\tau(t_0+t-1)(H+GR)^2\\
        &\quad +2t_0(H^2+G^2 R^2)+2RF\tmix^\beta t_0(t+t_0-1)^{\beta}
    \end{align*}
    If $Q_t$ satisfies
    \begin{align*}
        \frac{\epsilon}{2}(t_0-2\tau)Q_t&\geq \epsilon(t_0-\tau)^2(H+GR)+4\tau(t_0+t-1)(H+GR)^2\\
        &\quad +2t_0(H^2+G^2 R^2)+2RF\tmix^\beta t_0(t+t_0-1)^{\beta},
    \end{align*}
    which is equivalent to $Q_t\geq \theta_{t_0}(t)$, then 
    $$\mathbb{E}_{t-1}[Q_{t+t_0}^2]\leq Q_t^2-\frac{\epsilon}{2}t_0 Q_t\leq \left(Q_t-\frac{\epsilon t_0}{4}\right)^2$$
    as required.
\end{proof}

Having prepared Lemmas~\ref{lem:process} and~\ref{slater-drift-ergodic}, we can derive refined bounds on the expected virtual queue size. Based on this, we are ready to prove \Cref{thm3}.
\begin{proof}[Proof of Theorem \ref{thm3}]
As the result is trivial when $T\leq 4H/\epsilon$ where $H$ and $\epsilon$ are constants, we may assume without loss of generality that $T\geq 4H/\epsilon$. Moreover, we will use Lemma~\ref{lem:process} for many values of $t_0$ greater than $2\tau$. Setting $\delta_{\max} = H+GR$, $\zeta = \epsilon/4$, and $\theta= \theta_{t_0}$ where $\theta_{t_0}$ is defined as in Lemma~\ref{slater-drift-ergodic}, Lemma~\ref{lem:process} gives us that
$$\mathbb{E}[Q_{t+1}]\leq \theta_{t_0}(t)+t_0(H+GR)+t_0\frac{16(H+GR)^2}{\epsilon}\log \left(\frac{128(H+GR)^2}{\epsilon^2}\right).$$
In particular, as this inequality holds for any $t_0> 2\tau$, the inequality with $t_0 = 2\tau + \lceil\sqrt{\tmix t}\rceil$ holds. Recall that in Lemma~\ref{slater-drift-ergodic}, we set
\begin{align*}
        \theta_{t_0}(t)&= \frac{2}{\epsilon(t_0-2\tau)}\left(\epsilon(t_0-\tau)^2(H+GR)+4\tau(t_0+t-1)(H+GR)^2\right)\\
        &\quad +\frac{2}{\epsilon(t_0-2\tau)}\left(2t_0(H^2+G^2 R^2)+2RF\tmix^\beta t_0(t+t_0-1)^{\beta}\right).
    \end{align*}
Moreover, $$\frac{\epsilon(t_0-2\tau)}{2} \theta_{t_0}(t)=\frac{\epsilon \lceil\sqrt{\tmix t}\rceil}{2}\theta_{t_0}(t)\geq   \frac{\epsilon \sqrt{\tmix t}}{2}\theta_{t_0}(t).$$
This implies the following.
\begin{align*}
    &\frac{\epsilon \sqrt{\tmix t}}{2}\theta_{t_0}(t) \\
    &\leq \epsilon(\tau+\lceil\sqrt{\tmix t}\rceil)^2(H+GR)+4\tau(2\tau+\lceil\sqrt{\tmix t}\rceil+t-1)(H+GR)^2\\
        &\quad +2(2\tau+\lceil\sqrt{\tmix t}\rceil)(H^2+G^2 R^2)+2RF\tmix^\beta (2\tau+\lceil\sqrt{\tmix t}\rceil)(t+2\tau+\lceil\sqrt{\tmix t}\rceil-1)^{\beta}\\
    &\leq 2\epsilon(H+GR)(\tau^2+4\tmix t)+4\tau(2\tau+2\sqrt{\tmix t}+t)(H+GR)^2\\
    &\quad +4(\tau+\sqrt{\tmix t})(H^2+G^2 R^2)+4RF\tmix^\beta(\tau+\sqrt{\tmix t})(t+2\tau+2\sqrt{\tmix t})^\beta\\
    &\leq \underbrace{2\epsilon(H+GR)(\tau^2+4\tmix t)}_{(a)}+\underbrace{4\tau(2\tau+\tmix+2t)(H+GR)^2}_{(b)}\\
    &\quad +\underbrace{4(\tau+\sqrt{\tmix t})(H^2+G^2 R^2)}_{(c)}+\underbrace{4RF\tmix^\beta(\tau+\sqrt{\tmix t})(2t+2\tau+\tmix)^\beta}_{(d)}
\end{align*}
where the second inequality follows from $(a+b)^2\leq 2(a^2+b^2), \lceil x \rceil\leq 2x$ which holds for $x\geq 1$ and the third inequality follows from AM-GM inequality.
We see that term $(c)$ grows more slowly than term $(d)$ which is of order $$O\left(\tmix^\beta (\tau+t)^\beta(\tau+\sqrt{\tmix t})\right)=O\left(\tmix^\beta (\tau^\beta+t^\beta)(\tau+\sqrt{\tmix t})\right).$$
Here, we used the known fact that $(a+b)^\beta \leq a^\beta + b^\beta$ for $a,b\geq0$ and $\beta<1$.
Likewise, term $(a)$ grows more slowly than term $(b)$ which is of order $$O\left(\tau(\tau+t)\right).$$
Then it follows that
\begin{align*}
    \theta_{t_0}(t)&= O\left(\tmix^\beta(\tau^\beta+t^\beta)\left(1+\frac{\tau}{\sqrt{\tmix t}}\right)+\frac{\tau(\tau+t)}{\sqrt{\tmix t}}\right)\\
    &=O\left(\underbrace{\tmix^\beta \tau^\beta}_{(a')}+\underbrace{\tmix^\beta t^\beta}_{(b')} +\underbrace{\frac{\tau^{1+\beta}\tmix^{\beta}}{\sqrt{\tmix t}}}_{(c')}+\underbrace{\frac{\tau \tmix^\beta t^\beta}{\sqrt{\tmix t}}}_{(d')}+\underbrace{\frac{\tau^2}{\sqrt{\tmix t}}}_{(e')}+\underbrace{\frac{\tau t}{\sqrt{\tmix t}}}_{(f')}\right).
\end{align*}
Here, term $(c')$ grows more slowly than term $(e')$, and term $(f')$ dominates term $(b')$. Also, terms $(a')$ and $(d')$ are dominated by $$\tau=O\left(\tau\sqrt{\frac{\tau}{\tmix}}\right)=O\left( (e')+(f')\right),$$ which is due to the AM-GM inequality. Therefore,
$$\theta_{t_0}(t)=O\left(\frac{\tau(\tau+t)}{\sqrt{\tmix t}}\right).$$
Furthermore, since $$t_0=2\tau+\lceil \sqrt{\tmix t}\rceil= O\left(\frac{\tau(\tau+t)}{\sqrt{\tmix t}}\right)$$ which holds because $\tau(\tau +t) \geq 2 \tau \sqrt{\tau t}$, we have
\begin{align}
    \mathbb{E}[Q_t]&\leq O\left(\frac{\tau(\tau+t)}{\sqrt{\tmix t}}\right),\label{ineq:ergodic-q}
\end{align}
which essentially gives rise to the desired refined bound on the expected virtual queue size. 

Having provided the bound on the expected queue size, we now proceed to provide bounds on the performance guarantees of \Cref{alg0}.
First of all, note that the regret bound given by \Cref{thm1} still holds with $\beta =1/2$, as we use the same algorithm parameters. Then we deduce that
\begin{align*}
    \mathbb{E}\left[\operatorname{Regret}(T)\right]=O\left(\frac{\tau \sqrt{T}}{\sqrt{\tmix}}\right).
\end{align*}
For constraint violation, we deduce from Lemma \ref{lem:constraint} and~\eqref{ineq:ergodic-q} that
\begin{align}
\begin{aligned}
\mathbb{E}\left[\operatorname{Violation}(T)\right]&\leq \mathbb{E}[Q_{T+1}]+\frac{FG}{2}\sum_{t=1}^T\frac{V_t}{\alpha_t}+\frac{G^2}{2}\sum_{t=1}^T\frac{\mathbb{E}[Q_t]}{\alpha_t}\notag\\
    &=O\left(\frac{\tau(\tau+T+1)}{\sqrt{\tmix (T+1)}}+{\tmix^{\beta-1}}\sum_{t=1}^T t^{\beta-1}+\frac{1}{\tmix}\sum_{t=1}^T \frac{\tau(\tau+t)}{t\sqrt{\tmix t}}\right)\\
    &=O\left(\frac{\tau(\tau+T)}{\sqrt{\tmix T}}\right).
    \end{aligned}
\end{align}
Next, we consider the optimality gap. In the proof of \Cref{thm2}, we argued that
\begin{align*}
    \mathbb{E}[f_{t+\tau-1}(\bmc)-f_{t+\tau-1}(\bm{x_{t+\tau-1}})]
    &\leq \frac{F^2\tmix^{\beta-1}}{2\beta}\left((t+\tau-2)^\beta-(t-1)^\beta\right)+\frac{FG}{2\tmix}\sum_{s=t}^{t+\tau-2}\frac{\mathbb{E}[Q_s]}{s}.
\end{align*}
Here, using~\eqref{ineq:ergodic-q} again, the second term on the right-hand side can be bounded as follows.
\begin{align*}
&\frac{1}{\tmix}\sum_{s=t}^{t+\tau-2}\frac{\mathbb{E}[Q_s]}{s}\\
&=O\left(\frac{\tau^2}{\tmix^{3/2}}\left(\underbrace{\lambda((t-1)^{-1/2})-\lambda((t+\tau-2)^{-1/2})}_{(a'')}\right)+\frac{\tau}{\tmix^{3/2}}(\sqrt{t+\tau-2}-\sqrt{t-1})\right)
\end{align*}
where $$\lambda(x)=\begin{cases}
    x, &\text{ if }x>0\\
    3/2, &\text{ if }x=0
\end{cases}.$$ 
Therefore, it follows that
\begin{align}\label{tau2}
\begin{aligned}
    \sum_{t=1}^{T-\tau+1}\mathbb{E}[f_{t+\tau-1}(\bmc)-f_{t+\tau-1}(\bm{x_{t+\tau-1}})]&=O\left(\frac{\tmix^{\beta-1}}{2\beta}\left((T-1)^\beta+\cdots+(T-\tau+1)^\beta\right)\right.\\
    &\quad +\frac{\tau^2}{\tmix^{3/2}}\left(\frac{3}{2}+1^{-1/2}+\cdots+(\tau-2)^{-1/2}\right)\\
    &\quad\left.+\frac{\tau}{\tmix^{3/2}}\left(\sqrt{T-1}+\cdots+\sqrt{T-\tau+1}\right)\right)
    \end{aligned}
\end{align}
if $\tau\geq 2$.
Thus, we deduce
\begin{align*}
&\sum_{t=1}^{T-\tau+1}\mathbb{E}[f_{t+\tau-1}(\bmc)-f_{t+\tau-1}(\bm{x_{t+\tau-1}})]\\
&=O\left(\tmix^{\beta-1}(\tau-1)T^\beta+\tau^2\sqrt{\tau-1}\tmix^{-3/2}+\tau(\tau-1)\tmix^{-3/2}\sqrt{T}\right)\\
&=O\left(\frac{\tau^{5/2}}{\tmix^{3/2}}+\frac{\tau^{2}\sqrt{T}}{\tmix^{3/2}}\right).
\end{align*}
We combine this result with the other parts of the proof of the first part of \Cref{thm2} to get
\begin{align*}
    \mathbb{E}\left[\sum_{t=1}^T \left(\bar{f}(\bmc)-\bar{f}(\bmx)\right)\right]= O\left(\frac{\tau^{5/2}}{\tmix^{3/2}}+\frac{\tau^{2}\sqrt{T}}{\tmix^{3/2}}+\tmix^{-\beta}\tau T^{1-\beta}+\tau\right)=O\left(\frac{\tau^{5/2}}{\tmix^{3/2}}+\frac{\tau^{2}\sqrt{T}}{\tmix^{3/2}}\right)
\end{align*}
for any $\bmx\in\mathcal{X}$ such that $\bar{g}(\bmx)\leq 0$ since $\beta = 1/2$. Similarly, we get
\begin{align*}
\sum_{t=1}^{T-\tau+1}\mathbb{E}[g_{t+\tau-1}(\bmc)-g_{t+\tau-1}(\bm{x_{t+\tau-1}})]=O\left(\frac{\tau^{5/2}}{\tmix^{3/2}}+\frac{\tau^{2}\sqrt{T}}{\tmix^{3/2}}\right).
\end{align*}
As before, we combine this result with the other parts of the proof of the second part of \Cref{thm2}. Then we get
\begin{align*}
    \mathbb{E}\left[\sum_{t=1}^T \bar{g}(\bmc)\right]= O\left(\frac{\tau^{5/2}}{\tmix^{3/2}}+\frac{\tau^{2}\sqrt{T}}{\tmix^{3/2}}+\frac{\tau(\tau+T)}{\sqrt{\tmix T}}+\tau\right)=O\left(\frac{\tau^{5/2}}{\tmix^{3/2}}+\frac{\tau^{2}\sqrt{T}}{\tmix^{3/2}}+\frac{\tau^2}{\sqrt{\tmix T}}\right).
\end{align*}
Finally, we obtain that
\begin{align*}
    \mathbb{E}[\operatorname{Gap}(\bm{\bar x_T})]&=O\left(\frac{\tau^{5/2}}{\tmix^{3/2}T}+\frac{\tau^{2}}{\tmix^{3/2}\sqrt{T}}\right),\\
    \mathbb{E}[\operatorname{Infeasibility}(\bm{\bar x_T})]&=\left(\frac{\tau^{5/2}}{\tmix^{3/2}T}+\frac{\tau^{2}}{\tmix^{3/2}\sqrt{T}} + \frac{\tau^2}{\sqrt{\tmix} T^{3/2}}\right),
\end{align*}
as required.
\end{proof}

\section{Analysis of MLMC Adaptive Drift-Plus-Penalty for the Unknown Mixing Time Case}\label{sec:mdpp-proof}

This section presents a complete proof of \Cref{thm5}. \Cref{sec:mlmc-queue} provides bounds on several terms that involve the virtual queue size. In \Cref{sec:thm5proof}, we state the proof of Lemma~\ref{lem:st1}, and lastly, we prove \Cref{thm5}.

\subsection{Controlling the Expected Virtual Queue Size}\label{sec:mlmc-queue}

In this section we provide upper bounds on terms $\mathbb{E}\left[Q_t/V_t\right]$ and $\mathbb{E}\left[Q_t\right]$ under the MLMC Adaptive Drift-Plus-Penalty algorithm (\Cref{alg1}). Lemma~\ref{lem:newbound} gives a refined bound on the term $\mathbb{E}\left[g_t(\bmx)\right]$. Using this, Lemma~\ref{lem:qv} derives an adaptive upper bound on the term $\mathbb{E}\left[Q_t/V_t\right]$, and Lemma~\ref{lem:q-bound-mdpp} deduces an adaptive upper bound on the term $\mathbb{E}\left[Q_t\right]$.

Note that if $\bar x\in\mathcal{X}$ satisfies $\bar g(\bmx)\leq 0$, then $g_t(\bmx)\leq g_t(\bmx)- \bar g(\bmx)$. Then we deduce
\begin{align*}
    \mathbb{E}_{t-1}[g_t(\bmx)]\leq \mathbb{E}_{t-1}[g_t(\bmx)-\bar{g}(\bmx)]
    =\mathbb{E}_{t-1}\left[g_t^{2^{j_{\max}}}(\bmx)-\bar{g}(\bmx)\right]
    \leq \mathbb{E}_{t-1}\left[|g_t^{2^{j_{\max}}}(\bmx)-\bar{g}(\bmx)|\right].
\end{align*}
where the equality is due to Lemma~\ref{lem:MLMC}. Applying Lemma~\ref{lem:concentration} together with Jensen's inequality on the right-hand side, we obtain
\begin{equation}\label{ineq:gt}
    \mathbb{E}_{t-1}[g_t(\bmx)]\leq C(T)\tmix^{1/2}T^{-1}=\tilde O\left(\tmix^{1/2}T^{-1}\right).
\end{equation}
In fact, we may derive a tighter bound on the term $\mathbb{E}_{t-1}[g_t(\bmx)]$ as follows.
\begin{lemma}\label{lem:newbound}
    For any $\bmx\in\mathcal{X}$ such that $\bar{g}(\bmx)\leq 0$,
    $$\mathbb{E}_{t-1}[g_t(\bmx)]\leq D(T)\tmix^{1/4}T^{-1/2},$$
    where $D(T)=\sqrt{C(T)H(4\log_2 T+3)}$
\end{lemma}
\begin{proof} Note that  $g_t(\bmx)\leq |g_t(\bmx)|\leq (2N_t+1)H$ by the definition of $g_t$. This implies that
 \begin{equation*}
     \mathbb{E}_{t-1}[g_t(\bmx)]\leq\mathbb{E}_{t-1}[(2N_t+1)H]=(4\log_2 T+3)H,
 \end{equation*}
 where the equality follows from Lemma \ref{lem:MLMC} and the fact that $N_t$ is independent of $\mathcal{F}_{t-1}$. Moreover, as we argued above, we have
 \begin{align*}
 \begin{aligned}
     \mathbb{E}_{t-1}[g_t(\bmx)]&\leq \mathbb{E}_{t-1}\left[|g_t^{2^{j_{\max}}}(\bmx)-\bar{g}(\bmx)|\right]\leq \sqrt{\mathbb{E}_{t-1}\left[|g_t^{2^{j_{\max}}}(\bmx)-\bar{g}(\bmx)|^2\right]}\leq C(T)\tmix^{1/2}T^{-1},
 \end{aligned}
 \end{align*}
 where the last inequality follows from Jensen's inequality and the equality follows from Lemma \ref{lem:concentration}.
By taking the geometric mean of these two upper bounds, we get that
$$\mathbb{E}_{t-1}[g_t(\bmx)]\leq \sqrt{C(T)H(4\log_2 T+3)}\tmix^{1/4}T^{-1/2},$$
as required.
\end{proof}

Next, Lemma \ref{lem:q} implies that $Q_{s+1}$ is at most
\begin{align*}
 \underbrace{\sqrt{2\sum_{t=1}^s RV_t F_t}}_{(a)}+\underbrace{\sqrt{2\sum_{t=1}^s (R^2G_t^2+H_t^2)}}_{(b)}+\underbrace{\sqrt{2R^2\alpha_s}}_{(c)}+\sqrt{2\sum_{t=1}^s Q_t(g_t(\bmc)+\grad g_t(\bmc)^\top(\bmx-\bmc))}.
\end{align*}
Here, term $(c)$ is equal to $\sqrt{2S_{s-1}}$ which is at most $\sqrt{2S_{s}}$, and term $(b)$ is at most $\sqrt{2S_{s}}$. For term $(a)$, we have
$$(a)^2= 2\sum_{t=1}^s S_{t-1}^\beta F_t\leq2 S_{s}^\beta\sum_{t=1}^s F_t\leq 2 S_{s}^\beta\sqrt{s\sum_{t=1}^s F_t^2}\leq 2\sqrt{s}S_{s}^{\beta+1/2},$$
where the second inequality holds by the power mean inequality.
Thus,
\begin{equation}\label{ineq:q-square}
    Q_{s+1}^2\leq 2\sqrt{s}S_s^{\beta+1/2}+4S_s+2\sum_{t=1}^s Q_t(g_t(\bmx_t)+\grad g_t(\bmx_t)^\top(\bmx-\bmx_t)),
\end{equation}
which in turn implies that
\begin{align}\label{ineq:q2}
\begin{aligned}
    Q_{s+1}&\leq2\sqrt{2}S_{s}^{1/2}+\sqrt{2}s^{1/4}S_{s}^{\beta/2+1/4}+\sqrt{2\sum_{t=1}^s Q_t\left(g_t(\bmc)+\grad g_t(\bmc)^\top(\bmx-\bmc)\right)}
    \end{aligned}
\end{align}
Moreover, \eqref{ineq:q2} implies that
$$\frac{Q_{s+1}}{V_{s+1}}\leq2\sqrt{2}RS_{s}^{1/2-\beta}+\sqrt{2}Rs^{1/4}S_{s}^{1/4-\beta/2}+\sqrt{\frac{2}{V_{s+1}}\sum_{t=1}^s \frac{Q_t}{V_t}\left(g_t(\bmc)+\grad g_t(\bmc)^\top(\bmx-\bmc)\right)}$$
as $V_t$ is non-decreasing. Taking the expectation, applying Jensen's inequality, and using the fact that $V_{s+1}\geq s\delta$, it follows that
\begin{align*}
    \mathbb{E}\left[\frac{Q_{s+1}}{V_{s+1}}\right]&\leq2\sqrt{2}R\mathbb{E}[S_{s}]^{1/2-\beta}+\sqrt{2}Rs^{1/4}\mathbb{E}[S_{s}]^{1/4-\beta/2}\\
    &\quad +\sqrt{\frac{2R}{(s\delta)^\beta}\sum_{t=1}^s \mathbb{E}\left[\frac{Q_t}{V_t}\mathbb{E}_{t-1}[g_t(\bmc)+\grad g_t(\bmc)^\top(\bmx-\bmc)]\right]}.
\end{align*}
Here, the last term on the right-hand side can be further upper bounded as follows.
\begin{align*}
    \sqrt{\frac{2R}{(s\delta)^\beta}\sum_{t=1}^s \mathbb{E}\left[\frac{Q_t}{V_t}\mathbb{E}_{t-1}[g_t(\bmc)+\grad g_t(\bmc)^\top(\bmx-\bmc)]\right]}\leq \sqrt{\frac{2R}{(s\delta)^\beta}\sum_{t=1}^s \mathbb{E}\left[\frac{Q_t}{V_t}\mathbb{E}_{t-1}[ g_t(\bmx)]\right]}
\end{align*}
where we used the fact that
\begin{align}\label{eq:convexity}
\begin{aligned}
\mathbb{E}_{t-1}[g_t(\bmc)+\grad g_t(\bmc)^\top(\bmx-\bmc)]&=\mathbb{E}_{t-1}\left[g^{j_{\max}}(\bmc)+\grad g^{j_{\max}}(\bmc)^\top(\bmx-\bmc)\right]\\
&\leq\mathbb{E}_{t-1}\left[g^{j_{\max}}(\bmx)\right]\\
&=\mathbb{E}_{t-1}\left[g_t(\bmx)\right]
\end{aligned}
\end{align}
holds due to Lemma~\ref{lem:MLMC}. Furthermore,
\begin{align*}
    \sqrt{\frac{2R}{(s\delta)^\beta}\sum_{t=1}^s \mathbb{E}\left[\frac{Q_t}{V_t}\mathbb{E}_{t-1}[ g_t(\bmx)]\right]} &\leq\sqrt{\frac{2R}{(s\delta)^\beta}D(T)\tmix^{1/4}T^{-1/2}\sum_{t=1}^s \mathbb{E}\left[\frac{Q_t}{V_t}\right]}\\
    &\leq \frac{2Rs}{(s\delta)^\beta}D(T)\tmix^{1/4}T^{-1/2}+\frac{1}{2s}\sum_{t=1}^s \mathbb{E}\left[\frac{Q_t}{V_t}\right]\\
    &\leq\frac{2RD(T)}{\delta^\beta}\tmix^{1/4}T^{1/2-\beta}+\frac{1}{2s}\sum_{t=1}^s \mathbb{E}\left[\frac{Q_t}{V_t}\right].
\end{align*}
where the second inequality follows from the AM-GM inequality and the last equality follows from $s\leq T$. Thus, we have
\begin{align}\label{ineq:qv}
\begin{aligned}
&\mathbb{E}\left[\frac{Q_{s+1}}{V_{s+1}}\right]\\
&\leq2\sqrt{2}R\mathbb{E}[S_{s}]^{1/2-\beta}+\sqrt{2}Rs^{1/4}\mathbb{E}[S_{s}]^{1/4-\beta/2}+\frac{2RD(T)}{\delta^\beta}\tmix^{1/4}T^{1/2-\beta}+\frac{1}{2s}\sum_{t=1}^s \mathbb{E}\left[\frac{Q_t}{V_t}\right]\\
&\leq 2\sqrt{2}R\mathbb{E}[S_{T}]^{1/2-\beta}+\sqrt{2}Rs^{1/4}\mathbb{E}[S_{T}]^{1/4-\beta/2}+\frac{2RD(T)}{\delta^\beta}\tmix^{1/4}T^{1/2-\beta}+\frac{1}{2s}\sum_{t=1}^s \mathbb{E}\left[\frac{Q_t}{V_t}\right].
\end{aligned}
\end{align}
Now we can bound $\mathbb{E}\left[{Q_t}/{V_t}\right]$ as follow.
\begin{lemma}\label{lem:qv}
    For $t\in [T+1]$, 
    $$\mathbb{E}\left[\frac{Q_t}{V_t}\right]\leq 4\sqrt{2}R\mathbb{E}[S_{T}]^{1/2-\beta}+2\sqrt{2}RT^{1/4}\mathbb{E}[S_{T}]^{1/4-\beta/2}+\frac{4RD(T)}{\delta^\beta}\tmix^{1/4}T^{1/2-\beta}.$$
\end{lemma}
\begin{proof}
    For $t=1$, $Q_1=0$ and the inequality of the lemma holds. Assume that it holds for $t=1,\ldots, s$ with  $s\leq T$. Substituting the inequalities for $t=1,\ldots,s$ into \eqref{ineq:qv}, we derive the inequality for $t=s+1$, as required.
\end{proof}

Next, we deduce an upper bound on $\mathbb{E}\left[Q_t\right]$. First, we observe that \eqref{ineq:q2} and Jensen's inequality imply the following. For $s\in [T]$,
\begin{align}\label{ineq:q-bound-mdpp}
\begin{aligned}
    \mathbb{E}\left[Q_{s+1}\right]&\leq 2\sqrt{2}\mathbb{E}[S_{s}]^{1/2}+\sqrt{2}s^{1/4}\mathbb{E}[S_{s}]^{\beta/2+1/4}+\sqrt{2C(T)\frac{\sqrt{\tmix}}{T}\sum_{t=1}^s \mathbb{E}[Q_t]}\\
    &\leq 2\sqrt{2}\mathbb{E}[S_{s}]^{1/2}+\sqrt{2}s^{1/4}\mathbb{E}[S_{s}]^{\beta/2+1/4}+sC(T)\frac{\sqrt{\tmix}}{T}+\frac{1}{2s}\sum_{t=1}^s \mathbb{E}[Q_t]\\
    &\leq 2\sqrt{2}\mathbb{E}[S_{s}]^{1/2}+\sqrt{2}s^{1/4}\mathbb{E}[S_{s}]^{\beta/2+1/4}+C(T)\tmix^{1/2}+\frac{1}{2s}\sum_{t=1}^s \mathbb{E}[Q_t],
\end{aligned}
\end{align}
where the first inequality is implied by combining~\eqref{ineq:q2} and~\eqref{eq:convexity} and applying Lemma~\ref{lem:newbound}. Next we provide an adaptive upper bound on $\mathbb{E}\left[Q_t\right]$

\begin{lemma}\label{lem:q-bound-mdpp}
    For any $t\in [T+1]$,
    $$\mathbb{E}[Q_t]\leq 4\sqrt{2}\mathbb{E}[S_{t-1}]^{1/2}+\frac{5\sqrt{2}}{3}(t-1)^{1/4}\mathbb{E}[S_{t-1}]^{\beta/2+1/4}+2C(T)\tmix^{1/2}.$$
\end{lemma}

\begin{proof}
We argue by induction. The inequality trivially holds when $t=1$ as $Q_1=0$.  Suppose that the inequality holds for $t=1,2,\ldots,s$. Then
    \begin{align*}
        \mathbb{E}[Q_{s+1}]
        &\leq 2\sqrt{2}\mathbb{E}[S_{s}]^{1/2}+\sqrt{2}s^{1/4}\mathbb{E}[S_{s}]^{\beta/2+1/4}+C(T)\tmix^{1/2}\\
        &\quad +\frac{1}{2s}\sum_{t=1}^s \left( 4\sqrt{2}\mathbb{E}[S_{t-1}]^{1/2}+\frac{10}{3}(t-1)^{1/4}\mathbb{E}[S_{t-1}]^{\beta/2+1/4}+2C(T)\tmix^{1/2}\right)\\
        &\leq 2\sqrt{2}\mathbb{E}[S_s]^{1/2}+\sqrt{2}s^{1/4}\mathbb{E}[S_s]^{\beta/2+1/4}+C(T)\tmix^{1/2}\\
        &\quad +\frac{1}{2s}\left( 4\sqrt{2}s\mathbb{E}[S_{s}]^{1/2}+\frac{4\sqrt{2}}{3}s^{5/4}\mathbb{E}[S_{s}]^{\beta/2+1/4}+2sC(T)\tmix^{1/2}\right)\\
        &= 4\sqrt{2}\mathbb{E}[S_{s}]^{1/2}+\frac{5\sqrt{2}}{3}s^{1/4}\mathbb{E}[S_{s}]^{\beta/2+1/4}+2C(T)\tmix^{1/2},
    \end{align*}
    where the first inequality is from~\eqref{ineq:q-bound-mdpp} and the second inequality is by Corollary~\ref{cor3}.
\end{proof}

\subsection{MLMC Adaptive Drift-Plus-Penalty for Stochastic-Constrained Stochastic Optimization}\label{sec:thm5proof}

Based on the bounds on $\mathbb{E}_{t-1}[g_t(\bmx)]$, $\mathbb{E}[Q_t/V_t]$, and $\mathbb{E}[Q_t]$, in this section, we prove Lemma~\ref{lem:st1} and \Cref{thm5}. To better present the proof of Lemma~\ref{lem:st1}, we divide the analysis into two parts, one for the regret and the other for the constraint violation. Then, using Lemma~\ref{lem:st1}, we prove \Cref{thm5}.

\begin{proof}[Proof of the first part of lemma \ref{lem:st1}]
We first bound 
$$\mathbb{E}\left[\sum_{t=1}^T f_t(\bmc)-\sum_{t=1}^T f_t(\bmx)\right].$$
By Lemma \ref{lem:regret}, this is less than or equal to
$$\mathbb{E}\left[\frac{\alpha_T}{V_T}R^2\right]+\mathbb{E}\left[\sum_{t=1}^T\frac{V_t F_t^2}{4\alpha_t}\right]+\mathbb{E}\left[\frac{1}{2}\sum_{t=1}^T\frac{(H_t+G_t R)^2}{V_t}\right]+\mathbb{E}\left[\sum_{t=1}^T\frac{Q_t}{V_t}g_t(\bmx)\right].$$
Here, the first term is equal to $RS_{T-1}^{1-\beta}$ which is less or equal to $RS_T^{1-\beta}$, and
the second term satisfies
$$\mathbb{E}\left[\sum_{t=1}^T\frac{V_t F_t^2}{4\alpha_t}\right]=\frac{R}{4}\mathbb{E}\left[\sum_{t=1}^T S_{t-1}^{\beta-1}F_t^2\right]\leq R\left[\sum_{t=1}^T S_{t-1}^{\beta-1}a_t\right]=O\left(\mathbb{E}\left[S_T^{\beta}\right]\right)=O\left(\mathbb{E}\left[S_T^{\beta}\right]\right)$$
where the inequality holds by Corollary~\ref{cor4}.
The third term satisfies
$$\mathbb{E}\left[\sum_{t=1}^T\frac{(H_t+G_t R)^2}{2V_t}\right]\leq
\mathbb{E}\left[R\sum_{t=1}^T\frac{R^2G_t^2+H_t^2}{S_{t-1}^\beta}\right]\leq
\mathbb{E}\left[R\sum_{t=1}^T\frac{a_t}{S_{t-1}^\beta}\right]= O\left(\mathbb{E}\left[S_T^{1-\beta}\right]\right).$$
By Jensen's inequality, we have $\mathbb{E}\left[S_T^{1-\beta}\right]\leq \mathbb{E}\left[S_T\right]^{1-\beta}$ and $\mathbb{E}\left[S_T^{\beta}\right]\leq \mathbb{E}\left[S_T\right]^{\beta}$.
The fourth term can be bounded as follows. Note that
\begin{align*}
    \mathbb{E}\left[\sum_{t=1}^T \frac{Q_t}{V_t}g_t(\bmx)\right]&=\sum_{t=1}^T \mathbb{E}\left[\frac{Q_t}{V_t}\mathbb{E}_{t-1}[g_t(\bmx)]\right].
\end{align*}
We have two upper bounds on $\mathbb{E}_{t-1}\left[g_t(\bmx)\right]$, one given in \eqref{ineq:gt} and the other one due to Lemma \ref{lem:newbound}.
Let $m(T)$ be the minimum of the two upper bounds, i.e.,
$$m(T) = \min\left\{C(T)\tmix^{1/2} T^{-1},\ D(T) \tmix^{1/4} T^{-1/2}\right\}.$$
Then it follows that
\begin{align*}
\begin{aligned}
    &\sum_{t=1}^T \mathbb{E}\left[\frac{Q_t}{V_t}\mathbb{E}_{t-1}[g_t(\bmx)]\right]\\
    &\leq \sum_{t=1}^T m(T)\left(4\sqrt{2}R\mathbb{E}[S_{T}]^{1/2-\beta}+2\sqrt{2}RT^{1/4}\mathbb{E}[S_{T}]^{1/4-\beta/2}+\frac{4RD(T)}{\delta^\beta}\tmix^{1/4}T^{1/2-\beta}\right)\\
    &\leq \sum_{t=1}^T \left(C(T)\tmix^{1/2}T^{-1}\left(4\sqrt{2}R\mathbb{E}[S_{T}]^{1/2-\beta}+2\sqrt{2}RT^{1/4}\mathbb{E}[S_{T}]^{1/4-\beta/2}\right)+\frac{4RD(T)^2}{\delta^\beta}\tmix^{1/2}T^{-\beta}\right)\\
    &=\tilde{O}\left(\tmix^{1/2}\mathbb{E}[S_T]^{1/2-\beta}+\tmix^{1/2}T^{1/4}\mathbb{E}[S_T]^{1/4-\beta/2}+\tmix^{1/2}T^{1-\beta}\right)
    \end{aligned}
\end{align*}
where the first inequality follows from \eqref{ineq:gt} and Lemmas~\ref{lem:newbound} and~\ref{lem:qv}. Combining the bounds on the four terms, we have proved the desired bound on $\mathbb{E}\left[\sum_{t=1}^T f_t(\bmc)-\sum_{t=1}^T f_t(\bmx)\right].$
\end{proof}

\begin{proof}[Proof of the second part of Lemma \ref{lem:st1}]
By Lemma~\ref{lem:constraint}, we have $$\sum_{t=1}^T g_t(\bmc)\leq Q_{T+1}+\sum_{t=1}^T\frac{G_t}{2\alpha_t}(V_t F_t+Q_t G_t).$$ 
Moreover, by Lemma~\ref{lem:q-bound-mdpp}, we have $$\mathbb{E}[Q_{T+1}]\leq 
\tilde{O}\left(\mathbb{E}[S_T]^{1/2}+T^{1/4}\mathbb{E}[S_T]^{\beta/2+1/4}+\tmix^{1/2}\right).$$
Next it follows from Corollary~\ref{cor4} that 
$$\mathbb{E}\left[\sum_{t=1}^T \frac{V_t F_t G_t}{2\alpha_t}\right]\leq \mathbb{E}\left[\sum_{t=1}^T \frac{S_{t-1}^{\beta-1}a_t}{2}\right]=O\left(\mathbb{E}\left[S_T^\beta\right]\right)=O\left(\mathbb{E}\left[S_T\right]^\beta\right).$$ 
Furthermore,
$$\mathbb{E}\left[ \sum_{t=1}^T \frac{G_t^2}{2\alpha_t}Q_t\right]\leq \mathbb{E}\left[\max_{t\in[T]}Q_t\cdot\sum_{t=1}^T \frac{R^2G_t^2}{2S_{t-1}}\right]\leq \mathbb{E}\left[\left(\frac{I+\delta}{2\delta}+\frac{1}{2}\log\frac{S_T-I-\delta}{\delta}\right)\max_{t\in[T]}Q_t\right]$$
where the second inequality is implied by Lemma \ref{lem:sum3}.
Here,
the last term can be upper bounded using the AM-GM inequality as follows.
\begin{align*}
&\left(\frac{I+\delta}{2\delta}+\frac{1}{2}\log\frac{S_T-I-\delta}{\delta}\right)\max_{t\in[T]}Q_t\\
&\leq \frac{\max_{t\in [T]}Q_t^2}{2(I\tmix)^{\beta/2+1/4}T^{\beta/2+1/2}}+\frac{1}{2}\left(\frac{I+\delta}{2\delta}+\frac{1}{2}\log\frac{S_T}{\delta}\right)^2(I\tmix)^{\beta/2+1/4}T^{\beta/2+1/2}\label{ineq:delta2}.
\end{align*}
Let $s=\argmax_{t\in[T]} Q_t$. Then by~\eqref{ineq:q-square} and~\eqref{eq:convexity},
\begin{align*}
    \mathbb{E}\left[Q_s^2\right]&\leq4\mathbb{E}[S_{s-1}]+ 2\sqrt{s}\mathbb{E}[S_{s-1}]^{\beta+1/2}+2\sum_{t=1}^{s-1} \mathbb{E}\left[Q_t \mathbb{E}_{t-1}\left[g_t(\bmc)+\grad g_t(\bmc)^\top (\bmx-\bmc)\right]\right]\\
    &\leq 4\mathbb{E}[S_{T}]+ 2\sqrt{T}\mathbb{E}[S_{T}]^{\beta+1/2}+2\sum_{t=1}^{s-1} \mathbb{E}[Q_t \mathbb{E}_{t-1}[g_t(\bmx)]].
\end{align*}
Moreover, 
\begin{align*}
    &\sum_{t=1}^{s-1} \mathbb{E}[Q_t \mathbb{E}_{t-1}[g_t(\bmx)]]\\&\leq C(T)\tmix^{1/2}T^{-1}\sum_{t=1}^{s-1} \mathbb{E}[Q_t]\\
    &\leq C(T)\tmix^{1/2}T^{-1}\sum_{t=1}^T\left(4\sqrt{2}\mathbb{E}[S_{T}]^{1/2}+\frac{5\sqrt{2}}{3}(t-1)^{1/4}\mathbb{E}[S_{T}]^{\beta/2+1/4}+2C(T)\tmix^{1/2}\right)\\
    &\leq C(T)\tmix^{1/2}T^{-1}\left(4\sqrt{2}T\mathbb{E}[S_{T}]^{1/2}+\frac{4\sqrt{2}}{3}T^{5/4}\mathbb{E}[S_T]^{\beta/2+1/4}+2C(T)\tmix^{1/2}T\right)\\
    &=\tilde{O}\left(\tmix^{1/2}\mathbb{E}[S_T]^{1/2}+\tmix^{1/2}T^{1/4}\mathbb{E}[S_T]^{\beta/2+1/4}+\tmix\right),
\end{align*}
where the first inequality follows from \eqref{ineq:gt}, the second inequality follows from Lemma \ref{lem:q-bound-mdpp}, and the third inequality follows from Corollary \ref{cor3}.
Thus, we get
 \begin{align*}
        \mathbb{E}\left[\sum_{t=1}^T g_t(\bmc)\right]
        &=\tilde{O}\left(\mathbb{E}[S_T]^{1/2}+T^{1/4}\mathbb{E}[S_T]^{\beta/2+1/4}+\tmix^{1/2}+\frac{\mathbb{E}[S_T]+T^{1/2}\mathbb{E}[S_T]^{\beta+1/2}}{\tmix^{\beta/2+1/4}T^{\beta/2+1/2}}\right.\\
        &\qquad\quad \left.+\frac{\tmix^{1/2}+\mathbb{E}[S_T]^{1/2}+\mathbb{E}[S_T]^{\beta/2+1/4}}{\tmix^{\beta/2+1/4}T^{\beta/2+1/2}}+(\log S_T)^2\tmix^{\beta/2-1/4}T^{\beta/2+1/2}\right),
    \end{align*}
    as required.
\end{proof}

\begin{proof}[Proof of Theorem \ref{thm5}]
As $\bar f$ is convex,
$$\bar{f}(\bar{\bmx}_T)-\bar{f}(\bmx^\#)\leq \frac{1}{T}\sum_{t=1}^T \left(\bar{f}(\bmc)-\bar{f}(\bmx^\#)\right). $$
We can decompose the right-hand side as follows
\begin{equation*}
    \frac{\bar{f}(\bmc)-\bar{f}(\bmx^\#)}{T}= \frac{\bar{f}(\bmc)-f_t(\bmc)}{T} + \frac{f_t(\bmc)-f_t(\bmx^\#)}{T} + \frac{f_t(\bmx^\#)-\bar{f}(\bmx^\#)}{T}.
\end{equation*}
Here, Jensen's inequality and Lemma \ref{lem:concentration} imply that
\begin{align*}
    \frac{1}{T}\mathbb{E}\left[|\bar{f}(\bmc)-f_t(\bmc)|\right]\leq\frac{1}{T}\sqrt{\mathbb{E}\left[|\bar{f}(\bmc)-f_t(\bmc)|^2\right]}=\tilde{O}(\tmix^{1/2}T^{-2}),\\
    \frac{1}{T}\mathbb{E}\left[|\bar{f}(\bmx^\#)-f_t(\bmx^\#)|\right]\leq\frac{1}{T}\sqrt{\mathbb{E}\left[|\bar{f}(\bmx^\#)-f_t(\bmx^\#)|^2\right]}=\tilde{O}(\tmix^{1/2}T^{-2}).
\end{align*}
Moreover, we have derived Proposition \ref{prop} which provides an upper bound on the expectation of $\sum_{t=1}^T f_t(\bmc)-\sum_{t=1}^T f_t(\bmx^\#)$.
Consequently,
$$\mathbb{E}\left[\bar{f}(\bar{\bmx}_T)-\bar{f}(\bmx^\#)\right]\leq \tilde{O}\left((\tmix)^{1-\beta}T^{-\beta}\right). $$

For the second part, Jensen's inequality and Lemma \ref{lem:concentration} imply that
\begin{align*}
    \frac{1}{T}\mathbb{E}\left[|\bar{g}(\bmc)-g_t(\bmc)|\right]\leq\frac{1}{T}\sqrt{\mathbb{E}\left[|\bar{g}(\bmc)-g_t(\bmc)|^2\right]}=\tilde{O}(\tmix^{1/2}T^{-2})
\end{align*}
As a result, 
\begin{align*}
    \mathbb{E}[\bar{g}(\bar{\bmx}_T)]&\leq\frac{1}{T}\sum_{t=1}^T \mathbb{E}[\bar{g}(\bmc)]\\&=\frac{1}{T}\sum_{t=1}^T\mathbb{E}[\bar{g}(\bmc)-g_t(\bmc)]+\frac{1}{T}\sum_{t=1}^T\mathbb{E}[g_t(\bmc)]\\
    &=\tilde{O}\left(\tmix^{1/2}T^{-2}+\tmix^{\beta/2+1/4}T^{\beta/2-1/2}+\tmix^{3/4-\beta/2}T^{-\beta/2-1/2}\right)\\        
    &=\tilde{O}\left(\tmix^{\beta/2+1/4}T^{\beta/2-1/2}+\tmix^{3/4-\beta/2}T^{-\beta/2-1/2}\right),
\end{align*}
where the first inequality holds because $\bar{g}$ is convex and the second equality follows from Proposition \ref{prop}.
\end{proof}

\acks{This research is supported, in part, by KAIST Starting Fund (KAIST-G04220016), FOUR Brain Korea 21 Program (NRF-5199990113928), and National Research Foundation of Korea (NRF-2022M3J6A1063021).}

\newpage

\appendix

\section{Analysis of Adaptive Drift-Plus-Penalty}\label{sec:adpp}

\Cref{alg:dpp} is a general template for adaptive variants of the drift-plus-penalty algorithm whose parameters $V_t$ and $\alpha_t$ satisfy that $\{V_t\}_{t=1}^T$, $\{\alpha_t\}_{t=1}^T$, and $\{\alpha_t/V_t\}_{t=1}^T$ are non-decreasing sequences of non-negative numbers. In this section, we analyze the general template of DPP given by \Cref{alg:dpp}, based on which we deduce performance guarantees on \Cref{alg0,alg1}.
\begin{algorithm}[tb]
	\caption{Adaptive Drift-Plus-Penalty}
	\label{alg:dpp}
	\begin{algorithmic}
		\STATE {\bfseries Initialize:} Initial iterates $\bm{x_1}\in\mathcal{X}$, $Q_1=0$.
		\FOR{$t=1$ {\bfseries to} $T$}
		\STATE Observe $f_t$ and $g_t$.
		\STATE Set penalty parameter $V_t$ and step size parameter $\alpha_t$ such that $0\leq V_{t-1}\leq V_{t}$ , $0\leq \alpha_{t-1}\leq \alpha_{t}$, and $0\leq \alpha_{t-1}/V_{t-1}\leq \alpha_{t}/V_{t}$.
		 \STATE {\bfseries Primal update:} Set $\bmn$ as
		$$\bmn = \argmin_{\bmx\in\mathcal{X}}\left\{\left(V_t\grad f_t(\bmc)+Q_t\grad g_t(\bmc)\right)^\top \bmx + \alpha_t D(\bmx,\bmc)\right\}$$
		\STATE {\bfseries Dual update:} Set $Q_{t+1}$ as
		$$Q_{t+1}=\left[Q_t + g_t(\bmc) +\grad g_t(\bmc)^\top(\bmn-\bmc) \right]_+$$
		\ENDFOR
	\end{algorithmic}
\end{algorithm}

Recall that  $\Delta_t=({Q_{t+1}^2}-{Q_t^2})/{2}$ is the {Lyapunov drift} term. The following lemma provides a bound on the drift term. 

\begin{lemma}\label{lemma:drift-bound}
	For $t\geq 1$,  $$\Delta_t\leq Q_t\left(g_t(\bmc)+\grad g_t(\bmc)^\top(\bmn-\bmc)\right)+\frac{1}{2}(H_t+G_tR)^2.$$
\end{lemma}
\begin{proof}
	As $Q_{t+1}=\max\left\{Q_t + g_t(\bmc) +\grad g_t(\bmc)^\top(\bmn-\bmc),0\right\}$, we have that $Q_{t+1}^2\leq \big( Q_t+g_t(\bmc)+\grad g_t(\bmc)^\top(\bmn-\bmc)\big)^2$. Hence, it follows that
	\begin{align*}
		\Delta_t&=\frac{Q_{t+1}^2}{2}-\frac{Q_t^2}{2}\\&\leq Q_t\left(g_t(\bmc)+\grad g_t(\bmc)^\top(\bmn-\bmc)\right)
		+\frac{1}{2}\big(g_t(\bmc)+\grad g_t(\bmc)^\top(\bmn-\bmc)\big)^2\\
		&\leq Q_t\left(g_t(\bmc)+\grad g_t(\bmc)^\top(\bmn-\bmc)\right)
		+\frac{1}{2}(H_t+G_t R)^2
	\end{align*}
	where the last inequality is from \Cref{ass:lip}.
\end{proof}

Since our drift-plus-penalty algorithm is a mirror descent version, we need the following lemma, which is obtained by substituting 
$\bmy=\bmc$, $\bmx^*=\bmn$, and $f(\bmx)=\left(V_t\grad f_t(\bmc)+Q_t\grad g_t(\bmc)\right)^\top \bmx$ into \cite[Lemma 2.1]{dpp-md}.
\begin{lemma}\citep[Equation (22)]{dpp-md}\label{lem:bregman}
	For any $\bmx\in\mathcal{X}$ and $t\geq 1$, 
	\begin{align*}
		&\left(V_t\grad f_t(\bmc)+Q_t\grad g_t(\bmc)\right)^\top (\bmn-\bmc) + \alpha_t D(\bmn,\bmc)\\
		&\leq \left(V_t\grad f_t(\bmc)+Q_t\grad g_t(\bmc)\right)^\top (\bmx-\bmc) + \alpha_t D(\bmx,\bmc)-\alpha_t D(\bmx,\bmn).
	\end{align*}
\end{lemma}

Recall that $f_t$ and $g_t$ for the known mixing time case in \Cref{sec:tmix} correspond to one sample and are assumed to be convex. In contrast, $f_t$ and $g_t$ for the unknown mixing time setting in \Cref{sec:unknown} come from the MLMC estimation scheme with multiple data samples and thus are not necessarily convex. Nevertheless, we use the fact that $\mathbb{E}_{t-1}[f_t]$ and  $\mathbb{E}_{t-1}[g_t]$ are convex. Based on this, we deduce the following lemma.

\begin{lemma}\label{lem:regret}
	Suppose that $\mathbb{E}_{t-1}[f_t]=\mathbb{E}_{t-1}[\hat f_t]$ and $\mathbb{E}_{t-1}[g_t]=\mathbb{E}_{t-1}[\hat g_t]$ where $\hat f_t$ and $\hat g_t$ are convex functions and that $\mathbb{E}_{t-1}[\grad f_t]=\mathbb{E}_{t-1}[\grad \hat f_t]$ and $\mathbb{E}_{t-1}[\grad g_t]=\mathbb{E}_{t-1}[\grad \hat g_t]$. Then for any $\bmx\in\mathcal{X}$,
	\begin{align*}\mathbb{E}\left[\sum_{t=1}^T f_t(\bmc)-\sum_{t=1}^T f_t(\bmx)\right]&\leq
	\mathbb{E}\left[\sum_{t=1}^T\frac{Q_t}{V_{t}}g_t(\bmx)\right]+\mathbb{E}\left[\frac{\alpha_T}{V_T}R^2\right]\\
	&\quad +\mathbb{E}\left[\sum_{t=1}^T\frac{V_t F_t^2}{4\alpha_t}\right]+\frac{1}{2}\mathbb{E}\left[\sum_{t=1}^T\frac{(H_t+G_t R)^2}{V_t}\right].
	\end{align*}
\end{lemma}
\begin{proof}
	Dividing both sides of the inequality given in \Cref{lem:bregman} by $V_t$ and addding $f_t(\bmc)+{Q_t}g_t(\bmc)/ V_t$ to both sides, we get
	\begin{align*}
		&f_t(\bmc)+\frac{Q_t}{V_t}g_t(\bmc)+\left(\grad f_t(\bmc)+\frac{Q_t}{V_t}\grad g_t(\bmc)\right)^\top (\bmn-\bmc) + \frac{\alpha_t}{V_t} D(\bmn,\bmc)\notag\\
		&\leq f_t(\bmc)+\frac{Q_t}{V_t}g_t(\bmc)+\left(\grad f_t(\bmc)+\frac{Q_t}{V_t}\grad g_t(\bmc)\right)^\top (\bmx-\bmc) + \frac{\alpha_t}{V_t} D(\bmx,\bmc)-\frac{\alpha_t}{V_t} D(\bmx,\bmn).
	\end{align*}
Here, the left-hand side is bounded below by
	$$f_t(\bmc)+\grad f_t(\bmc)^\top (\bmn-\bmc)+\frac{\alpha_t}{V_t} D(\bmn,\bmc)+\frac{\Delta_t}{V_t}-\frac{1}{2V_t}(H_t+G_t R)^2$$
	by Lemma \ref{lemma:drift-bound}. Moreover, we have	
	\begin{align*}
		\grad f_t(\bmc)^\top(\bmn-\bmc)+\frac{\alpha_t}{V_t} D(\bmn,\bmc)&\geq -F_t\|\bmn-\bmc\|+\frac{\alpha_t}{V_t}\|\bmn-\bmc\|^2\\
		&=-\frac{V_t F_t^2}{4\alpha_t}+\frac{\alpha_t}{V_t}\left( \|\bmn-\bmc\|-\frac{V_t F_t}{2\alpha_t}\right)^2\\
		&\geq -\frac{V_t F_t^2}{4\alpha_t}
	\end{align*}
	where the first inequality holds by $2$-strong convexity of $\Phi$ with respect to $\| \cdot\|$. Hence, we deduce that
		\begin{align*}
		&f_t(\bmc)-\frac{V_t F_t^2}{4\alpha_t}+\frac{\Delta_t}{V_t}-\frac{1}{2V_t}(H_t+G_t R)^2\\
		&\leq f_t(\bmc)+\frac{Q_t}{V_t}g_t(\bmc)+\left(\grad f_t(\bmc)+\frac{Q_t}{V_t}\grad g_t(\bmc)\right)^\top (\bmx-\bmc) + \frac{\alpha_t}{V_t} D(\bmx,\bmc)-\frac{\alpha_t}{V_t} D(\bmx,\bmn).
	\end{align*}
	For the right-hand side of this inequality, consider
	\begin{align*}
		\sum_{t=1}^T\frac{\alpha_t}{V_t}\left(D(\bmx,\bmc)-D(\bmx,\bmn)\right)&= \frac{\alpha_1}{V_1} D(\bmx,\bmc)+\sum_{t=2}^T D(\bmx,\bmc)\left(\frac{\alpha_t}{V_t}-\frac{\alpha_{t-1}}{V_{t-1}}\right)-\frac{\alpha_T}{V_T}D(\bmx, \bmn)\\
		&\leq  \frac{\alpha_1}{V_1}R^2+\sum_{t=2}^T R^2\left(\frac{\alpha_t}{V_t}-\frac{\alpha_{t-1}}{V_{t-1}}\right)\\
		&=\frac{\alpha_T}{V_T}R^2
	\end{align*}
	where the inequality holds since $\left\{{\alpha_t}/{V_t}\right\}_{t=1}^T$ is a non-decreasing sequence. Furthermore,
	\begin{equation*}
		\sum_{t=1}^T \frac{\Delta_t}{V_t}=\frac{1}{2}\sum_{t=1}^T \frac{1}{V_t}(Q_{t+1}^2-Q_t^2)=-\frac{Q_1^2}{2V_1}+\frac{Q_{T+1}^2}{2V_T}+\frac{1}{2}\sum_{t=2}^T Q_t^2\left(\frac{1}{V_{t-1}}-\frac{1}{V_t}\right)\geq -\frac{Q_1^2}{2V_1}=0,
	\end{equation*}
	where the inequality holds since the sequence $\{V_t\}_{t=1}^T$ is non-negative and non-decreasing. Combining these inequalities, we deduce that
	\begin{align}\label{lem:regret-1}
 \begin{aligned}
		&\sum_{t=1}^T f_t(\bmc)-\sum_{t=1}^T \left(f_t(\bmc)+\grad f_t(\bmc)^\top(\bmx-\bmc)\right)\\
		&\leq
		\sum_{t=1}^T\frac{Q_t}{V_t}\left(g_t(\bmc)+\grad g_t(\bmc)^\top(\bmx-\bmc)\right)+\frac{\alpha_T}{V_T}R^2+\sum_{t=1}^T\frac{V_t F_t^2}{4\alpha_t}+\frac{1}{2}\sum_{t=1}^T\frac{(H_t+G_t R)^2}{V_t}.
  \end{aligned}
	\end{align}
    Recall that $\mathbb{E}_{t-1}[f_t]=\mathbb{E}_{t-1}[\hat f_t]$ and $\mathbb{E}_{t-1}[g_t]=\mathbb{E}_{t-1}[\hat g_t]$ where $\hat f_t$ and $\hat g_t$ are convex and that $\mathbb{E}_{t-1}[\grad f_t]=\mathbb{E}_{t-1}[\grad \hat f_t]$ and $\mathbb{E}_{t-1}[\grad g_t]=\mathbb{E}_{t-1}[\grad \hat g_t]$. Then it follows that
    \begin{align*}
        \mathbb{E}_{t-1}[f_t(\bmc)]+\mathbb{E}_{t-1}[\grad f_t(\bmc)]^\top(\bmx-\bmc)&=\mathbb{E}_{t-1}[\hat f_t(\bmc)]+\mathbb{E}_{t-1}[\grad \hat f_t(\bmc)]^\top(\bmx-\bmc)\\
        &\leq \mathbb{E}_{t-1}[\hat f_t(\bmx)]\\
        &= \mathbb{E}_{t-1}[f_t(\bmx)]
    \end{align*}
    where the second inequality holds because $\hat f_t$ is convex and $\bmc$ is $\mathcal{F}_{t-1}$-measurable. Likewise, we deduce that
    $$\mathbb{E}_{t-1}[g_t(\bmc)]+\mathbb{E}_{t-1}[\grad g_t(\bmc)]^\top(\bmx-\bmc)\leq \mathbb{E}_{t-1}[g_t(\bmx)].$$
    Taking the expectations of both sides of~\eqref{lem:regret-1}, we obtain the inequality of this lemma, as required.
\end{proof}

Next, we state a lemma that will be useful to provide an upper bound on the constraint violation.
\begin{lemma}\label{lem:constraint}
\cref{alg:dpp} achieves
	$$\|\bmn-\bmc\|\leq\frac{1}{2\alpha_t}(V_t F_t+Q_t G_t),\quad \sum_{t=1}^T g_t(\bmc)\leq Q_{T+1}+\sum_{t=1}^T\frac{G_t}{2\alpha_t}(V_t F_t+Q_t G_t).$$
\end{lemma}
\begin{proof}
As $Q_{t+1}\geq Q_t+g_t(\bmc)+\grad g_t(\bmc)^\top(\bmn-\bmc)$, it follows that
	$$g_t(\bmc)\leq Q_{t+1}-Q_t-\grad g_t(\bmc)^\top(\bmn-\bmc)\leq Q_{t+1}-Q_t+G_t\|\bmn-\bmc\|.$$
	On the other hand, if we set $\bmx=\bmc$ for the inequaity of Lemma~\ref{lem:bregman}, we get
	\begin{equation*}
		\alpha_t D(\bmn,\bmc)+\alpha_t D(\bmc, \bmn)\leq \big(V_t\grad f_t(\bmc)+Q_t\grad g_t(\bmc)\big)^\top(\bmc-\bmn).
	\end{equation*}
	Here, the left-hand side is greater or equal to $2\alpha_t \|\bmn-\bmc\|^2$ while the right-hand side is less or equal to $(V_t F_t+Q_t G_t)\|\bmn-\bmc\|$. Therefore, it follows that
	\begin{equation}\label{ineq:distance}
		\|\bmn-\bmc\|\leq\frac{1}{2\alpha_t}(V_t F_t+Q_t G_t),
	\end{equation}
	which implies
	\begin{equation*}
		\sum_{t=1}^T g_t(\bmc)\leq Q_{T+1}+\sum_{t=1}^T\frac{G_t}{2\alpha_t}(V_t F_t+Q_t G_t),
	\end{equation*}
 as required.
\end{proof}

To bound the constraint violation, we still need to bound the virtual queue size $Q_{T+1}$. We also need the following lemma.
\begin{lemma}\label{lem:drift2}
	For any $\bmx\in\mathcal{X}$,
	$$\Delta_t\leq H_t^2+R^2G_t^2+Q_t \left(g_t(\bmc)+\grad g_t(\bmc)^\top (\bmx-\bmc)\right)+V_t RF_t+\alpha_t(D(\bmx,\bmc)-D(\bmx,\bmn).$$
\end{lemma}

\begin{proof}
	By Lemma \ref{lemma:drift-bound},
	\begin{align*}
		\Delta_t
		&\leq Q_t\left(g_t(\bmc)+\grad g_t(\bmc)^\top (\bmn-\bmc)\right)  + \frac{1}{2}(H_t+R G_t)^2\\
		&\leq Q_t\left(g_t(\bmc)+\grad g_t(\bmc)^\top (\bmn-\bmc)\right)  + H_t^2+R^2 G_t^2
	\end{align*}
	where the last inequality comes from the fact that $(A+B)^2\leq 2(A^2+B^2)$. Here, using Lemma~\ref{lem:bregman}, the right-hand side can be further bounded above as follows.
	\begin{align*}
		&Q_t(g_t(\bmc)+\grad g_t(\bmc)^\top(\bmn-\bmc))+H_t^2+R^2 G_t^2 \\
  &\leq 
		Q_t\big(g_t(\bmc)+\grad g_t(\bmc)^\top(\bmx-\bmc)\big)+V_t\grad f_t(\bmc)^\top(\bmx-\bmc)
		-V_t\grad f_t(\bmc)^\top(\bmn-\bmc)\\
		&\quad -\alpha_t D(\bmn,\bmc)
		+\alpha_t\big(D(\bmx,\bmc)-D(\bmx,\bmn)\big) +H_t^2+R^2 G_t^2.
	\end{align*}
Moreover, it follows from the Cauchy-Schwarz inequality that
	$$V_t\grad f_t(\bmc)^\top(\bmx-\bmc)-V_t\grad f_t(\bmc)^\top(\bmn-\bmc)=V_t\grad f_t(\bmc)^\top (\bmx-\bmn)\leq V_tRF_t$$ 
	by Cauchy-Schwarz inequality. Then we have proved the lemma, as desired.
\end{proof}

Based on Lemma~\ref{lem:drift2}, we may provide the following bound on the virtual queue size.
\begin{lemma}\label{lem:q}
	For any $\bmx\in\mathcal{X}, s\in [T]$,
	$$Q_{s+1}\leq \sqrt{2\sum_{t=1}^s\big(H_t^2+R^2 G_t^2+V_t RF_t\big)+ 2\sum_{t=1}^s Q_t \left(g_t(\bmc)+\grad g_t(\bmc)^\top (\bmx-\bmc)\right) +2R^2\alpha_s}.$$
\end{lemma}
\begin{proof}
	From Lemma \ref{lem:drift2}, we get
	\begin{align*}
		\frac{Q_{s+1}^2}{2}&=\sum_{t=1}^s\Delta_t\\
  &\leq  \sum_{t=1}^s \big(H_t^2+R^2 G_t^2+Q_t \left(g_t(\bmc)+\grad g_t(\bmc)^\top (\bmx-\bmc)\right)+V_t RF_t\big)\\
	&\quad 	+\alpha_1 D(\bmx,\bm{x_1})+\sum_{t=2}^s D(\bmx,\bmc)(\alpha_t-\alpha_{t-1})-\alpha_s D(\bmx, \bmn).
	\end{align*}
	Since $\alpha_t$ is non-decreasing and $D(\bmx, \bmc)\leq R^2$, it follows that
 $$\alpha_1 D(\bmx,\bm{x_1})+\sum_{t=2}^s D(\bmx,\bmc)(\alpha_t-\alpha_{t-1})-\alpha_s D(\bmx, \bmn)\leq R^2\alpha_s.$$ This implies the desired bound on $Q_{s+1}$.
\end{proof}

\section{Sum of Sequences}\label{sec:sums}

In this section, we consider some series of numbers and provide bounds on their partial sums to make our paper self-contained. Given a sequence $\{x_t\}_{t=1}^\infty$ of numbers, we use notation $X_s:=\sum_{t=1}^s x_t$ to denote its partial sums.

\begin{lemma}\label{lem:sum1}
    If $f:\mathbb{R}_{+}\to\mathbb{R}_{+}$ is continuous and non-increasing, then
    $$\sum_{t=1}^T f(X_t) x_t \leq x_1 f(X_1)+\int_{X_1}^{X_T} f(x) dx\leq \int_{0}^{X_T} f(x)dx$$
    for any nonnegative $x_1,\ldots,x_T$.
\end{lemma}
\begin{proof}
    By considering the area between $f(x)$ and the $x$-axis over the interval $[X_{t-1}, X_t]$ of length $x_t$, we have $f(X_t)x_t\leq \int_{X_{t-1}}^{X_t} f(x) dx$ since $f$ is non-increasing.
    Then it follows that $$\sum_{t=1}^T f(X_t) x_t= x_1 f(X_1)+\sum_{t=2}^T f(X_t) x_t\leq x_1 f(X_1)+\int_{X_1}^{X_T} f(x) dx\leq \int_{0}^{X_T} f(x)dx,$$
    as required.
\end{proof}
As a consequence of Lemma~\ref{lem:sum1}, we deduce the following list of bounds on series.
\begin{corollary}\citep{AUER200248}\label{cor1}
    $$\sum_{t=1}^T \frac{x_t}{\sqrt{X_t}}\leq 2\sqrt{X_T}.$$
\end{corollary}
\begin{corollary}\label{cor2}
    $$\sum_{t=1}^T \frac{x_t}{X_t}\leq 1+\log_{X_1}(X_T).$$
\end{corollary}
Next we consider the following.
\begin{lemma}\label{lem:sum2}
    If $f:\mathbb{R}_{+}\to\mathbb{R}_{+}$ is continuous and non-decreasing, then
    $$\sum_{t=1}^T f(X_t) x_{t+1} \leq \int_{X_1}^{X_{T+1}} f(x)dx$$
    for any nonnegative $x_1,\ldots,x_T$.
\end{lemma}
\begin{proof}
    By considering the area between $f(x)$ and the $x$-axis over the interval $[X_t, X_{t+1}]$ of length $x_{t+1}$, we have $f(X_t)x_{t+1}\leq \int_{X_t}^{X_{t+1}} f(x) dx$ since $f$ is non-decreasing. Then it follows that 
    $$\sum_{t=1}^T f(X_t) x_{t+1}= x_2f(X_1)+\sum_{t=2}^T f(X_t) x_{t+1}\leq x_2 f(X_1)+\int_{X_2}^{X_{T+1}} f(x) dx\leq \int_{X_1}^{X_{T+1}} f(x)dx,$$
    as required.
\end{proof}
Lemma~\ref{lem:sum2} implies the following.
\begin{corollary}\label{cor3}
    For $q>0$,
    $$\sum_{t=1}^T t^q\leq \frac{1}{q+1}\big( (T+1)^{q+1}-1\big).$$
\end{corollary}
Moreover, when each $x_t$ is bounded by some fixed constant, we can deduce the following result. For ease of notation, we start a sequence with $x_0$ and, with abuse of notation, define partial sum $X_s=\sum_{t=0}^s x_t$ with $x_0=X_0=\delta$.

\begin{lemma}\label{lem:sum3}
 If $0\leq x_t\leq C$ for $t=0,\ldots, T$ for some fixed constant $C$ and  $f:\mathbb{R}_{++}\to\mathbb{R}_{+}$ is continuous and non-increasing, then
    $$\sum_{t=1}^T f(X_{t-1})x_{t}\leq
        Cf(\delta)+\int_{\delta}^{\max\{\delta, X_T-C\}}f(x)dx.$$
\end{lemma}
\begin{proof}
    If we consider the area between $f(x)$ and the $x$-axis over the interval $[X_{t-1}, X_t]$, we have $f(X_{t-1})x_t\geq \int_{X_{t-1}}^{X_t}f(x)dx$ in which the inequality direction is the opposite of what we want. However, since $x_t\leq C$, we can use the idea of translation by $C$ in the $x$-axis direction in the following way. Let
    $$\tilde{f}(x)=\begin{cases}
        f(\delta), &\quad x\in (-\infty,\delta],\\
        f(x), &\quad x\in (\delta,\infty).
    \end{cases}$$
    Then the graph of the translation $\tilde{f}(x-C)$ is above the squares of height $f(X_{t-1})$ on the interval $[X_{t-1}, X_t]$. Thus, 
    $$\sum_{t=1}^T f(X_{t-1})x_{t}\leq\int_{X_0}^{X_T}\tilde{f}(x-C)dx\leq\begin{cases}
        (X_T-\delta)f(\delta)\leq Cf(\delta), &\quad \text{if }X_T\leq \delta+C,\\
        Cf(\delta)+\int_{\delta}^{X_T-C}f(x)dx, &\quad \text{if }X_T> \delta+C
    \end{cases}$$
   which implies the desired statement of this lemma.
\end{proof}
As a corollary of Lemma~\ref{lem:sum3} with $f(x) = x^{-\gamma}$, we deduce the following inequality.
\begin{corollary}\label{cor4}
    If $x_t\leq C, 0<\gamma\neq 1$, then
    $$\sum_{t=1}^T X_{t-1}^{-\gamma}x_t\leq C\delta^{-\gamma}+\frac{1}{1-\gamma}\left(\max\{\delta, X_T-C\}^{1-\gamma}-\delta^{1-\gamma}\right).$$
\end{corollary}

\section{Online Convex Optimization with Adversarial Losses and Constraints}\label{sec:oco-adversarial}

In this section, we show the performance of \Cref{alg1}, which is an AdaGrad-style variant of the drift-plus-penalty algorithm for online convex optimization with adversarial loss and constraint functions. To deal with adversarial constraint functions, we set the parameters $V_t$ and $\alpha_t$ differently as follows.
\begin{align*}
    a_t:=\frac{F_t^2}{4}+R^2 G_t^2+H_t^2, \quad S_t:=\sum_{s=1}^t a_s
\end{align*}
and then set parameters as
\begin{equation}\label{eq:param2}
    V_t=\frac{S_{t}^\beta}{R},\quad \alpha_t=\frac{S_{t}}{R^2},
\end{equation}
for some $0<\beta\leq 1/2$. One distinction from~\eqref{sec:unknown} is the presence of additional parameter $\delta$, and another difference is that $V_t$ and $\alpha_t$ are defined with $S_t$, not $S_{t-1}$.
We now assume that the convex constraint functions $g_1, \ldots, g_T$ as well as the convex loss functions $f_1, \ldots, f_T$ are chosen adversarially. 
Following~\cite{yu-neely}, we set the benchmark $\bm{x^\circ}$ as an optimal solution to
    $$\min\quad \sum_{t=1}^T f_t(\bmx)\quad\text{subject to}\quad g_t(\bmx)\leq 0\quad \text{for $t=1,\ldots,T$}.$$
Then the goal is to obtain upper bounds on
\begin{align*}
    \operatorname{Regret}(T)=\sum_{t=1}^T f_t(\bmc)-\sum_{t=1}^T f_t(\bm{x^\circ}),\quad
    \operatorname{Violation}(T)=\sum_{t=1}^T g_t(\bmc)
\end{align*}
in sublinear orders of $T$ by properly choosing our inputs $\bmx_t$. For this, we need the following theorem.
\begin{theorem}\label{thm:main}
    Algorithm~\ref{alg1} with $V_t$ and $\alpha$ set as in~\eqref{eq:param2} guarantees that 
    \begin{align*}
        \operatorname{Regret}(T)&= O\left(S_T^{1-\beta}\right),\quad \operatorname{Violation}(T)= O\left(S_T^{1/2}+T^{1/4}S_T^{\beta/2+1/4}\right).
    \end{align*}
\end{theorem}
\begin{proof}
  Applying Lemma~\ref{lem:regret} with $\hat{f}_t=f_t$, $\hat{g}_t=g_t$, and $\bmx=\bm{x^\circ}$, we obtain
    $$\operatorname{Regret}(T)\leq
    \frac{\alpha_T}{V_T}R^2+\sum_{t=1}^T\frac{V_t F_t^2}{4\alpha_t}+\frac{1}{2}\sum_{t=1}^T\frac{(H_t+G_t R)^2}{V_t}.$$
    The first term of the right hand side is $RS_{T}^{1-\beta}=O(S_T^{1-\beta})$, and
    the second term satisfies
    $$\sum_{t=1}^T\frac{V_t F_t^2}{4\alpha_t}=\frac{R}{4}\sum_{t=1}^T S_{t}^{\beta-1}F_t^2\leq R\sum_{t=1}^T S_{t}^{\beta-1}a_t\leq \frac{R}{\beta}S_T^\beta= O\left(S_T^{\beta}\right),$$
    where the last inequality follows from Lemma \ref{lem:sum1}.
    The third term satisfies
    $$\sum_{t=1}^T\frac{(H_t+G_t R)^2}{2V_t}\leq
    R\sum_{t=1}^T\frac{R^2G_t^2+H_t^2}{S_{t}^\beta}\leq
    R\sum_{t=1}^T\frac{a_t}{S_{t}^\beta}\leq \frac{R}{1-\beta}S_T^{1-\beta}=O\left(S_T^{1-\beta}\right).$$
    Combining these two inequalities, we get
    \begin{equation*}
    \operatorname{Regret} (T)= O\left(S_T^{1-\beta}\right).
    \end{equation*}
    Next, we prove the second part of the theorem.
    By Lemma \ref{lem:constraint}, we have
    $$\sum_{t=1}^T g_t(\bmx_t)\leq Q_{T+1}+\sum_{t=1}^T\frac{G_t}{2\alpha_t}(V_t F_t+Q_t G_t).$$
    If we apply Lemma \ref{lem:q} with $\bmx=\bm{x^\circ}$, we obtain
    \begin{align*}
        Q_{T+1}&\leq \sqrt{2\sum_{t=1}^T\big(H_t^2+R^2 G_t^2+V_t RF_t\big)+2R^2\alpha_T}\\
        &\leq \underbrace{\sqrt{2\sum_{t=1}^T RV_t F_t}}_{(a)}+\underbrace{\sqrt{2\sum_{t=1}^T (R^2G_t^2+H_t^2)}}_{(b)}+\underbrace{\sqrt{2R^2\alpha_T}}_{(c)}.
    \end{align*}
    where the first inequality follows from the convexity of $g_t$ and $g_t(\bm{x^\circ})\leq 0$.
    Here, term $(c)$ is equal to $\sqrt{2S_{T}}$, and term  $(b)$ is less than or equal to $\sqrt{2S_{T}}$. For term $(a)$, we have
    $$(a)^2= 2\sum_{t=1}^T S_{t}^\beta F_t\leq2 S_T^\beta\sum_{t=1}^T F_t\leq 2 S_T^\beta\sqrt{T\sum_{t=1}^T F_t^2}\leq 4\sqrt{T}S_T^{\beta+1/2},$$
    where the second inequality holds by the power mean inequality.
    Thus, 
    \begin{equation}\label{ineq:q-adpp}
        Q_{T+1}\leq (a)+(b)+(c)\leq 2\sqrt{2}S_T^{1/2}+2T^{1/4}S_T^{\beta/2+1/4}=O\left(S_T^{1/2}+T^{1/4}S_T^{\beta/2+1/4}\right).
    \end{equation}
    We also have that
    \begin{align*}
    \sum_{t=1}^T\frac{V_t F_t G_t}{2\alpha_t}&=\sum_{t=1}^T S_{t}^{\beta-1}RF_t G_t/2\\
    &\leq \sum_{t=1}^T S_{t}^{\beta-1}(F_t^2/4+R^2 G_t^2)/2\\
    &\leq \sum_{t=1}^T S_{t}^{\beta-1}a_t/2\leq \frac{S_T^\beta}{2\beta}\\
    &=O\left(S_T^\beta\right),
    \end{align*}
    where the last inequality follows from Lemma \ref{lem:sum1}.
    Lastly,
    \begin{align*}
    \sum_{t=1}^T \frac{G_t^2}{2\alpha_t}Q_t&\leq \sum_{t=1}^T \sqrt{2}\frac{G_t^2}{\alpha_t} S_{t}^{1/2}+\sum_{t=1}^T \frac{G_t^2}{\alpha_t}T^{1/4}S_{t}^{\beta/2+1/4}\\
    &\leq \sqrt{2}\sum_{t=1}^T \frac{a_t}{S_t^{1/2}} +T^{1/4}\sum_{t=1}^T \frac{a_t}{S_t^{3/4-\beta/2}}\\
    &\leq 2\sqrt{2S_T}+\frac{T^{1/4}}{\beta/2+1/4}S_T^{\beta/2+1/4}\\
    &=O\left(S_T^{1/2}+T^{1/4}S_T^{\beta/2+1/4}\right),
    \end{align*}
    where the first inequality follows from (\ref{ineq:q-adpp}) and the last inequality follows from Lemma \ref{lem:sum1}.
    Combining the results, we get
    $$\operatorname{Violation}(T)= O\left(S_T^{1/2}+T^{1/4}S_T^{\beta/2+1/4}\right),$$
    as required.
\end{proof}

\section{Proof of the Time-Varying Drift Lemma}\label{sec:appendix-drift}

In this section, we prove Lemma~\ref{lem:process} for the case of time-varying parameter $\theta_{t_0}(t)$.
We closely follow the proof of \citep[Lemma 5]{OCO-stochastic}. 
\begin{lemma}\label{lem:process'} Let $r = {\zeta}/(4t_{0}\delta_{\max}^{2})$ and $\rho = 1- {\zeta^{2}}/(8\delta_{\max}^{2}) = 1 - {rt_{0}\zeta}/{2}$. Then
\begin{align*}
\mathbb{E}\left[e^{rZ(t)}\right] \leq \frac{e^{rt_{0}\delta_{\max}}}{1-\rho} e^{r\theta(t)}
\end{align*}
for all $t\geq 0$.
\end{lemma}
\begin{proof}
Since $0<\zeta <\delta_{\max}$, we have $0<\rho<1 < e^{r\delta_{\max}}$. Define $\eta(t) = Z(t+t_{0}) - Z(t)$. Note that $|\eta(t)| \leq t_{0}\delta_{\max}$ for all $t\geq 0$ which implies that $|r\eta(t) | \leq  {\zeta}/(4\delta_{\max}) \leq 1$. Then, 
\begin{align}\label{eq:pf-lm-random-process-bound-eq1}
e^{rZ(t+t_{0})} = e^{rZ(t)} e^{r\eta(t)} 
\leq e^{rZ(t)} \left[1 +r\eta(t) + 2r^2 t_{0}^{2}\delta_{\max}^2\right] = e^{rZ(t)} \left[1 +r\eta(t) + \frac{1}{2}r t_{0}\zeta\right]
\end{align}
where the inequality follows from the fact that $e^x\leq 1+x+2x^2$ for $|x|\leq 1$, $|r\eta(t)|\leq 1$, and $|\eta(t)|\leq t_{0}\delta_{\max}$ while  the equality follows by substituting $r = {\zeta}/(4t_{0}\delta_{\max}^{2})$.

Next, we consider the cases $Z(t)\geq \theta(t) $ and $Z(t)<\theta(t)$, separately. First, consider the case $Z(t)\geq \theta(t)$. Taking the conditional expectation of each side  of \eqref{eq:pf-lm-random-process-bound-eq1} gives us the following.
\begin{align*}
\mathbb{E}\left[e^{rZ(t+t_{0})} \mid Z(t)\right] &\leq \mathbb{E} \left[e^{rZ(t)} (1 +r\eta(t) + \frac{1}{2}rt_{0}\zeta) \mid Z(t)\right] \\
&\leq e^{rZ(t)} \left[ 1 -rt_{0}\zeta + \frac{1}{2}rt_{0}\zeta\right]\\
&= e^{rZ(t)} \left[ 1- \frac{rt_{0}\zeta}{2}\right]\\
&= \rho e^{rZ(t)}
\end{align*}
where the inequality follows from the fact that $\mathbb{E}[Z(t+t_{0}) - Z(t)|\mathcal{F}(t)] \leq -t_{0}\zeta$ when $Z(t) \geq \theta(t)$ while the second equality follows from the fact that $\rho = 1 - {rt_{0}\zeta}/{2}$. Likewise, for the case $Z(t)<\theta(t)$, we deduce that
\begin{align*}
\mathbb{E}\left[e^{rZ(t+t_{0})} \mid Z(t) \right] = \mathbb{E}\left[e^{rZ(t)}e^{r\eta(t)} \mid Z(t)\right] = e^{rZ(t)} \mathbb{E}\left[e^{r\eta(t)}\mid Z(t)\right]\leq e^{rt_{0}\delta_{\max}}e^{rZ(t)}, 
\end{align*}
where the inequality follows from the fact that $\eta(t) \leq t_{0}\delta_{\max}$.

Putting the two cases together, we deduce that
\begin{align*}
&\mathbb{E}\left[e^{rZ(t+t_{0})}\right]\\
&=\mathbb{P}(Z(t) \geq \theta(t)) \mathbb{E}\left[e^{rZ(t+t_{0})} \mid Z(t) \geq \theta(t)\right] + \mathbb{P}(Z(t) < \theta(t)) \mathbb{E}\left[e^{rZ(t+t_{0})} \mid Z(t) < \theta(t)\right] \nonumber\\
&\leq \rho \mathbb{E}\left[e^{rZ(t)}\mid Z(t) \geq \theta(t)\right] \mathbb{P}(Z(t) \geq \theta(t)) + e^{rt_{0}\delta_{\max}} \mathbb{E}\left[e^{rZ(t)}\mid Z(t) < \theta(t)\right] \mathbb{P}(Z(t) < \theta(t)) \nonumber\\\
&= \rho \mathbb{E}\left[e^{rZ(t)}\right] + \left(e^{r t_{0}\delta_{\max}} - \rho\right) \mathbb{E}\left[e^{rZ(t)}\mid Z(t) < \theta(t)\right] \mathbb{P}(Z(t) < \theta(t)) \nonumber\\\
&\leq \rho \mathbb{E}\left[e^{rZ(t)}\right] + \left(e^{rt_{0}\delta_{\max}} - \rho\right) e^{r\theta(t)}\nonumber\\
&\leq \rho \mathbb{E}\left[e^{rZ(t)}\right] + e^{rt_{0}\delta_{\max}} e^{r\theta(t)}
\end{align*}
where the first inequality follows from the analysis of the two separate cases and the second inequality follows from the fact that $e^{rt_{0}\delta_{\max}}> \rho$.  

Then we argue by induction to prove the statement of this lemma. 
We first consider the base case $t\in\{0,1,\ldots,t_0\}$. Since $Z(t) \leq t \delta_{\max}$ for all $t \geq  0$, it follows that $$\mathbb{E}[e^{rZ(t)}] \leq e^{rt\delta_{\max}} \leq e^{r t_{0} \delta_{\max}} \leq \frac{e^{r t_{0}\delta_{\max}}}{1-\rho} e^{r\theta(t)}$$ for all  $t\in\{1,\ldots, t_{0}\}$, where the last inequality follows because ${e^{r\theta(t)}}/({1-\rho}) \geq 1$. Next we assume that the inequality holds for all $t\in\{0,1,\ldots,s\}$ with some $s\geq t_0$ and consider iteration $t=s+1$.Note that
\begin{align*}
\mathbb{E}\left[e^{rZ(s+1)}\right]&\leq \rho \mathbb{E}\left[e^{rZ(s+1-t_0)}\right]+e^{rt_0\delta_{\max}}e^{r\theta(s+1-t_0)}\\
    &\leq \rho\frac{e^{rt_0\delta_{\max}}}{1-\rho}e^{r\theta(s+1-t_0)}+e^{rt_0\delta_{\max}}e^{r\theta(s+1-t_0)}\\
    &\leq \frac{e^{rt_0\delta_{\max}}}{1-\rho}e^{r\theta(s+1-t_0)}\\
    &\leq\frac{e^{rt_0\delta_{\max}}}{1-\rho}e^{r\theta(s+1)}
\end{align*}
where the second inequality comes from the induction hypothesis by noting that $0\leq \tau+1-t_0\leq \tau$ while the last inequality follows from the fact that $\theta(t)$ is non-decreasing.
\end{proof}
Based on this lemma, we prove Lemma~\ref{lem:process}.
\begin{proof}[Proof of Lemma \ref{lem:process}]
Note that $e^{rx}$ is convex in $x$ when $r>0$. By Jensen's inequality,
\begin{align*}
e^{r\mathbb{E}[Z(t)]} \leq \mathbb{E}[e^{rZ(t)}] \leq\frac{e^{r(\theta(t)+t_{0}\delta_{\max})}}{1-\rho}
\end{align*}
where the inequality is implied by Lemma~\ref{lem:process'}. Taking logarithm on both sides and dividing by $r$ yields that
\begin{align*}
\mathbb{E}[Z(t)]&\leq \theta(t) + t_0\delta_{\max} + \frac{1}{r} \log\big[\frac{1}{1-\rho}\big] =\theta(t) + t_0\delta_{\max} + t_{0}\frac{4\delta_{\max}^{2}}{\zeta} \log\big[\frac{8\delta_{\max}^{2}}{\zeta^{2}}\big],
\end{align*}
where the equality holds because recalling that $r = \frac{\zeta}{4t_{0}\delta_{\max}^{2}}$ and $\rho = 1- \frac{\zeta^{2}}{8\delta_{\max}^{2}}$.
\end{proof}

\section{Properties of the MLMC Estimator}\label{sec:appendix-mlmc}

In this section, we prove Lemmas~\ref{lem:MLMC} and~\ref{lem:concentration}.

\begin{lemma}\citep[Lemma A.6]{Levy22}\label{lem:levy0}
    Let $h:\mathcal{X}\times\mathcal{S}\to \mathbb{R}^k$ for some $k\geq 1$. Suppose that there exists some constant $L>0$ such that $\|h(\bmx,\bmxi)\|\leq L$ for every $(\bmx,\bmxi)\in\mathcal{X}\times\mathcal{S}$, where the norm $\|\cdot\|$ satisfies $\|\cdot \| \leq \eta \|\cdot\|_2$ for some $\eta>0$. We denote by $$\bar{h}(\bmx):=\mathbb{E}_{\bmxi\sim \mu}[h(\bmx,\bmxi)],\quad  h_t^N(\bmx):=\frac{1}{N}\sum_{i=1}^N h(\bmx, \bmxi_t^{(i)}).$$ Suppose that $\bmx$ is $\mathcal{F}_{t-1}$ measurable and $N\leq A$ for some $A\in \mathbb{N}$. If $2\tmix\lceil2\log A\rceil\leq N$, then
    \begin{align*}
        \mathbb{E}_{t-1}\left[\lVert h_t^N(\bmx)-\bar{h}(\bmx)\rVert\right]&\leq12L\eta\sqrt{\frac{\tmix\lceil 2\log A\rceil}{N}}\left(1+\sqrt{\log(\tmix\lceil 2\log A\rceil N)}\right)\\
        &\quad +\frac{6L\eta\tmix\lceil2\log A\rceil}{N}+\frac{4L\eta}{N},\\
        \mathbb{E}_{t-1}\left[\lVert h_t^N(\bmx)-\bar{h}(\bmx)\rVert^2\right]&\leq\frac{576L^2\eta^2\tmix\lceil2\log A\rceil}{N}(1+\log(\tmix\lceil 2\log A\rceil N))\\
        &\quad +\frac{72L^2\eta^2\tmix^2\lceil 2\log A\rceil^2}{N^2}+\frac{8L^2\eta^2}{N}.
    \end{align*}
\end{lemma}
We point out that \citep[Lemma A.6]{Levy22} was originally stated for the $\ell_2$ norm, but the statement holds for any norm over a finite-dimensional vector space assuming that $\|\cdot \| \leq \eta \|\cdot\|_2$ for some fixed constant $\eta>0$ and $O$ hides the dependence on $\eta$. In the following lemma, we simplify the upper bounds of Lemma~\ref{lem:levy0}.
\begin{lemma}\label{lem:levy}
    Let $h:\mathcal{X}\times\mathcal{S}\to \mathbb{R}^k$ for some $k\geq 1$. Suppose that there exists some constant $L>0$ such that $\|h(\bmx,\bmxi)\|\leq L$ for every $(\bmx,\bmxi)\in\mathcal{X}\times\mathcal{S}$, where the norm $\|\cdot\|$ satisfies $\|\cdot \| \leq \eta \|\cdot\|_2$ for some $\eta>0$. We denote by $$\bar{h}(\bmx):=\mathbb{E}_{\bmxi\sim \mu}[h(\bmx,\bmxi)],\quad  h_t^N(\bmx):=\frac{1}{N}\sum_{i=1}^N h(\bmx, \bmxi_t^{(i)}).$$ Suppose that $\bmx$ is $\mathcal{F}_{t-1}$ measurable and $N\leq A$ for some $A\in \mathbb{N}$. Then
    \begin{align*}
        \mathbb{E}_{t-1}\left[\lVert h_t^N(\bmx)-\bar{h}(\bmx)\rVert\right]&\leq12L\max(1,\eta)\sqrt{\frac{\tmix\lceil 2\log A\rceil}{N}}\left(2+\sqrt{\log(\tmix\lceil 2\log A\rceil N)}\right),\\
        \mathbb{E}_{t-1}\left[\lVert h_t^N(\bmx)-\bar{h}(\bmx)\rVert^2\right]&\leq\frac{576L^2\max(1,\eta)^2\tmix\lceil2\log A\rceil}{N}(2+\log(\tmix\lceil 2\log A\rceil N)).\\
    \end{align*}
\end{lemma}

\begin{proof}
    First, we consider the case where $2\tmix\lceil2\log A\rceil> N$. Then by the triangle inequality, 
    \begin{align*}
        \mathbb{E}_{t-1}\left[\lVert h_t^N(\bmx)-\bar{h}(\bmx)\rVert\right]&\leq 2L< 2L\sqrt{\frac{2\tmix\lceil2 \log A\rceil}{N}}\\
        &\leq12L\max(1,\eta)\sqrt{\frac{\tmix\lceil 2\log A\rceil}{N}}\left(2+\sqrt{\log(\tmix\lceil 2\log A\rceil N)}\right),\\
        \mathbb{E}_{t-1}\left[\lVert h_t^N(\bmx)-\bar{h}(\bmx)\rVert^2\right]&\leq 4L^2 < 4L^2\frac{2\tmix\lceil 2\log A\rceil}{N}\\
        &\leq \frac{576L^2\max(1,\eta)^2\tmix\lceil2\log A\rceil}{N}(2+\log(\tmix\lceil 2\log A\rceil N)).
    \end{align*}
    Next, we consider the case where $2\tmix\lceil2\log A\rceil\leq N$. By Lemma \ref{lem:levy0},
    \begin{align*}
        \mathbb{E}_{t-1}\left[\lVert h_t^N(\bmx)-\bar{h}(\bmx)\rVert\right]&\leq12L\eta\sqrt{\frac{\tmix\lceil 2\log A\rceil}{N}}\left(1+\sqrt{\log(\tmix\lceil 2\log A\rceil N)}\right)\\
        &\quad +\frac{6L\eta\tmix\lceil2\log A\rceil}{N}+\frac{4L\eta}{N},\\
        \mathbb{E}_{t-1}\left[\lVert h_t^N(\bmx)-\bar{h}(\bmx)\rVert^2\right]&\leq\frac{576L^2\eta^2\tmix\lceil2\log A\rceil}{N}(1+\log(\tmix\lceil 2\log A\rceil N))\\
        &\quad +\frac{72L^2\eta^2\tmix^2\lceil 2\log A\rceil^2}{N^2}+\frac{8L^2\eta^2}{N}.
    \end{align*}
    Since ${\tmix\lceil2\log A\rceil}/{N}\leq \frac{1}{2}$, we deduce that
    \begin{align*}
        \frac{6L\eta\tmix\lceil2\log A\rceil}{N}+\frac{4L\eta}{N}\leq \frac{12L\eta\tmix\lceil2\log A\rceil}{N}\leq 12L\eta\sqrt{\frac{\tmix\lceil2\log A\rceil}{N}}
    \end{align*}
    and that
    \begin{align*}\frac{72L^2\eta^2\tmix^2\lceil 2\log A\rceil^2}{N^2}+\frac{8L^2\eta^2}{N}&\leq \frac{72L^2\eta^2\tmix\lceil 2\log A\rceil}{N}+\frac{8L^2\eta^2\tmix\lceil2\log A\rceil}{N}\\
    &\leq \frac{576L^2\eta^2\tmix\lceil2\log A\rceil}{N}.
    \end{align*}
As a result, we obtain
    \begin{align*}
        \mathbb{E}_{t-1}\left[\lVert h_t^N(\bmx)-\bar{h}(\bmx)\rVert\right]&\leq 12L\eta\sqrt{\frac{\tmix\lceil 2\log A\rceil}{N}}\left(2+\sqrt{\log(\tmix\lceil 2\log A\rceil N)}\right)\\
        &\leq 12L\max(1,\eta)\sqrt{\frac{\tmix\lceil 2\log A\rceil}{N}}\left(2+\sqrt{\log(\tmix\lceil 2\log A\rceil N)}\right),\\
        \mathbb{E}_{t-1}\left[\lVert h_t^N(\bmx)-\bar{h}(\bmx)\rVert^2\right]&\leq\frac{576L^2\eta^2\tmix\lceil2\log A\rceil}{N}(2+\log(\tmix\lceil 2\log A\rceil N))\\
        &\leq\frac{576L^2\max(1,\eta)^2\tmix\lceil2\log A\rceil}{N}(2+\log(\tmix\lceil 2\log A\rceil N)),
    \end{align*}
   as required.
\end{proof}

\begin{proof}[Proof of Lemma \ref{lem:MLMC}]
We first argue that
$$ \mathbb{E}_{t-1}[f_t(\bmx)]=\mathbb{E}_{t-1}\left[f_t^{2^{j_{\max}}}(\bmx)\right]$$
for any $\bmx$ and $t$. Note that
\begin{align*}
        \mathbb{E}_{t-1}\left[f_t\right]&=\mathbb{E}_{t-1}\left[f_t^1\right]+\sum_{j=1}^{j_{\max}}\mathbb{P}\left(J_t=j\right)2^j\mathbb{E}_{t-1}\left[f_t^{2^j}-f_t^{2^{j-1}}\right]=\mathbb{E}_{t-1}\left[f_t^{2^{j_{\max}}}\right]
    \end{align*}
    as $\mathbb{P}\left( J_t = j\right) = 1/2^j$
Similarly, we can show that
$$\mathbb{E}_{t-1}[g_t]=\mathbb{E}_{t-1}\left[g_t^{2^{j_{\max}}}\right],\quad \mathbb{E}_{t-1}[\grad f_t]=\mathbb{E}_{t-1}\left[\grad f_t^{2^{j_{\max}}}\right],\quad\mathbb{E}_{t-1}[\grad g_t]=\mathbb{E}_{t-1}\left[\grad g_t^{2^{j_{\max}}}\right]$$
holds for any $\bmx$ and $t$. For the second part, we have $$\mathbb{E}\left[\left|g_t\right|^2\right]\leq 2\mathbb{E}\left[\left| g_t - g_t^1\right|^2\right]+2H^2$$
    since $\lVert \bmx+\bmy\rVert^2\leq (\lVert\bmx\rVert+\lVert\bmy\rVert)^2\leq 2\lVert\bmx\rVert^2+2\lVert\bmy\rVert^2$ and $\left|g_t^1\right|\leq H$. Note that
    \begin{align*}
        \mathbb{E}\left[\left|g_t-g_t^1\right|^2\right]=\sum_{j=1}^{j_{\max}}\mathbb{P}(J_t=j) 2^{2j}\mathbb{E}\left[\left|g_t^{2^j}-g_t^{2^{j-1}}\right|^2\right]=\sum_{j=1}^{j_{\max}}2^{j}\mathbb{E}\left[\left|g_t^{2^j}-g_t^{2^{j-1}}\right|^2\right]
    \end{align*}
    because $\mathbb{P}\left( J_t = j\right) = 1/2^j$. Here we can bound the right-most side based on the following.
    \begin{align*}
        \mathbb{E}\left[\left|g_t^{2^j}-g_t^{2^{j-1}}\right|^2\right]\leq 2\mathbb{E}\left[\left|g_t^{2^j}-\bar{g}\right|^2\right]+2\mathbb{E}\left[\left|g_t^{2^{j-1}}-\bar{g}\right|^2\right]= \tilde{\mathcal{O}}\left(\frac{\tmix}{2^j}\right)
    \end{align*}
    where the last inequality follows from Lemma~\ref{lem:levy}.
    Then it follows that $$\mathbb{E}\left[\left|g_t-g_t^1\right|^2\right]\leq \tilde{\mathcal{O}}( j_{\max}\tmix)=\tilde{\mathcal{O}}(\tmix)$$
    where the last equality holds because $j_{\max}= O(\log T)$.
    For the last part, $$\mathbb{E}[N_t]=1+\sum_{j=1}^{j_{\max}}\mathbb{P}(J_t=j)(2^j-1)\leq 1+j_{\max} \leq 1+2\log_2 T,$$
    as required.
\end{proof}

\begin{proof}[Proof of Lemma~\ref{lem:concentration}]
    By Assumption \ref{ass:lip}, we can apply Lemma \ref{lem:levy} to $g_t, \grad g_t, \grad f_t$.  Assumptions \ref{ass:bounded}, \ref{ass:lip} implies that $|f_t(\bmx_t)|\leq J$ for some $J>0$. Thus, we can apply Lemma \ref{lem:levy} to $f_t$ as well. Let $L=\max(F,G,H,J), A=T^2$, and $\|\cdot\|_*\leq \eta\|\cdot\|_2$ for some $\eta>0$. By Lemma \ref{lem:levy}, the statement follows for
    \begin{align*}
        C(T)&=\sqrt{576L^2\max(1,\eta)^2\lceil 4\log T\rceil(2+\log(\tmix\lceil 4\log T\rceil 2^{j_{\max}}))}.
    \end{align*}
    Since $2^{j_{\max}}=\mathcal{O}(T^2)$, the order of $C(T)$ follows. 
\end{proof}

\vskip 0.2in
\bibliography{mybibfile}

\end{document}